\documentclass{amsart}

\usepackage[a4paper]{geometry}
\usepackage[shortlabels]{enumitem}

\usepackage{etoolbox}
\makeatletter
\patchcmd{\subsubsection}{\itshape}{\bfseries}{}{}
\makeatother
\usepackage[expansion=false]{microtype}

\usepackage{amssymb}
\usepackage{thmtools}\usepackage{mathabx}

\usepackage{multirow}
\usepackage{graphicx}
\usepackage{subcaption}

\usepackage{url}

\usepackage[noabbrev]{cleveref}

\newcommand{\N}{\mathbb{N}}
\newcommand{\Z}{\mathbb{Z}}
\newcommand{\R}{\mathbb{R}}
\newcommand{\C}{\mathbb{C}}
\renewcommand{\o}{\omega}
\newcommand{\Id}{\mathrm{Id}}
\newcommand{\Hess}{\mathrm{Hess}}
\newcommand{\ld}{\lambda}
\newcommand{\vp}{\varphi}
\newcommand{\g}{\gamma}
\renewcommand{\d}{\nabla}
\newcommand{\w}{\wedge}
\newcommand{\pp}{\partial}
\newcommand{\il}{\langle}
\newcommand{\ir}{\rangle}
\newcommand{\eps}{\varepsilon}
\newcommand{\Ind}{\mathrm{Ind}}
\newcommand{\im}{\text{im}}
\newcommand{\loc}{\mathrm{loc}}
\newcommand{\crit}{\mathrm{Crit}}
\newcommand{\gM}{\widehat{\gamma}}
\newcommand{\gm}{\widecheck{\gamma}}
\newcommand{\fT}{\mathcal{T}}
\newcommand{\ev}{\mathrm{ev}}
\newcommand{\pr}{\mathrm{pr}}

\newtheorem{thm}{Theorem}[section]
\Crefname{thm}{Theorem}{Theorems}
\newtheorem{lmm}[thm]{Lemma}
\Crefname{lmm}{Lemma}{Lemmas}
\newtheorem{prop}[thm]{Proposition}
\newtheorem{rmk}[thm]{Remark}
\newtheorem{cor}[thm]{Corollary}
\newtheorem{note}[thm]{Note}
\newtheorem{remark}[thm]{Remark}
\newtheorem{main}{Main Theorem}

\begin{document}
	\title{Mutations of Local Floer Chain Complexes across Hamiltonian Bifurcations}
	
	\author{Hong-Kwon Jo and Dongho Lee}
	
	\begin{abstract}
    We compute the chain-level change of local Floer complexes across the generic bifurcations of periodic orbits in four-dimensional Hamiltonian systems. Although local Floer homology is invariant across a bifurcation, the underlying chain complex mutates: periodic orbits and their Conley--Zehnder indices change discontinuously, and the Floer differential must change to compensate.
    Two tools make this mutation computable. First, we prove a Floer cascade counting theorem for Floer cylinders between multiply covered orbits of an autonomous system, which is not specific to the bifurcation analysis carried out here.
    Second, for each bifurcation type in Meyer's classification, and for two additional $\mathbb{Z}_2$-symmetric types, we construct a model Hamiltonian whose Floer cylinders are lifts of planar gradient trajectories, which we call the gradient revolutions, reducing the count to finite-dimensional Morse theory and cascade counting. We determine all seven mutations explicitly.
	\end{abstract}
	\maketitle
	\setcounter{tocdepth}{1}
	\tableofcontents
    
	\section{Introduction}\label{Section - Introduction}

In his 1970 paper \cite{Meyer_70}, Meyer classified the generic bifurcations of periodic points of one-parameter families of symplectomorphisms of a two-dimensional disk. Through the Poincar\'{e} return map, this classification translates into the local bifurcation theory of periodic orbits in four-dimensional Hamiltonian systems, as described for instance in \cite[Chapter 8]{Abraham_Marsden_78}. Subsequent works extended this picture to more structured settings, for example to systems with symmetries; see, for example, \cite{Golubitsky_Stewart_Marsden_87,Golubitsky_Stewart_Schaeffer_88}.

On the other hand, Floer theory, introduced by Floer \cite{Floer_89} and now a central tool in symplectic topology and geometry, also studies Hamiltonian periodic orbits, but from a homological viewpoint. In Floer theory, periodic orbits appear as generators of a chain complex, graded by the Conley--Zehnder index, and the differential counts Floer cylinders. Thus Floer homology provides topological constraints on the creation, cancellation, and continuation of periodic orbits.

Although local Floer homology is invariant under suitable perturbations of the Hamiltonian, the chain complex computing it is not fixed across a bifurcation. Across the bifurcation, periodic orbits may be created or annihilated, and their Conley--Zehnder indices may change. Thus the generators and gradings of the local Floer chain complex can undergo a discontinuous change. The invariance of local Floer homology then forces the Floer differential to change accordingly. In this sense, a bifurcation produces a chain-level mutation whose homology remains unchanged.

Recently, the Conley--Zehnder index has become a key tool for detecting bifurcation phenomena, especially in problems arising from celestial mechanics; see, for example, \cite{Aydin_Batkhin_25,Joung_Koh_vanKoert_26}. Along a non-degenerate family of periodic orbits, the Conley--Zehnder index is locally constant, and hence an index jump can occur only when a relevant iterate becomes degenerate, precisely the regime in which bifurcations arise. Moreover, the invariance of local Floer homology can be used to guarantee the existence of new periodic orbits when the index changes. In this direction, \cite{Aydin_Frauenfelder_vanKoert_Koh_Moreno_24} introduced the SFT-Euler characteristic as a symplectic invariant of bifurcations, defined as the alternating count of nearby orbits weighted by the parity of their Conley--Zehnder indices.

While such invariants extract robust numerical information from the nearby orbits, a natural next step is to ask how the underlying chain complex mutates across the bifurcation. More precisely, the local Floer chain complex is a discrete object which may change as the bifurcation parameter crosses a degeneracy, even though its homology is invariant. We call this phenomenon a \textbf{mutation} of the local Floer chain complex.

Making this chain-level mutation explicit is the central aim of this paper. To this end, we construct local models for computing the local Floer chain complex of each bifurcation type. More precisely, we specify the Hamiltonian, the almost complex structure, and the Morse functions on the periodic orbits, which are required to define the relevant Floer complex.

Our first main result is a Floer cascade counting theorem for Morse--Bott Floer homology in the presence of multiple covers. This formula determines the algebraic contribution of a Floer cylinder connecting multiply covered periodic orbits of an autonomous Hamiltonian system. It is not specific to the bifurcation models considered later. Rather, it applies to a pair of simple periodic orbits whose suitable covers are one-periodic and whose relative Conley--Zehnder index is one, under the standard compactness assumptions excluding bubbling and breaking.

\begin{main}[\Cref{Thm - Cylinder Counting}]
Let $\g_\pm$ be two simple non-constant $1/k_\pm$-periodic orbits, so that $\g_\pm^{k_\pm}$ are 1-periodic orbits. Also, let $J$ be a generic almost complex structure.
        Moreover, assume that
        \begin{equation*}
            \mu_{CZ}^\fT(\gamma^{k_+}_+)-\mu_{CZ}^\fT(\gamma^{k_-}_-)=1.
        \end{equation*}
Then, the quotient $\mathcal{M}_0(\g^{k_-}_-,\g^{k_+}_+)/S^1$ corresponding to $\fT$ is a zero-dimensional $\Z_{(k_+,k_-)}$-orbifold, and we have
        \begin{align*}
            \#_2\mathcal{M}(\gM_-^{k_-},\gM_+^{k_+})&\equiv k_-\cdot\#(\mathcal{M}_0(\g^{k_-}_-,\g^{k_+}_+)/S^1)\pmod 2\\
            \#_2\mathcal{M}(\gm_-^{k_-},\gm_+^{k_+})&\equiv k_+\cdot\#(\mathcal{M}_0(\g^{k_-}_-,\g^{k_+}_+)/S^1)\pmod 2,
        \end{align*}
        where $(k_+,k_-)$ is the greatest common divisor of $k_+$ and $k_-$, and the right hand side is an orbifold count.
\end{main}

This theorem provides the bridge between the local dynamics of the return map and the Floer differential. We next clarify the precise meaning of the local models considered in this paper. Let $\bar H_\varepsilon$ be a generic one-parameter family of Hamiltonians undergoing one of the bifurcations considered below. In \Cref{Subsection - Model}, we associate to this bifurcation a model Hamiltonian $H_\varepsilon$, which retains the data relevant for computing the local Floer chain complex. The relation between this model Hamiltonian and a generic Hamiltonian family is discussed in detail at the end of \Cref{Subsection - Model}.

In these bifurcation models, the computation of the local Floer chain complex is reduced to three steps. First, the gradient-flow structure of the generating Hamiltonian defined on a Poincar\'{e} section determines the Floer cylinders between the corresponding autonomous orbits. Second, for autonomous orbits of relative Conley--Zehnder index one, the cascade-counting theorem gives the algebraic contribution of these cylinders directly. Third, for autonomous orbits of relative Conley--Zehnder index two, the contributions of two-step Floer cascades are obtained by combining these cylinders with Morse--Bott gradient trajectories through a combinatorial argument. In this way, the chain-level mutation of each model Hamiltonian becomes explicitly computable.

The second main result is the explicit computation of the resulting local Floer chain complexes for the model Hamiltonians associated with Hamiltonian bifurcations. We carry out the three-step procedure above for the five generic Hamiltonian bifurcations in Meyer's classification: birth-death, period doubling, materialization, phantom kiss, and emission. We also treat two additional bifurcations arising in $\mathbb Z_2$-symmetric involutive systems: pitchfork and double emission.

For clarity, let $H_\eps$ be a one-parameter family of Hamiltonians, and let $\g_0$ denote the \emph{core orbit} at $\eps=0$ whose Floquet multiplier is a primitive $k$-th root of
unity, so that the $k$-fold cover $\g_0^k$ is degenerate and undergoes one of the bifurcations considered in this paper.
Whenever the core orbit admits a continuation along the bifurcation parameter $\eps$, we denote it by $\g_\eps$. If $k>1$, such a continuation exists uniquely because the simple orbit $\g_0$ is non-degenerate. The pitchfork model also has a distinguished continuation of the core orbit. For a sufficiently small $\delta>0$, we write $\g_\pm := \g_{\pm \delta}$ whenever the corresponding continuations are defined.

Except for the birth-death and phantom-kiss, exactly one side of $\eps=0$ contains no nearby periodic orbit other than the continuation of the core orbit. We call this the \emph{unbranched side} of the bifurcation. In the normal forms used in \Cref{Section - Classification of bifurcations,Section - Symmetric}, the unbranched side is chosen to be $\eps<0$.

\begin{main}
For the model Hamiltonians associated with the generic Hamiltonian
bifurcations and the $\Z_2$-symmetric involutive bifurcations considered in
this paper, the mutations of the local Floer chain complexes are explicitly
computed with $\Z_2$-coefficients. The resulting complexes are described in
\Cref{Theorem - FCBD,Theorem - FCPD,Theorem - FCPK,Theorem - FCEM,Theorem - FCPF,Theorem - FCDE}.
Consequently, the local Floer homology at the bifurcating orbit is given as
follows.
\begin{enumerate}
    \item
    For a birth-death $(k=1)$ or a $3$-phantom kiss $(k=3)$,
    \[
        HF_*^\loc(\g_0^k,H_0)=0.
    \]
    \item
    For a period-doubling $(k=2)$, materialization $(k=2)$, $4$-phantom kiss $(k=4)$,
    emission $(k\geq4)$, pitchfork $(k=1$, $\Z_2$-symmetric$)$, or double emission $(k\geq3$ odd, $\Z_2$-symmetric$)$
    \[
        HF_*^\loc(\g_0^k,H_0)
        =
        \begin{cases}
            \Z_2, & *=N,N+1,\\
            0,    & \text{otherwise},
        \end{cases}
    \]
    where
    \[
        N=
        \begin{cases}
        \displaystyle
        \frac{1}{2}
        \bigl(
            \mu_{CZ}(\g_-^4)+\mu_{CZ}(\g_+^4)
        \bigr),
        &
        \text{for the $4$-phantom kiss},\\
                \mu_{CZ}(\g_-^k),
        &
        \text{otherwise}.
        \end{cases}
    \]
    All Conley--Zehnder indices are computed with respect to the common framing near the bifurcating orbit.
\end{enumerate}
\end{main}

The paper is organized as follows. In \Cref{Section - Prelim}, we review the essential concepts of Floquet multipliers, the Conley--Zehnder index, and Birkhoff normal forms. \Cref{Section - Classification of bifurcations} revisits Meyer's classification, carefully extracting the Morse indices relevant to our homological setup. \Cref{Section - Symmetric} extends this analysis to involutive systems. In \Cref{Section - LFH}, we define the local Floer chain complex and Morse--Bott cascades, and prove the cascade counting theorem for multiple covers. \Cref{Section - cylinder} is devoted to the construction of the model Hamiltonians and to the proof of regularity and uniqueness for the gradient revolutions. Finally, in \Cref{Section - LFCC}, we present the main computations, explicitly describing the mutations of the local Floer chain complexes for the five generic bifurcations and for the two $\Z_2$-symmetric involutive bifurcations.

\subsubsection*{Acknowledgements} Hong-kwon Jo was supported by National Research Foundation of Korea grants NRF-2020R1A5A1016126 and RS-2023-00211186.
Dongho Lee was supported by the National Research Foundation of Korea (NRF) grant funded by the Korea government (MSIT) (No.2020R1A5A1016126).
The authors are grateful to Jungsoo Kang and Otto van Koert (Seoul National University) for their valuable advice and encouragement.
They would also like to thank Chankyu Joung (Seoul National University) and Sungho Kim (University of Science and Technology of China) for insightful discussions.

\section{Periodic Orbits of Hamiltonian Systems}\label{Section - Prelim}

\subsection{Floquet Multipliers and Bifurcations}\label{Section - Floquet and Bifurcation}
Let $(W^{2n},\o)$ be a symplectic manifold and $H$ a Hamiltonian, which is a smooth function $H:W\to\R$.
The \textbf{Hamiltonian vector field} $X_H$ is defined by $i_{X_H}\o = -dH$.
Since $dH(X_H)=-\o(X_H,X_H)=0$, the function $H$ is constant along its flow, so each energy hypersurface $H^{-1}(c)$ is invariant under $X_H$, and for a regular value $c$ it is a smooth manifold.
We call $c$ an energy level of $H$.

An orbit $\g(t)=Fl^{X_H}_t(p)$ is a \textbf{periodic orbit} if $\g(t+T)=\g(t)$ for some $T>0$.
Let $\g\subset H^{-1}(c)$ be a periodic orbit on a regular energy level.
A \textbf{Poincar\'{e} section} at $p\in\g$ is a codimension $2$ symplectic submanifold $\Sigma^{2n-2}\subset H^{-1}(c)$ transverse to $X_H$.
For $q$ in a small neighborhood $N\subset\Sigma$ of $p$, the \textbf{return time} $\tau(q)=\min\{t>0:Fl^{X_H}_t(q)\in\Sigma\}$ is smooth, and the \textbf{return map}
\[
\Psi:N\to\Sigma,\qquad q\mapsto Fl^{X_H}_{\tau(q)}(q)
\]
is a symplectomorphism \cite{Abraham_Marsden_78, Meyer_Hall_Offin_13}.
Although $\Psi$ need not be defined on all of $\Sigma$, it captures the local dynamics of $X_H$ near $\g$, and we write $\Psi:\Sigma\to\Sigma$ by abuse of notation.

We now restrict to a $4$-dimensional system, so $\Sigma\simeq D^2$ is $2$-dimensional, meeting $\g$ at the origin.
The return map $\Psi$ fixes $0$, and $d\Psi(0)\in Sp(2)$.
Its two eigenvalues, the \textbf{Floquet multipliers} of $\g$, fall into three cases:
\begin{itemize}
    \item $\ld\in S^1$ ($\ld^{-1}=\bar\ld$): $\g$ is \textbf{elliptic};
    \item $\ld\in\R\setminus\{0\}$ ($\ld=\bar\ld$): $\g$ is \textbf{hyperbolic};
    \item $\ld=\pm1$, a pair $\{1,1\}$ or $\{-1,-1\}$ with multiplicity.
\end{itemize}
Ellipticity and hyperbolicity are the \textbf{stability type} of $\g$.
Since the multipliers determine each other, we usually record only one.
If $\ld$ is a multiplier of $\g$, then $\ld^k$ is a multiplier of the $k$-th cover $\g^k$, so all covers share the stability type of $\g$.
We call $\g$ \textbf{non-degenerate} if $\ld\neq1$ and \textbf{degenerate} if $\ld=1$; note that if $\ld$ is a primitive $k$-th root of unity, then $\g$ is non-degenerate but $\g^k$ is degenerate.

Now consider a $1$-parameter family $H_\eps$ for small $\eps$ and a periodic orbit $\g_0$ of $H_0$.
\begin{lmm}[Orbit cylinder]\label{Lemma - Orbit Cylinder}
If $\g_0$ is non-degenerate, there is a smooth family of periodic orbits $\g_\eps$ of $H_\eps$ through $\g_0$, forming
\[
\Gamma = \{(\g_\eps,\eps)\,:\,\g_\eps\text{ is a periodic orbit of }H_\eps\}.
\]
\end{lmm}
\begin{proof}
Since $\ld\neq1$, the map $d\Psi_\eps-\Id$ is invertible, and the claim follows from the implicit function theorem \cite{Abraham_Marsden_78}.
\end{proof}
We call $\Gamma$ an \textbf{orbit cylinder} centered at $\g_0$, or a \textbf{continuation} of $\g_0$ along the bifurcation parameter $\eps$.
Along $\Gamma$ the multiplier $\ld_\eps$ varies continuously.
When $\ld_\eps=1$ the orbit degenerates and the cylinder may break: the family $\g_\eps$ may vanish or new orbits may be born.
This is a \textbf{bifurcation}, the main subject of this paper.

\subsection{Conley--Zehnder Index}\label{Subsection - CZ index}
The Conley--Zehnder index assigns to a periodic orbit an integer that, in Floer theory, determines the grading of its generator.
We briefly recall the version we use, following \cite{Robbin_Salamon_93}; this index is most naturally phrased for a \emph{framed} orbit, and it is the framing-dependence that we exploit throughout the paper.

Let $\g$ be a $\tau$-periodic orbit of a Hamiltonian $H$ on $(W^{2n+2},\o)$, lying on a regular level $H^{-1}(c)$.
The kernel $\ker(dH)$ carries the symplectic subbundle
\[
\xi^{2n}=\{\, v\in\ker(dH) : \o(v,X_H)=0 \,\}
\]
transverse to $X_H$ along $\g$, on which the linearized flow $d Fl^{X_H}_t|_\xi$ acts symplectically.
A framing (or trivialization) $\fT$ of $\xi$ along $\g$ presents this linearized flow as a path
\[
\psi^{\fT}:[0,\tau]\to Sp(2n),\qquad \psi^T(0)=\Id.
\]

For a path $\psi:[0,\tau]\to Sp(2n)$ with $\psi(0)=\Id$ whose endpoint satisfies $\det(\psi(\tau)-\Id)\neq0$, the \textbf{Robbin--Salamon index} $\mu_{RS}(\psi)\in\Z$ is the symplectic path invariant of \cite{Robbin_Salamon_93}.
It is a homotopy invariant of $\psi$ relative to endpoints, additive under concatenation, and normalized so that it computes the usual Conley--Zehnder index in the non-degenerate case.
We define the \textbf{Conley--Zehnder index} of $\g$ with respect to the framing $\fT$ by
\[
\mu_{CZ}^\fT(\g) = \mu_{RS}(\psi^\fT).
\]

The value $\mu_{CZ}^\fT(\g)$ depends on the homotopy class of the framing $\fT$.
Replacing $\fT$ by a framing that differs by a loop of winding number $N\in\pr_1(Sp(2n))\cong\Z$ shifts the index by a fixed amount,
\[
\mu_{CZ}^{\fT'}(\g) = \mu_{CZ}^\fT(\g) + 2N.
\]
Consequently no absolute integer is canonically attached to a single orbit without extra data.
What is canonical is the \textbf{relative Conley--Zehnder index} of two orbits $\g_0,\g_1$ admitting a common framing $\fT$,
\[
\mu_{CZ}(\g_1,\g_0) := \mu_{CZ}^\fT(\g_1) - \mu_{CZ}^\fT(\g_0),
\]
which is independent of the choice of $\fT$ since the ambiguity $2N$ cancels.
For instance, we can consider two orbits in a single tubular neighborhood, or a family $\g_\eps$ along an orbit cylinder.

\begin{remark}\rm
It is this relative index that we track across a bifurcation. Indeed, in the study of bifurcations, the relevant information is not the absolute value of the Conley--Zehnder index of a single orbit, but how the index changes as the orbit passes through a degeneracy. Along a non-degenerate orbit cylinder, $\mu_{CZ}^{\fT}(\g_\eps)$ is locally constant for any fixed choice of framing $\fT$. Hence an index jump can occur only when the relevant orbit becomes degenerate, precisely the situation in which a bifurcation may occur. Thus in this case, the information carried by the Conley--Zehnder index is naturally relative.

There is also a topological reason for using relative indices in our local models. The symplectic manifolds used in \Cref{Section - cylinder,Section - LFCC} are topologically  $S^1\times \mathbb R^3$, and the periodic orbits under consideration are homotopically non-trivial. In particular, they do not admit capping disks, which means there is no canonical capping trivialization from which one could define an absolute Conley--Zehnder index. For this reason, all index computations below are formulated in terms of relative Conley--Zehnder indices.
\end{remark}

    \subsection{Birkhoff Normal Form}\label{Subsection - BNF}

Let $(W^4,\o)$ be a symplectic manifold, $H:W\to\R$ an autonomous Hamiltonian, and $\g$ a periodic orbit.
We work on the $2$-dimensional Poincar\'e section $\Sigma\simeq D^2$ at the origin, with return map $\Psi:\Sigma\to\Sigma$.

\begin{lmm}\label{Lemma - General Generating Hamiltonian}
Write $\Psi(p)=Lp+\psi(p)$ with $L\in Sp(2)$ and $\psi(0)=d\psi(0)=0$.
If $L$ has a logarithm, then $\Psi$ is the time-$1$ map of a $1$-periodic Hamiltonian.
\end{lmm}
\begin{proof}
See \cite[Theorem 6.2.1, \S10.6]{Meyer_Hall_Offin_13}.
\end{proof}

Let $G:\Sigma\to\R$ be the time-dependent $1$-periodic Hamiltonian generating $\Psi$.
By the Floquet--Lyapunov theorem \cite[Theorem 3.4.2]{Meyer_Hall_Offin_13}, a symplectic coordinate change makes the quadratic part of $G$ time-independent.
We write it as $G_0(p)=\frac{1}{2} x^TSx$ with $S$ symmetric, so that $L=\exp(JS)$.
The Jordan--Chevalley decomposition
\[
JS = A + N,\qquad [A,N]=0,
\]
splits $JS$ into its \textbf{semisimple part} $A$ (diagonalizable) and nilpotent $N$; correspondingly
$$
L = \exp(A)\exp(N) = RU,
$$
where $R=\exp(A)$ is a rotation or hyperbolic scaling and $U=\exp(N)$ is a shear.

\begin{thm}[Birkhoff Normal Form]\label{Theorem - BNF}
Let $G$ be a $1$-periodic Hamiltonian with quadratic part $\frac{1}{2} p^TSp$, and let $A$ be the semisimple part of $JS$.
Then there is a formal symplectic change of variables $p=p(t,q)=q+\cdots$ such that, writing $G(t,q)=\sum_{m\geq0}G_m(t,q)$ with $G_m$ homogeneous of degree $(m+2)$ and $G_0(t,q)=\tfrac12 q^TSq$,
\[
\pp_t G_m + \{G_m,G_0\}=0,
\quad\text{equivalently}\quad
G_m(t,q) = G_m\left(0,\exp(-At)q\right).
\]
\end{thm}
\begin{proof}
See \cite{Meyer_Hall_Offin_13}.
\end{proof}

Thus the normal form rewrites $G$ as a formal series equivariant under the linear flow $\exp(At)$ of its semisimple part.
The series $G=\sum G_m$ is formal and generically divergent; convergence is addressed later.

\subsubsection*{Degenerate case}
When $L$ has $1$ as an eigenvalue, $L$ is a shear with $1$'s on the diagonal, so $A=0$ and the equivariance condition reduces to $G_m(t,p)=G_m(0,p)$.
The series is then autonomous, giving the following.
\begin{cor}\label{Corollary - Shear case}
If $\g$ has Floquet multiplier $1$, there is a formal series $G=\sum_{m\geq 0} G_m$ of homogeneous polynomials of degree $(m+2)$, an autonomous Hamiltonian generating $\Psi$.
\end{cor}

\subsubsection*{Degenerate multiple cover case}
Now suppose $L$ has eigenvalues $\ld=\exp(\pm2\pi i l/k)$ with $(l,k)=1$ and $k\geq2$: for $k=2$, $L$ is a shear with $-1$'s on the diagonal; for $k\geq3$, $L$ is the rotation $R$ by $2\pi l/k$.
Identifying $\Sigma$ with a subset of $\C$, the semisimple flow $\exp(At)$ acts as rotation and induces the natural $\Z_k$-action on $\C$.
For a monomial $e^{ict}z^a\bar z^b$, equivariance under $\exp(At)$ forces
\[
z^a\bar z^b = \exp\!\big(i(2\pi(a-b)l/k+c)t\big)\,z^a\bar z^b,
\]
i.e.\ $a-b\equiv0\pmod k$ and $c=-2\pi(a-b)l/k$.
Hence the invariants are generated by $|z|^2$, $e^{-2\pi l i t}z^k$, and $e^{2\pi l i t}\bar z^k$.
In polar coordinates $z=\rho e^{i\vp}$ these become $\rho^2$, $\rho^k\cos(k\vp-2\pi lt)$, $\rho^k\sin(k\vp-2\pi lt)$.
\begin{cor}\label{Corollary - ROU case}
If $\g$ has Floquet multiplier $\exp(2\pi i l/k)$ with $(l,k)=1$, $k\geq2$, there is a formal series $G=\sum G_m$ of $\Z_k$-symmetric homogeneous polynomials of degree $(m+2)$ with time-dependent coefficients generating $\Psi$, with
\[
G_m\in\R\left[\rho^2,\rho^k\cos(k\vp-2\pi lt),\rho^k\sin(k\vp-2\pi lt)\right].
\]
Note that $G_0$, the second order term, can be taken autonomous.
\end{cor}

From now on, we will use a finite-degree truncation of the generating Hamiltonian. The power series expansion of Birkhoff normal form obtained above converges only if the system is integrable, and in generic systems, infinite-order flat terms exist, which may lack $\Z_k$-symmetry. However, the bifurcations we will consider in \Cref{Section - Classification of bifurcations,Section - Symmetric} fundamentally depend only on the critical points of $G_\eps$. Furthermore, the gradient trajectories between these critical points can also be determined by the leading terms, which will be utilized to construct Floer cylinders in \Cref{Subsection - Gradient Revolution}. Therefore, for the purpose of analyzing the bifurcations and the associated Floer-theoretic information, it is sufficient to consider only the finite-degree truncation.

\subsection{The Unwrapped Generating Hamiltonian}\label{Subsection - Unwrapped generating Hamiltonian}

We now set $\ld=\exp(2\pi il/k)$ and seek a generating Hamiltonian for $\Psi^k$, the $k$-th iterate of the return map.
Naively one takes $kG$, which by \Cref{Corollary - ROU case} has the form
$$
kG(\rho,\vp,t) = \pi l\rho^2 + \sum_{m\geq1}G_m,
\quad
G_m\in\R\left[\rho^2,\rho^k\cos(k\vp-2\pi lt),\rho^k\sin(k\vp-2\pi lt)\right].
$$
The time dependence is entirely carried by the phase $k\vp-2\pi lt$, generated by the semisimple part $\pi l\rho^2$.
This suggests passing to the rotating frame in which that phase is constant; the obstruction is that the frame is not well-defined on $\Sigma$, since tracing $\g$ once rotates it by $2\pi l/k$ and a single loop fails to close up.

We therefore pass to the $k$-fold cover on which the frame closes.
Let $W$ be a tubular neighborhood of $\g$, and for the coprime pair $(k,l)$, let
$$
\mathrm{cov}_{k,l}:\tilde{W}_{k,l}\to W
$$
be the $k$-fold cover associated to the rotation by $2\pi l/k$ on $\Sigma$, on which $k$ full loops along $\g$ compose to a $2\pi l$ rotation, which is just the identity.

On $\tilde{W}_{k,l}$, we introduce the rotating frame generated by the semisimple rotation.
The semisimple part is $A=2\pi lJ$, whose flow $e^{At}$ is the rotation by angle $2\pi lt$.
Set
$$
w = e^{-At}z,
\quad\text{equivalently,}\quad
\vp' = \vp - 2\pi lt .
$$
This change of coordinates is well-defined on $\tilde{W}_{k,l}$, where $e^{-At}$ closes up after $k$ periods.
In the frame $(\rho,\vp')$, the time-dependent phase becomes
$$
k\vp - 2\pi lt = k\vp',
$$
so every generator in \Cref{Corollary - ROU case} loses its explicit $t$-dependence: $\rho^2$, $\rho^k\cos k\vp'$, $\rho^k\sin k\vp'$ are all autonomous.
Hence $kG$ is genuinely autonomous in the rotating frame, and the semisimple part $\pi l\rho^2$ has been absorbed into the frame rather than discarded.

Being built from the semisimple flow, the rotating frame is equivariant under the deck group $\Z_k$, so the resulting autonomous generator is $\Z_k$-invariant and descends along $\mathrm{cov}_{k,l}$ to a well-defined Hamiltonian on $W$.
We call this the \textbf{unwrapped generating Hamiltonian}.
Its critical points other than the origin correspond precisely to the $k$-periodic points of $\Psi$.

\begin{cor}[Unwrapped Generating Hamiltonian]\label{Corollary - Unwrapped Generating Hamiltonian}
Let $\g$ have Floquet multiplier $\exp(2\pi i l/k)$ with $(l,k)=1$.
Then there is a formal series $G=\sum_{m\geq0}G_m$ of autonomous $\Z_k$-symmetric homogeneous polynomials of degree $(m+2)$, obtained by lifting to $\tilde{W}_{k,l}$ and passing to the rotating frame of the semisimple flow, which generates $\Psi^k$.
In particular, $G_0$ generates shear for $k=1,2$, $G_0=0$ for $k\geq3$, and
\[
G\in
\begin{cases}
\R[x,y] & k=1,\\
\R[x^2,xy,y^2] & k=2,\\
\R[\rho^2,\rho^k\cos k\vp,\rho^k\sin k\vp] & k\geq3.
\end{cases}
\]
\end{cor}        
\section{Classification of the Generic Bifurcations}\label{Section - Classification of bifurcations}

In the 1970s, Meyer classified the generic bifurcations of the periodic points of 2-dimensional symplectomorphisms in \cite{Meyer_70}, which was followed by the classification of the bifurcations of the periodic orbits of Hamiltonian systems in dimension 4, as summarized in \cite{Abraham_Marsden_78} and \cite{Meyer_Hall_Offin_13}.
In this section, we introduce the results of this classification, slightly refined to contain more information related to the Floer chain complex.

In particular, we take the finite degree truncation of the unwrapped generating Hamiltonian constructed in \Cref{Subsection - Unwrapped generating Hamiltonian}. This forces the system to be integrable, but it's enough for the analysis of the bifurcation, as mentioned at the end of \Cref{Subsection - BNF}.

\subsection{Birth-death}\label{Subsection - Birth-death}
We first describe a fixed point $p$ of $\Psi_\eps$ with characteristic multiplier $1$.
Generically, the linear part of $\Psi_0$ is a shear, and we can find a formal autonomous generating Hamiltonian $G_0$ of $\Psi_0$.
The leading terms of the Hamiltonian can be written in the form
\[
G_0(x,y) =\frac{\alpha}{2}x^2 + \frac{\beta}{3}y^3+\ldots.
\]
Note that we fix a shear in $x$-direction.
In this case, the bifurcation on $G_0$ can be applied to any order, so the leading order terms are given by
\[
G_\eps(x,y) = \frac{\alpha}{2}x^2 +\frac{\beta}{3}y^3 +\eps(ax + by) + \ldots.
\]
If $\alpha, \beta$ and $b$ are nonzero, we call the origin an \textbf{extremal fixed point}.

The fixed points of $\Psi_\eps$ correspond to the critical points of $G_\eps$, and we need to solve the equation with initial terms
\[
\begin{aligned}
    \pp_x G_\eps&=\alpha x + \eps a=0, \\
    \pp_y G_\eps&= \beta y^2 + \eps b=0.
\end{aligned}
\]
The initial terms of the solutions, which extend by the implicit function theorem, are
\[
\begin{pmatrix}
    x_\eps\\
    y_\eps
\end{pmatrix}
=
\begin{pmatrix}
-\eps a/\alpha\\
\pm\sqrt{-\eps b/\beta}
\end{pmatrix}.
\]
This implies that there exist $0$ or $2$ solutions, say $p_0$ and $p_1$, depending on the sign of $-\eps b/\beta$.
By taking an appropriate direction for the bifurcation parameter, we may assume that $-\eps b/\beta>0$ for $\eps>0$.
We compute the Hessian of $G_\eps$ at these critical points to analyze the Morse index, which gives
\[
\Hess G_\eps (p_{0,1}) = \begin{pmatrix}
    \alpha&0\\
    0& 2\beta y
\end{pmatrix} = \begin{pmatrix}
    \alpha &0\\
    0&\pm 2\beta \sqrt{-\eps b/\beta}
\end{pmatrix}.
\]
This implies that one of the two critical points, say $p_H$, is a saddle point, while the other, say $p_E$, is a maximum or a minimum.
To summarize, we have the following theorem.

\begin{thm}[\cite{Meyer_70}, Extremal Fixed Point]\label{Theorem - FPbd}
    Let $p$ be an extremal fixed point of $\Psi_\eps$. By an appropriate choice of the bifurcation parameter, there exists a family of elliptic fixed points $p_E$ and a family of hyperbolic fixed points $p_H$ parametrized by $\eps>0$.
    As $\eps\to0$, $p_E$ and $p_H$ converge to $p$.
    In particular, the difference of the Morse indices of the two new fixed points with respect to the generating Hamiltonian is exactly $1$.
\end{thm}

We can translate this result into the language of Hamiltonian periodic orbits as follows.

\begin{thm}[Birth-death]\label{Theorem - PObd}
    Let $H_\eps$ be a 1-parameter family of Hamiltonians and $\g_0$ be a periodic orbit of $H_0$ with Floquet multiplier $\ld=1$.
    Assume that a point $p$ in $\g_0$ is an extremal fixed point of the return map. 
    
    Then, by an appropriate choice of the bifurcation parameter, there exists a family of elliptic periodic orbits $\g_E=\g_E(\eps)$ and a family of hyperbolic periodic orbits $\g_H=\g_H(\eps)$ parametrized by $\eps>0$.
    Moreover, $\g_E$ and $\g_H$ both converge to $\g_0$ as $\eps\to0$.
\end{thm}

We call this phenomenon a \textbf{birth-death bifurcation}, while \cite{Abraham_Marsden_78} referred to this as \emph{creation and annihilation}.
If $p_E$ is the maximum (resp. the minimum) of $G_\eps$, we call this a \textbf{source} (resp. \textbf{sink}), which focuses on the role of the elliptic fixed point.
\Cref{Figure - birth-death} illustrates the two situations on the Poincar\'{e} section.

\begin{rmk}\rm
    In Morse theory, we use the negative gradient flow of the function, so the maximum point is a source. It's the same for the Floer cylinders.
    However, Morse indices have a negative correlation with the Conley--Zehnder indices, and the Floer differential has the opposite direction with the Morse differential. For this reason, we will draw the positive gradient flows in the pictures, which indicates the contribution to the Floer differential.
\end{rmk}

\begin{note}\label{Note - Colors}\rm
    In \Cref{Figure - birth-death}, we denoted a minimum, saddle point, maximum by green, red, blue dots respectively. This convention is the same for \Cref{Figure - Period doubling,Figure - Materialization,Figure - Phantom kiss,Figure - Emission}.
    Also, in \Cref{Section - LFCC}, we indicated the periodic orbits in the same color with the corresponding critical points in Morse--Bott spectral sequences, in \Cref{Figure - SSBD,Figure - SSPD,Figure - SSPK,Figure - SSEM,Figure - SSPF,Figure - SSDE}.
\end{note}

\begin{figure}[htbp]
  \begin{subfigure}[b]{0.23\textwidth}
    \centering
    \includegraphics[width=\textwidth]{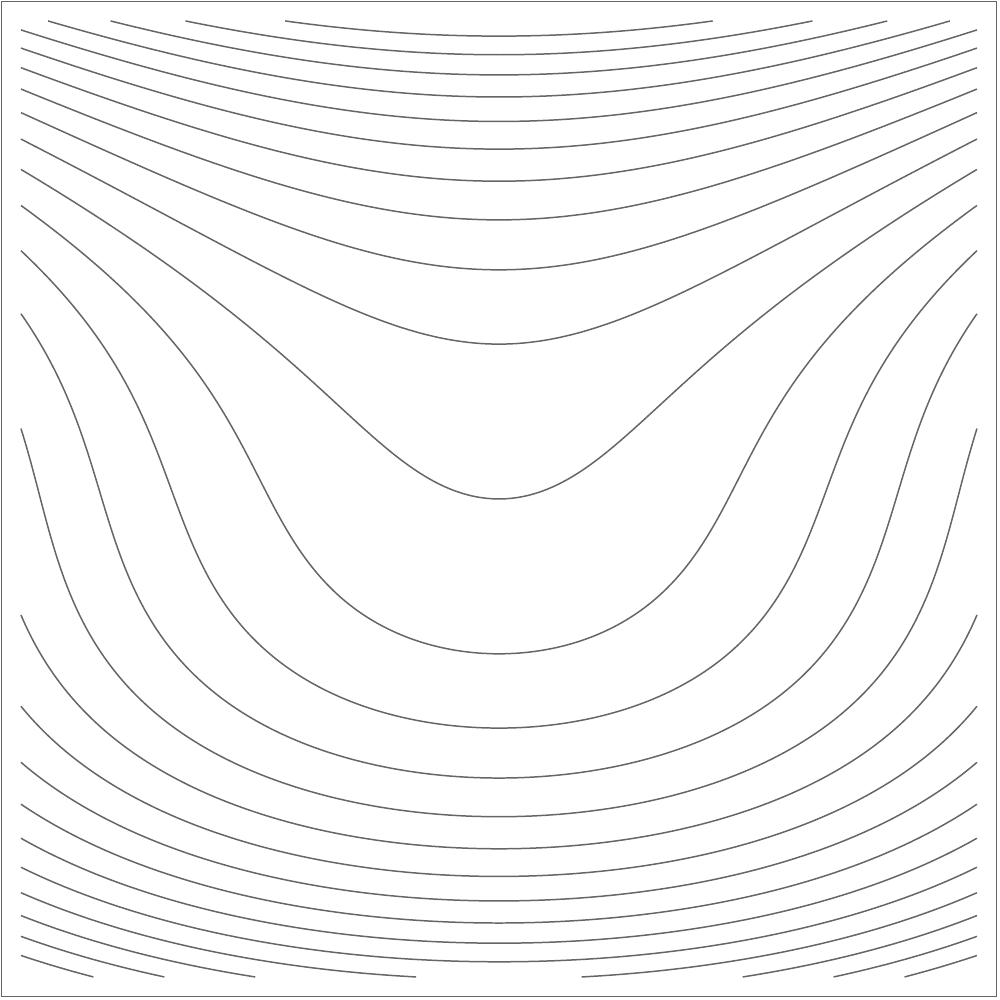}
    \caption{Source, $\eps<0$}
  \end{subfigure}
  \hfill
  \begin{subfigure}[b]{0.23\textwidth}
    \centering
    \includegraphics[width=\textwidth]{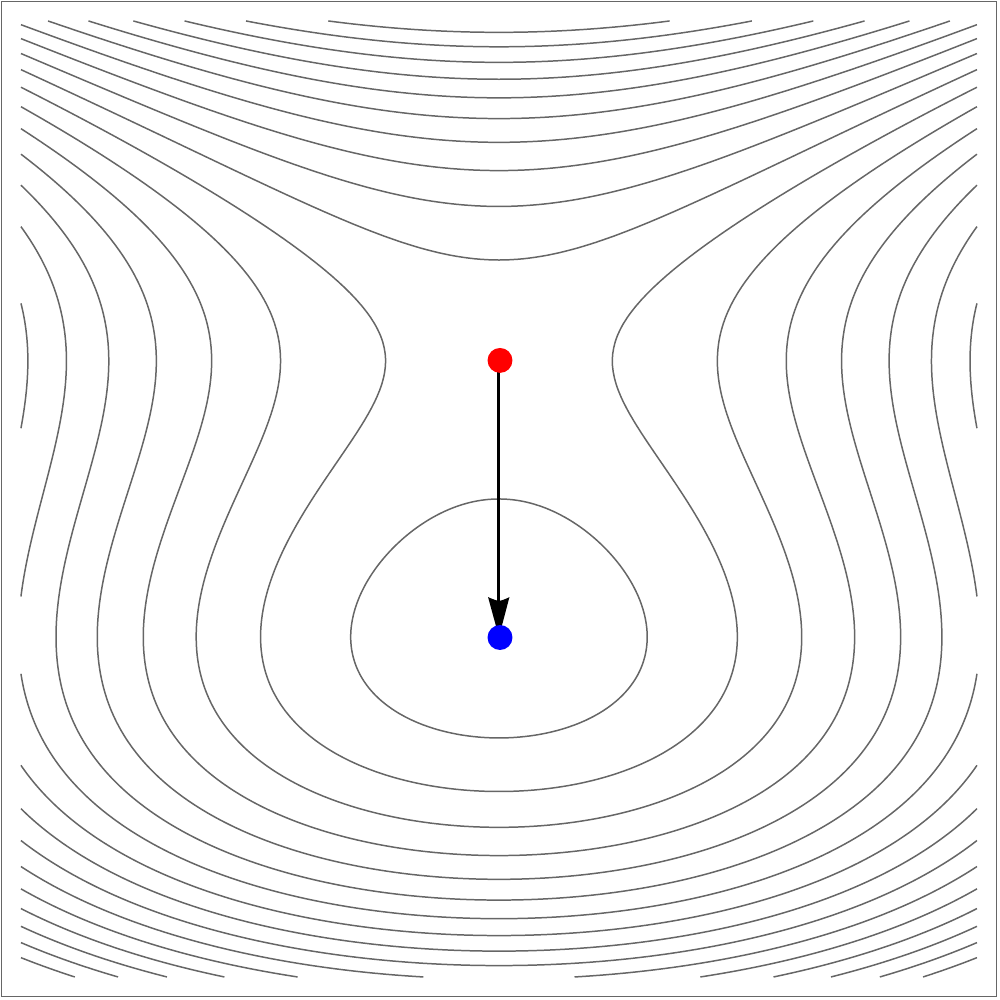}
    \caption{Source, $\eps>0$}
  \end{subfigure}
\hfill
    \begin{subfigure}[b]{0.23\textwidth}
    \centering
    \includegraphics[width=\textwidth]{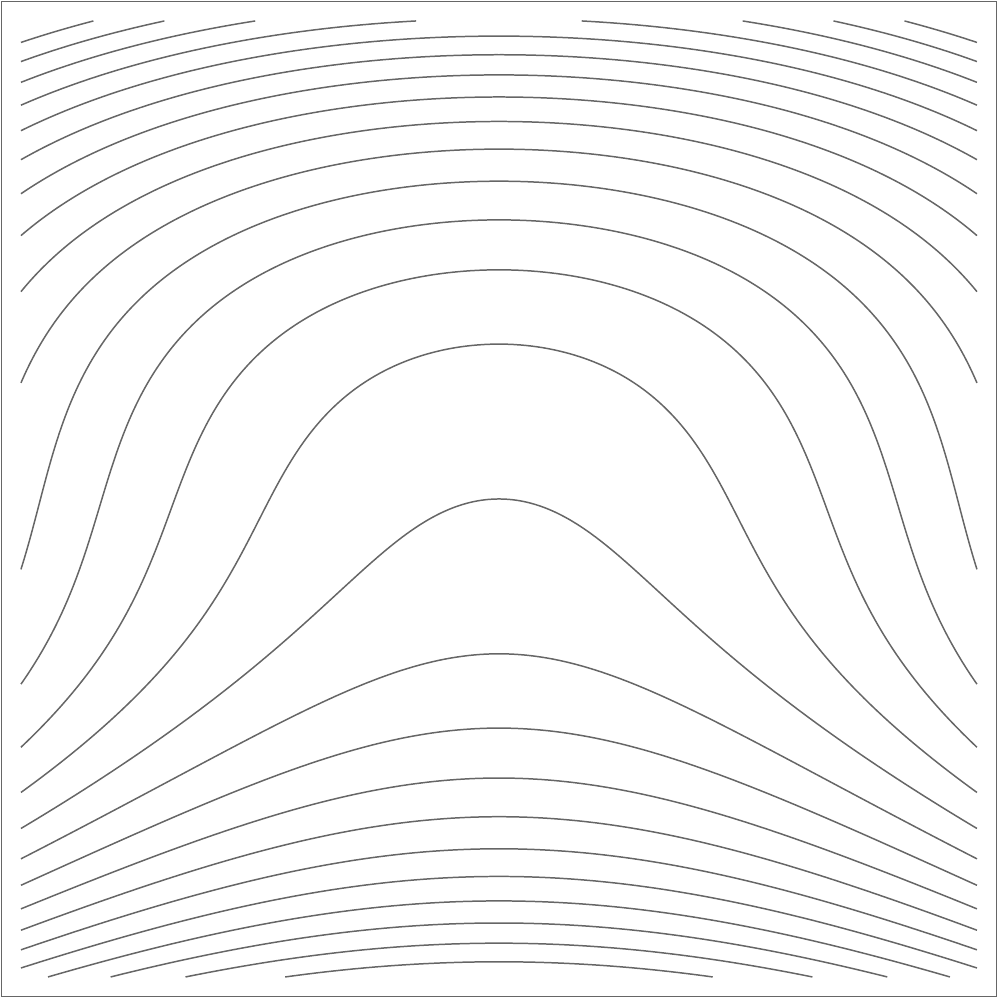}
    \caption{Sink, $\eps<0$}
  \end{subfigure}
  \hfill
  \begin{subfigure}[b]{0.23\textwidth}
    \centering
    \includegraphics[width=\textwidth]{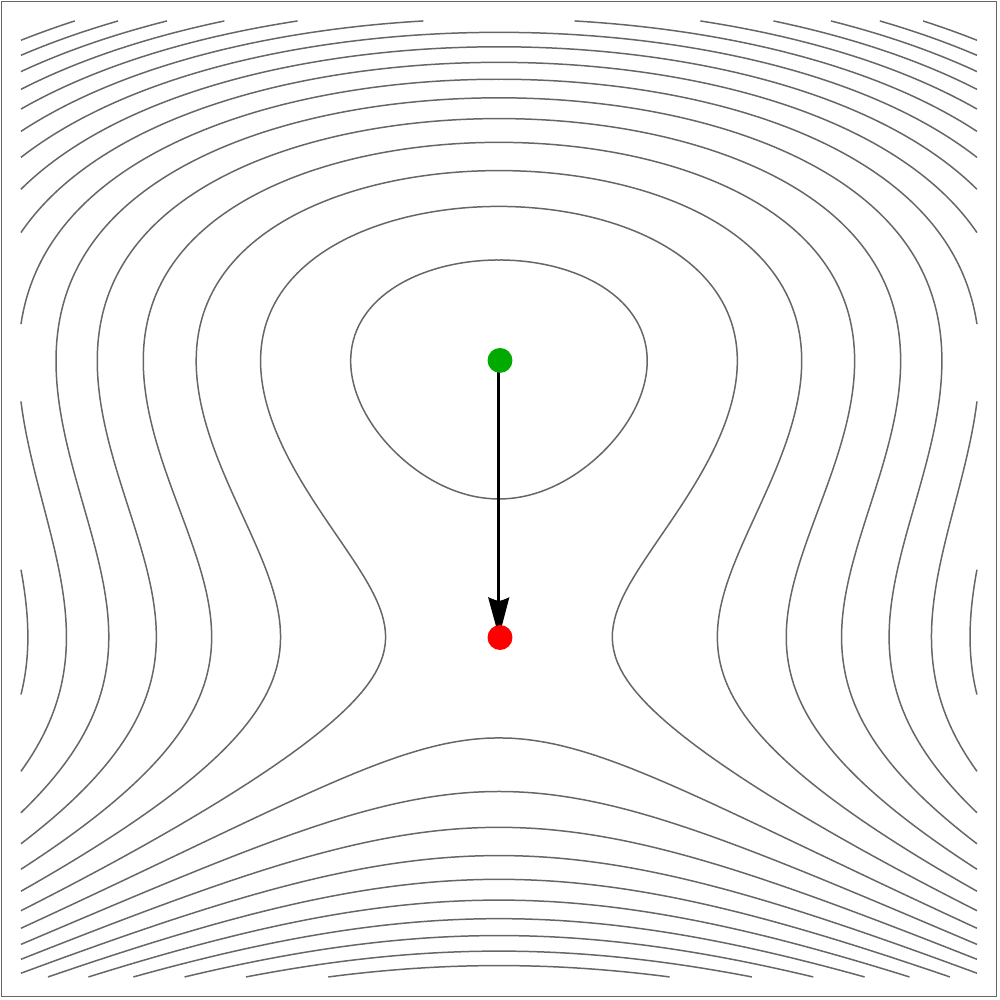}
    \caption{Sink, $\eps>0$}
  \end{subfigure}
  \caption{Level curves of $G_\eps$ and a positive gradient flow of birth-death.}
  \label{Figure - birth-death}
\end{figure}

\subsection{Period Doubling and Materialization}\label{Subsection - Period doubling}

We describe the fixed point with $\ld=-1$, in which case the double cover bifurcates.
The unwrapped generating Hamiltonian of $\Psi_0^2$ consists of $\Z_2$-symmetric polynomials, which can be written as
\[
G_0(x,y) = \frac{\alpha}{2}x^2 + \frac{\beta}{4}y^4+\ldots \in \R[x^2,y^2,xy].
\]

Note that the perturbation $G_\eps$ must also consist of $\Z_2$-symmetric polynomials, since $G_\eps$ is a generating Hamiltonian of $\Psi_\eps^2$.
Generically, we have
\[
G_\eps(x,y) = \frac{\alpha}{2}x^2 + \frac{\beta}{4}y^4 + \eps\frac{a}{2}y^2 + \ldots.
\]
The point $p$ is called a \textbf{transitional fixed point} if $\alpha$, $\beta$, and $a$ are nonzero.
For $\eps\neq0$, the initial terms of the equation $\d G_\eps=0$ are
\[
\begin{aligned}
    \pp_x G_\eps &= \alpha x  =0,\\
    \pp_y G_\eps &= \beta y^3 + \eps a y =0,
\end{aligned}
\]
whose solutions are the origin, denoted by $p_\eps$, and the points with initial terms
\[
p_{1,2}=p_{1,2}(\eps)=
\begin{pmatrix}
    x_\eps\\
    y_\eps
\end{pmatrix}
=\begin{pmatrix}
    0\\
    \pm\sqrt{-\eps a/\beta}
\end{pmatrix}.
\]
Again, by choosing an appropriate direction for the bifurcation parameter, we may assume $\eps a \beta<0$ for $\eps>0$.
Note that $p_1$ and $p_2$ cannot be fixed points of $\Psi_\eps$ for small $\eps$, since their minimal periods are 2 and they form a 2-periodic orbit of $\Psi_\eps$, i.e., $\Psi_\eps(p_{1,2})=p_{2,1}$.

The leading terms of the Hessian of $G_\eps$ are given by
\[
\begin{aligned}
\Hess G_\eps&=
\begin{pmatrix}
    \alpha &0\\
    0&\eps a + 2\beta y^2
\end{pmatrix},\\
\Hess G_\eps(p_\eps)&=\begin{pmatrix}
    \alpha &0\\
    0&\eps a\end{pmatrix},
\\
    \Hess G_\eps(p_{1,2})&=\begin{pmatrix}
    \alpha&0\\
    0&\eps a - 2\beta(\eps a/\beta)
    \end{pmatrix}
    =\begin{pmatrix}
        \alpha&0\\
        0&-\eps a
    \end{pmatrix}.
\end{aligned}
\]
We observe that the stability type of $p_\eps$ changes as $\eps$ passes through $0$, and the stability type, and moreover the Morse index, of $p_{1,2}$ is the same as that of $p_\eps$ for $\eps<0$.

\begin{thm}[\cite{Meyer_70}, Transitional Fixed Point]\label{Theorem - FPpd}
    Let $p$ be a transitional fixed point of $\Psi_\eps$.
    Then there exists a family of fixed points $p_\eps$ of $\Psi_\eps$ such that $p_0=p$, which changes its stability type as $\eps$ passes through $0$.
    
    Moreover, by an appropriate choice of the bifurcation parameter, there exists a family of 2-periodic orbits $\{p_{1,2}\}$ parametrized by $\eps>0$ which converge to $p_0$ as $\eps\to0$, such that exactly one of the following holds:
    \begin{enumerate}
        \item $p_\eps$ is elliptic if $\eps<0$ and hyperbolic if $\eps>0$.
        $p_{1,2}$ are elliptic, and have the same Morse index as $p_\eps$ for $\eps<0$ as critical points of the unwrapped generating Hamiltonian $G_\eps$.
        \item $p_\eps$ is hyperbolic if $\eps<0$ and elliptic if $\eps>0$.
        $p_{1,2}$ are hyperbolic.
    \end{enumerate}
\end{thm}
Traditionally, we divide this into two cases depending on the change in the stability type of $p_\eps$.
We again translate this into the periodic orbits of the Hamiltonian.
\begin{thm}[Period Doubling and Materialization]\label{Theorem - POpd}
    Let $H_\eps$ be a 1-parameter family of Hamiltonians and $\g$ be a periodic orbit of $H_0$ with Floquet multiplier $\ld=-1$.
    Assume that a point $p$ in $\g$ is a transitional fixed point of the return map $\Psi_\eps$, and let $\g_\eps$ be the orbit cylinder centered at $\g$.
    Then, by an appropriate choice of the bifurcation parameter, one of the following happens:
        \begin{enumerate}
            \item $\g_\eps$ is elliptic if $\eps<0$ and hyperbolic if $\eps>0$.
                 There exists a family of elliptic periodic orbits $\g_E(\eps)$ of $H_\eps$ parametrized by $\eps>0$, which converge to $\g^2$ as $\eps\to0$.
            \item $\g_\eps$ is hyperbolic if $\eps<0$ and elliptic if $\eps>0$.
                There is a family of hyperbolic periodic orbits $\g_H(\eps)$ of $H_\eps$ parametrized by $\eps>0$, which converge to $\g^2$ as $\eps\to0$.
        \end{enumerate}
\end{thm}

\cite{Abraham_Marsden_78} called the first case \textbf{period doubling} (or \emph{subtle doubling}) and the second case \textbf{materialization}.
If there exist incoming (resp. outgoing) gradient flows into the origin after the bifurcation, we call the period doubling or materialization \textbf{attracting} (resp. \textbf{repelling}).
The two situations are described in \Cref{Figure - Period doubling,Figure - Materialization}.

\begin{figure}[htbp]
  \begin{subfigure}[b]{0.23\textwidth}
    \centering
    \includegraphics[width=\textwidth]{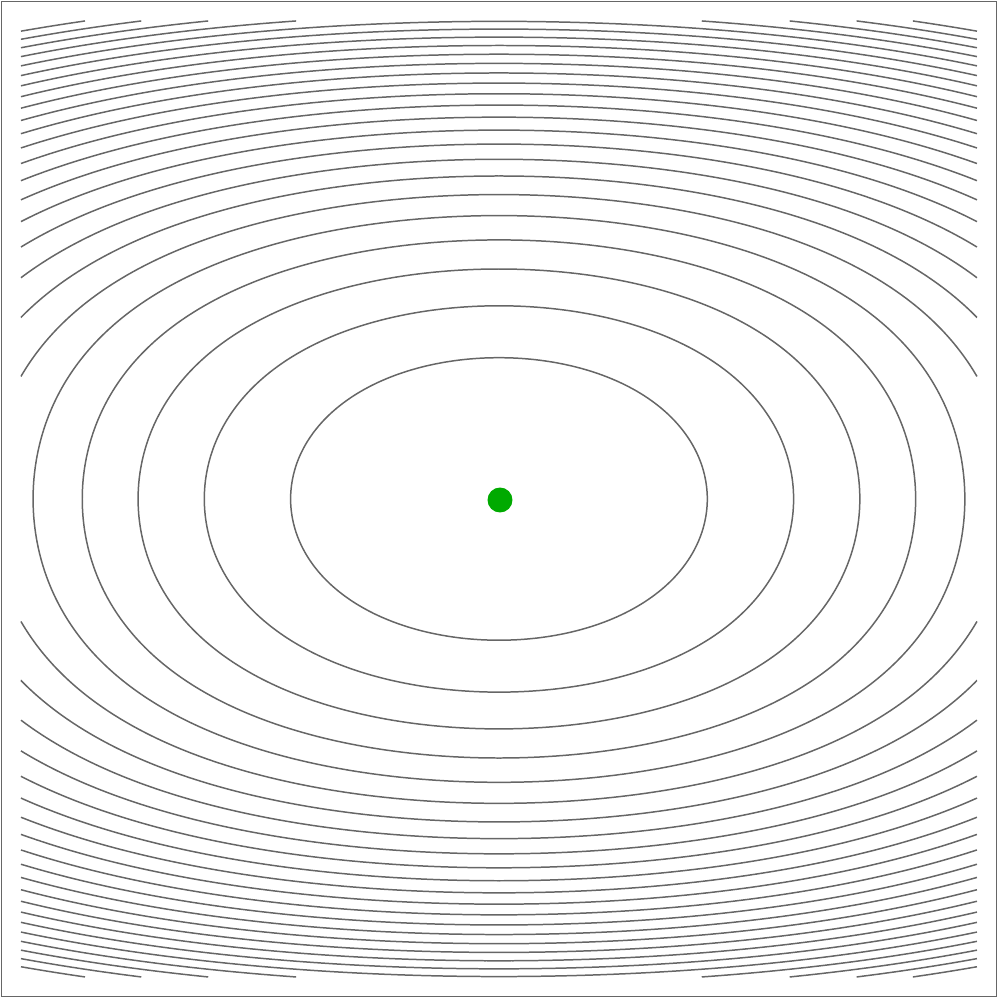}
    \caption{Attracting, $\eps<0$}
  \end{subfigure}
  \hfill
  \begin{subfigure}[b]{0.23\textwidth}
    \centering
    \includegraphics[width=\textwidth]{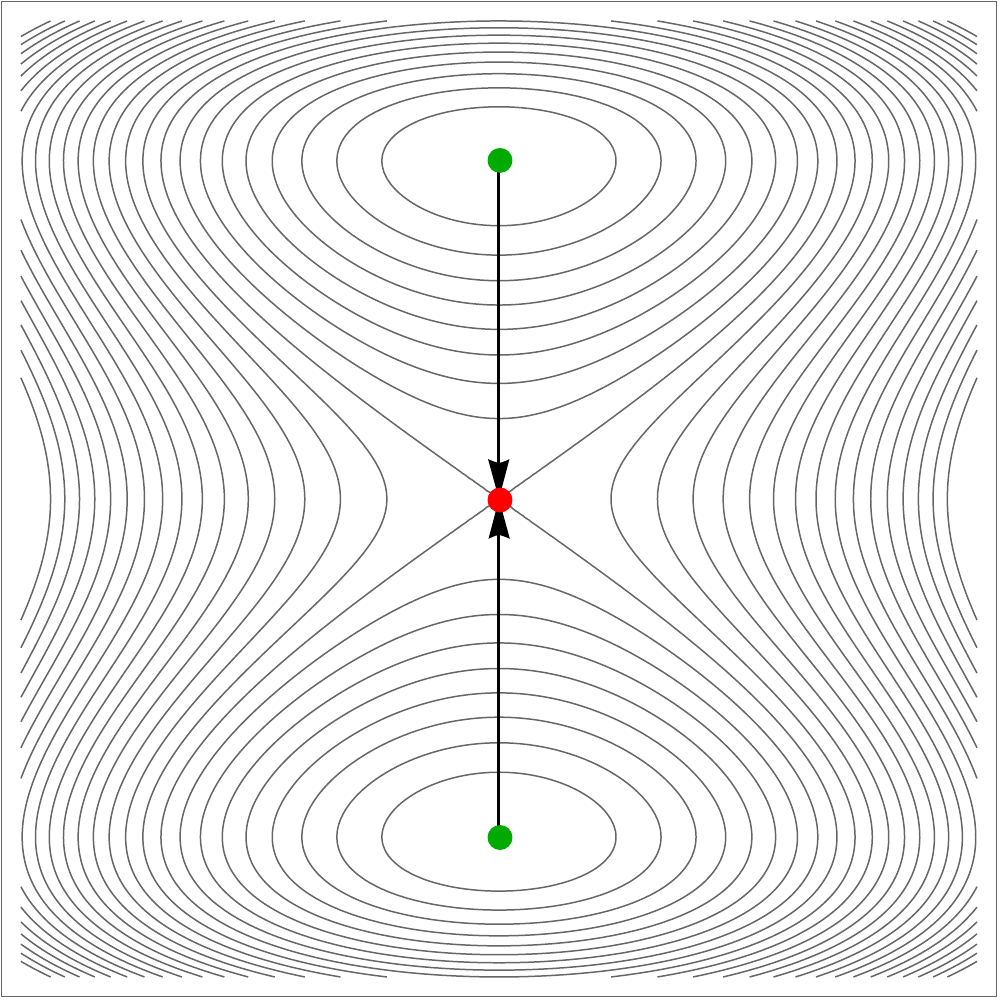}
    \caption{Attracting, $\eps>0$}
  \end{subfigure}
\hfill
    \begin{subfigure}[b]{0.23\textwidth}
    \centering
    \includegraphics[width=\textwidth]{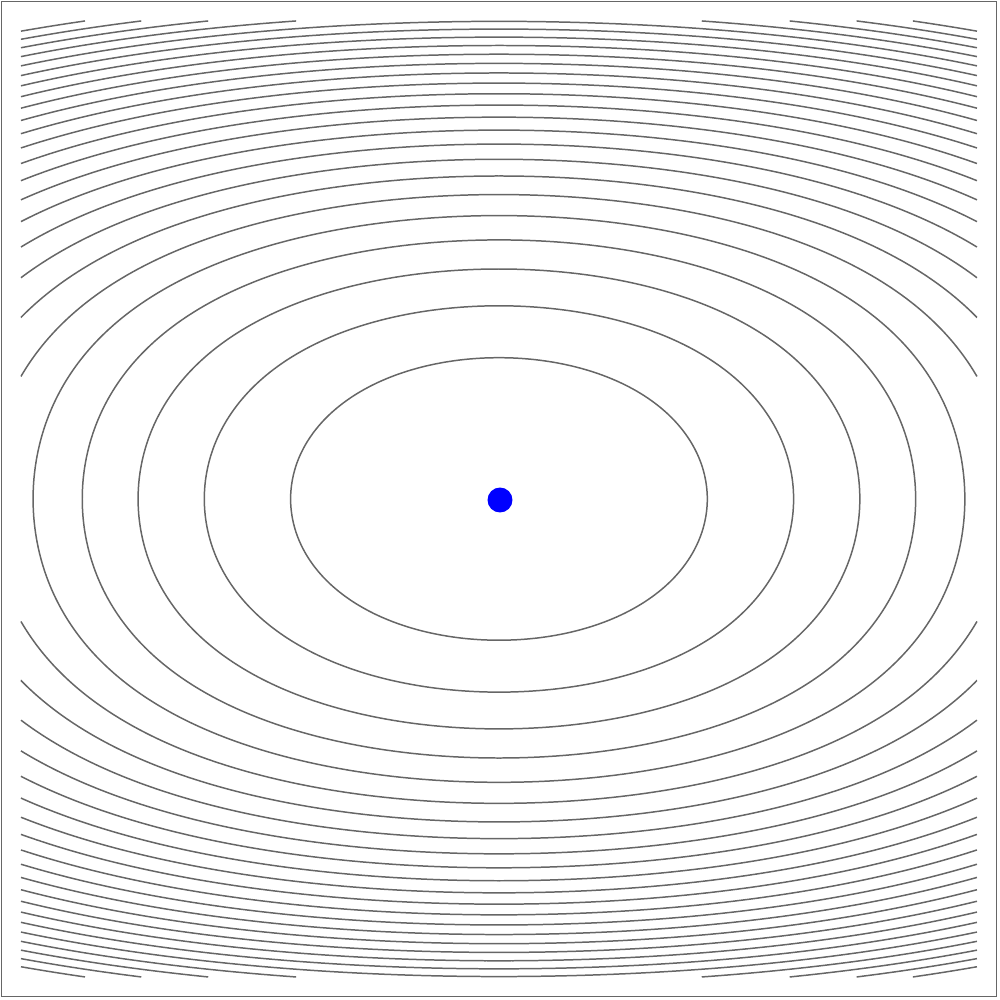}
    \caption{Repelling, $\eps<0$}
  \end{subfigure}
  \hfill
  \begin{subfigure}[b]{0.23\textwidth}
    \centering
    \includegraphics[width=\textwidth]{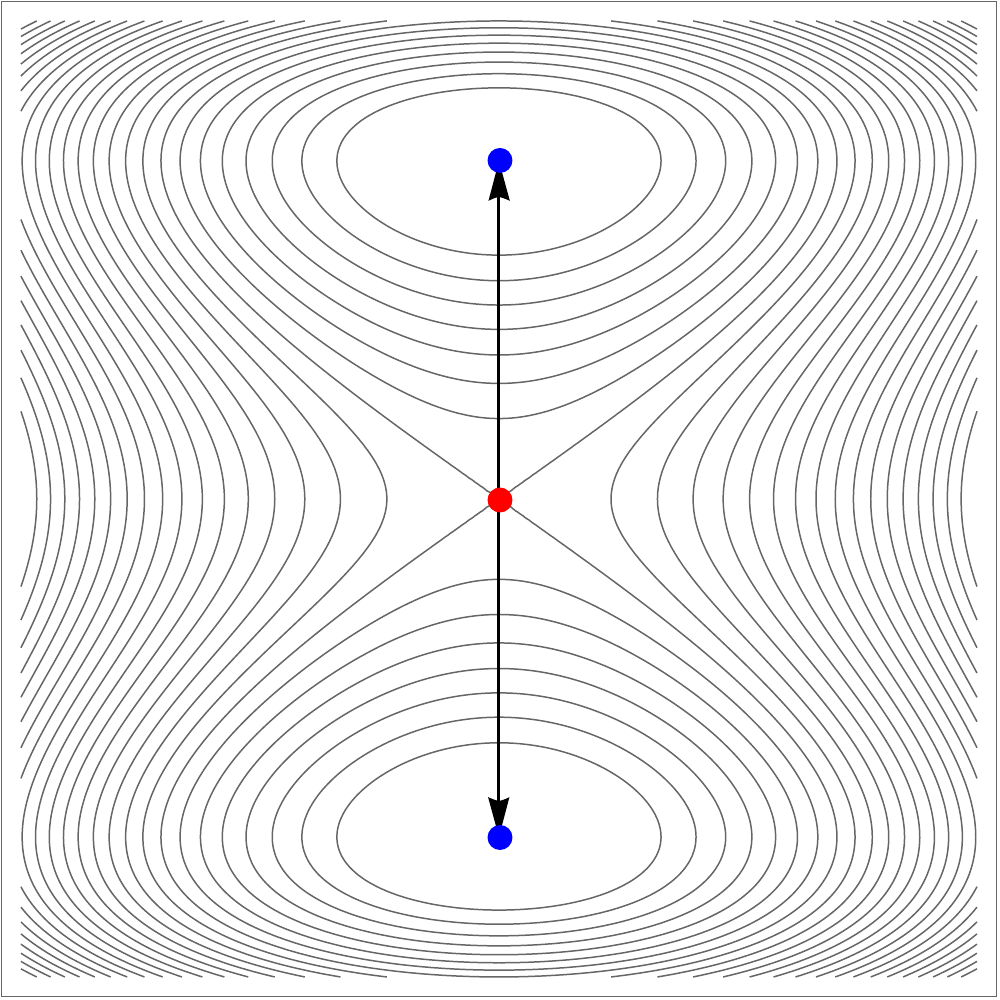}
    \caption{Repelling, $\eps>0$}
  \end{subfigure}
  \caption{Level curves of $G_\eps$ and gradient flows of period doubling.}
  \label{Figure - Period doubling}
\end{figure}

\begin{figure}[htbp]
  \begin{subfigure}[b]{0.23\textwidth}
    \centering
    \includegraphics[width=\textwidth]{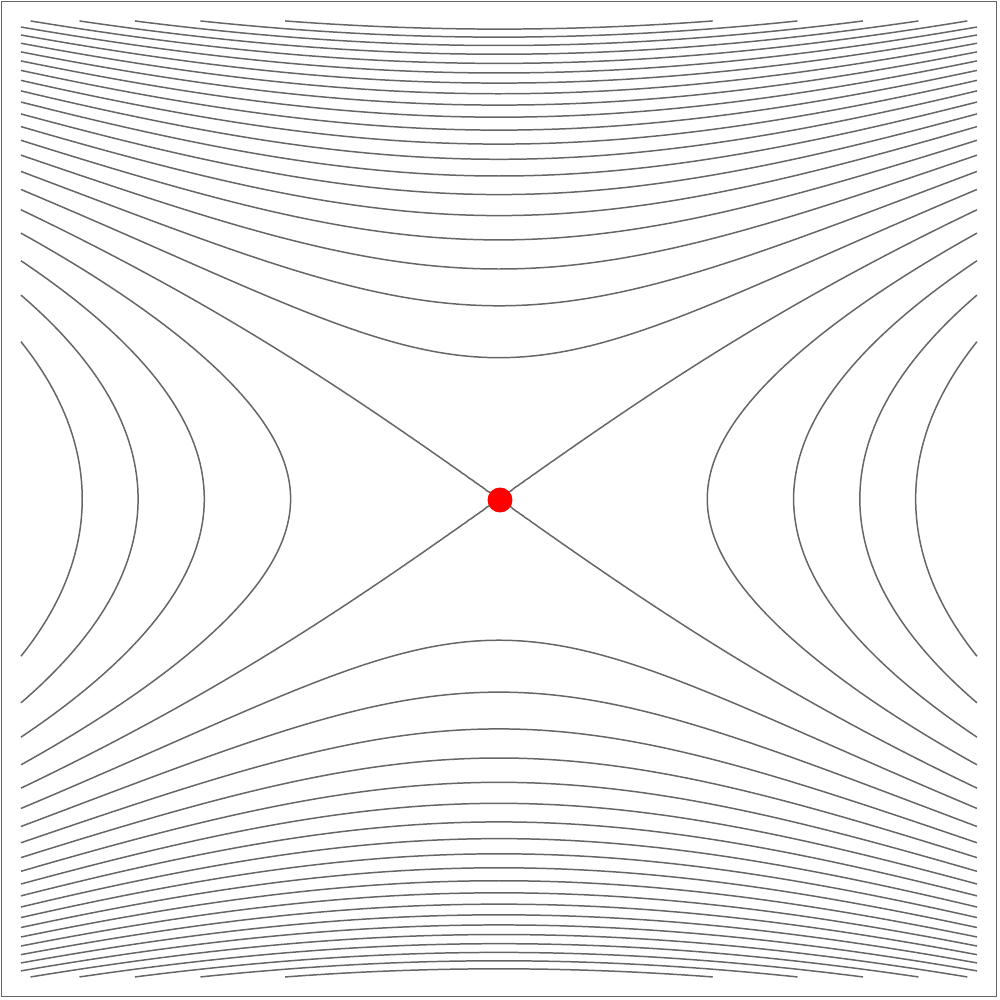}
    \caption{Attracting, $\eps<0$}
  \end{subfigure}
  \hfill
  \begin{subfigure}[b]{0.23\textwidth}
    \centering
    \includegraphics[width=\textwidth]{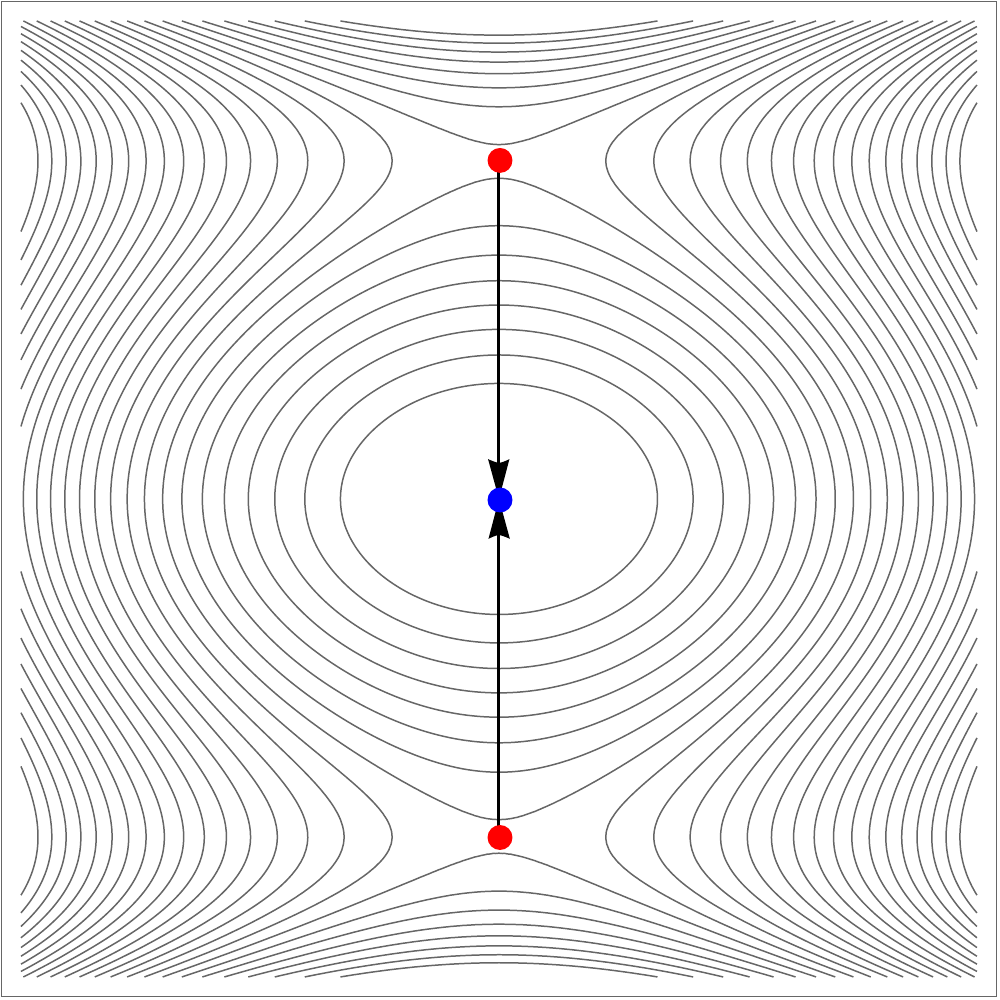}
    \caption{Attracting, $\eps>0$}
  \end{subfigure}
\hfill
    \begin{subfigure}[b]{0.23\textwidth}
    \centering
    \includegraphics[width=\textwidth]{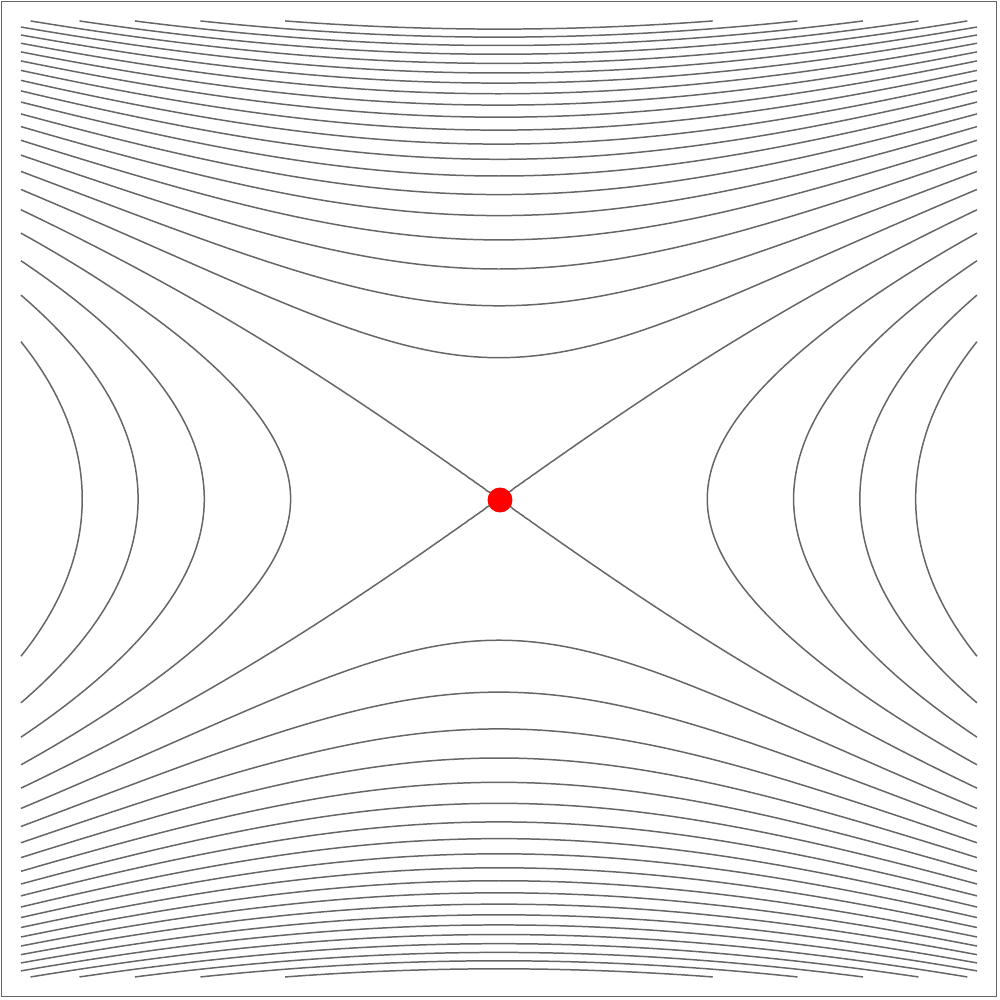}
    \caption{Repelling, $\eps<0$}
  \end{subfigure}
  \hfill
  \begin{subfigure}[b]{0.23\textwidth}
    \centering
    \includegraphics[width=\textwidth]{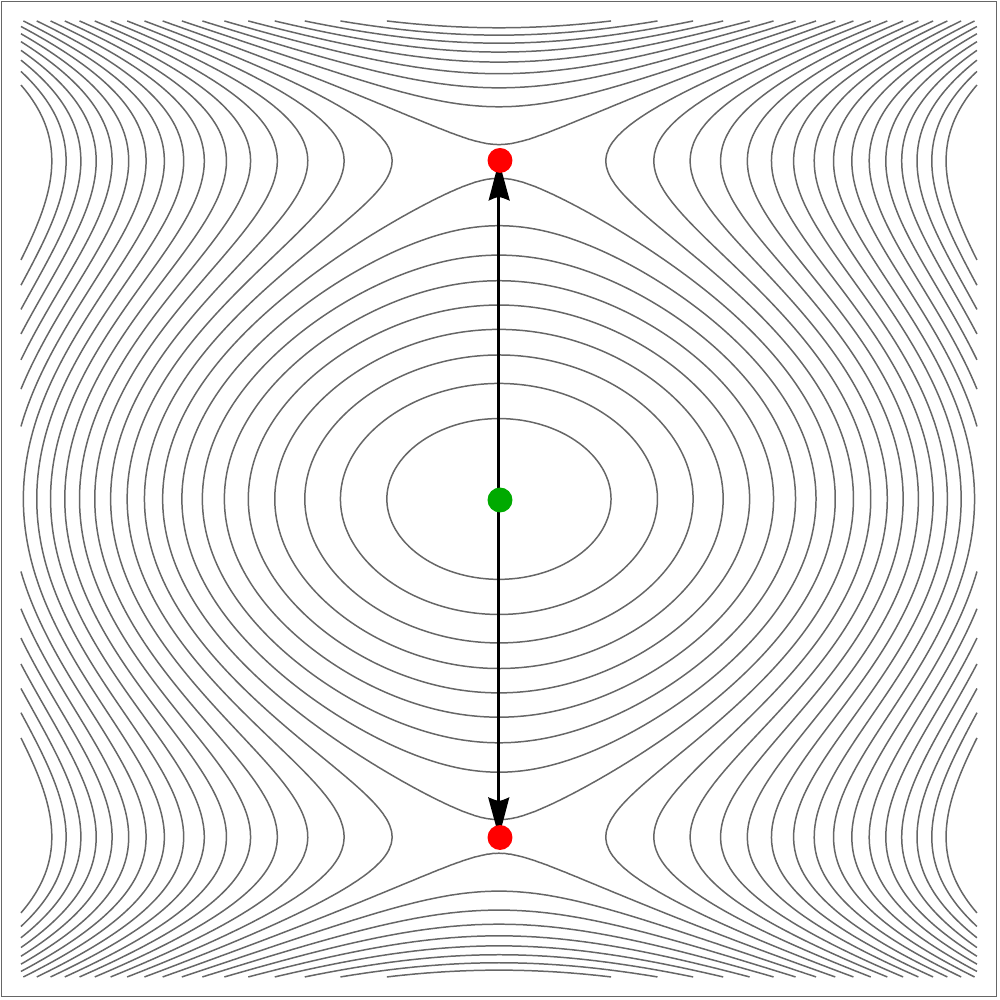}
    \caption{Repelling, $\eps>0$}
  \end{subfigure}
  \caption{Level curves of $G_\eps$ and gradient flows of materialization.}
  \label{Figure - Materialization}
\end{figure}

\subsection{Phantom Kiss and Emission}\label{Subsection - k bifurcation}

We describe the fixed point with a Floquet multiplier of $\exp(2\pi i l/k)$, where $k,l\in\N$ and $(k,l)=1$.
We again consider the $k$-th iteration of the return map, $\Psi_0^k$, which has $0$ as a fixed point, and consider its unwrapped generating Hamiltonian $G_0$. The leading terms of $G_0$ in polar coordinates $(\rho, \vp)$ are
    \[
    G_0(\rho,\vp) = \frac{\alpha}{4} \rho^4 + \frac{\beta}{k}\rho^k \cos k\vp + \ldots.
    \]
Note that the $\rho^k\sin k\vp$ term can be absorbed into $\rho^k\cos k\vp$ by a translation of $\vp$.
For the same reason as in the $k=2$ case, the perturbation terms must be chosen as symmetric polynomials in $\R[\rho^2,\rho^k\cos k\vp, \rho^k\sin k\vp]$.
The perturbed Hamiltonian has the form
    \[
    G_\eps(\rho,\vp) = \frac{\alpha}{k}\rho^k \cos k\vp + \frac{\beta}{4}\rho^4   +\eps\frac{ a}{2} \rho^2 +  \ldots.
    \]
We say that the origin is a \textbf{$k$-bifurcation point} if:
\begin{itemize}
    \item $\alpha$ and $a$ are nonzero for $k=3$,
    \item $\alpha$, $a$, and $\alpha\pm \beta$ are nonzero for $k=4$,
    \item $\alpha$, $a$, and $\beta$ are nonzero for $k\geq 5$.
\end{itemize}

The leading terms of $G_0$ essentially depend on $k$, so we divide the cases.
We first consider $k=3$.
The leading terms of $G_\eps$ are
    \[
    G_\eps(\rho,\vp)=  \frac{\alpha}{3}\rho^3 \cos 3\vp +\eps\frac{ a}{2} \rho^2 + \ldots.
    \]
For any $\eps$, $G_\eps$ has the origin as a critical point, say $p_\eps$.
Moreover, the leading terms of the other critical points should satisfy
    \[
    \begin{aligned}
        \pp_\rho G_\eps&= \alpha \rho^2 \cos 3\vp + \eps a\rho =0,\\
        \pp_\vp G_\eps&= -\alpha \rho^3 \sin 3\vp =0,
    \end{aligned}
    \]
which gives six critical points:
\[
\begin{aligned}
    \rho_j &= -\frac{\eps a}{\alpha \cos 3\vp_j} = \pm \frac{\eps a}{\alpha},\\
    \vp_j&= \frac{j\pi}{3}\text{ for }j=0,1,\ldots,5.
\end{aligned}
\]
Since $\rho_j$ should be positive, there exist three solutions for each sign of $\eps$, depending on the sign of $\cos 3\vp_j$.
Again, these are not fixed points of $\Psi_\eps$, so they form 3-periodic orbits for each sign of $\eps$.

We observe that the origin is an elliptic critical point, and changes from a minimum to a maximum, or vice versa, as $\eps$ passes through $0$, depending on the sign of $a$.
By taking an appropriate direction for the bifurcation parameter, we can assume that it changes from a minimum to a maximum.
Then the solutions with $\cos 3\vp=1$ survive for $\eps<0$, while those with $\cos 3\vp =-1$ survive for $\eps>0$.
We denote each case by $p_\pm^j$.
For the other critical points, we compute the leading terms of the Hessian:
    \[
    \Hess G_\eps = \begin{pmatrix}
        a\eps + 2\alpha \rho \cos 3\vp & -3\alpha \rho^2 \sin 3\vp\\
        -3\alpha \rho^2 \sin 3\vp & -3\alpha \rho^3 \cos 3\vp
    \end{pmatrix}
    =\begin{pmatrix}
        -\eps a&0\\
        0&\mp 2\alpha \rho^3
    \end{pmatrix}.
    \]
By our previous choice of $\eps$, we can see that the $p_\pm^j$'s are all saddle points.
    
\begin{thm}[\cite{Meyer_70}, 3-bifurcation]\label{Theorem - FPpk}
    Let $p$ be a 3-bifurcation point of $\Psi_\eps$.
    Then there exists a family of critical points $p_\eps$ such that $p_0=p$.
    Moreover, there exist two families of hyperbolic 3-periodic orbits $\{p_-^j(\eps)\}$ and $\{p_+^j(\eps)\}$ of $\Psi_\eps$, parametrized by $\eps<0$ and $\eps>0$ respectively, and both converge to $p$ as $\eps\to0$.
    
    Furthermore, if $p_\eps$ is a maximum (resp. minimum) when $\eps<0$, it becomes a minimum (resp. maximum) when $\eps>0$.
\end{thm}
    
\begin{thm}[Phantom Kiss]\label{Theorem - POpk}
    Let $H_\eps$ be a 1-parameter family of Hamiltonians and $\g$ be a periodic orbit of $H_0$ with Floquet multiplier $\ld=\exp(2\pi i l/3)$ where $l=1,2$.
    Assume that a point $p$ in $\g$ is a 3-bifurcation point of the return map, and let $\g_\eps$ be the orbit cylinder centered at $\g$.

    Then there exist two families of hyperbolic orbits $\{\g_{H,-}(\eps)\}$ and $\{\g_{H,+}(\eps)\}$ of $H_\eps$, parametrized by $\eps<0$ and $\eps>0$, and both converge to $\g^3$ as $\eps\to0$.
\end{thm}

We call this phenomenon a \textbf{phantom kiss}.
The term \emph{touch-and-go} is also used in much of the literature.
The situation is illustrated in \Cref{Figure - Phantom kiss}.
As with the other bifurcations discussed earlier, the origin can also change from a maximum to a minimum.
However, this can be handled by taking the opposite direction for the bifurcation parameter, so we do not distinguish these cases separately.

\begin{figure}[htbp]
  \begin{subfigure}[b]{0.23\textwidth}
    \centering
    \includegraphics[width=\textwidth]{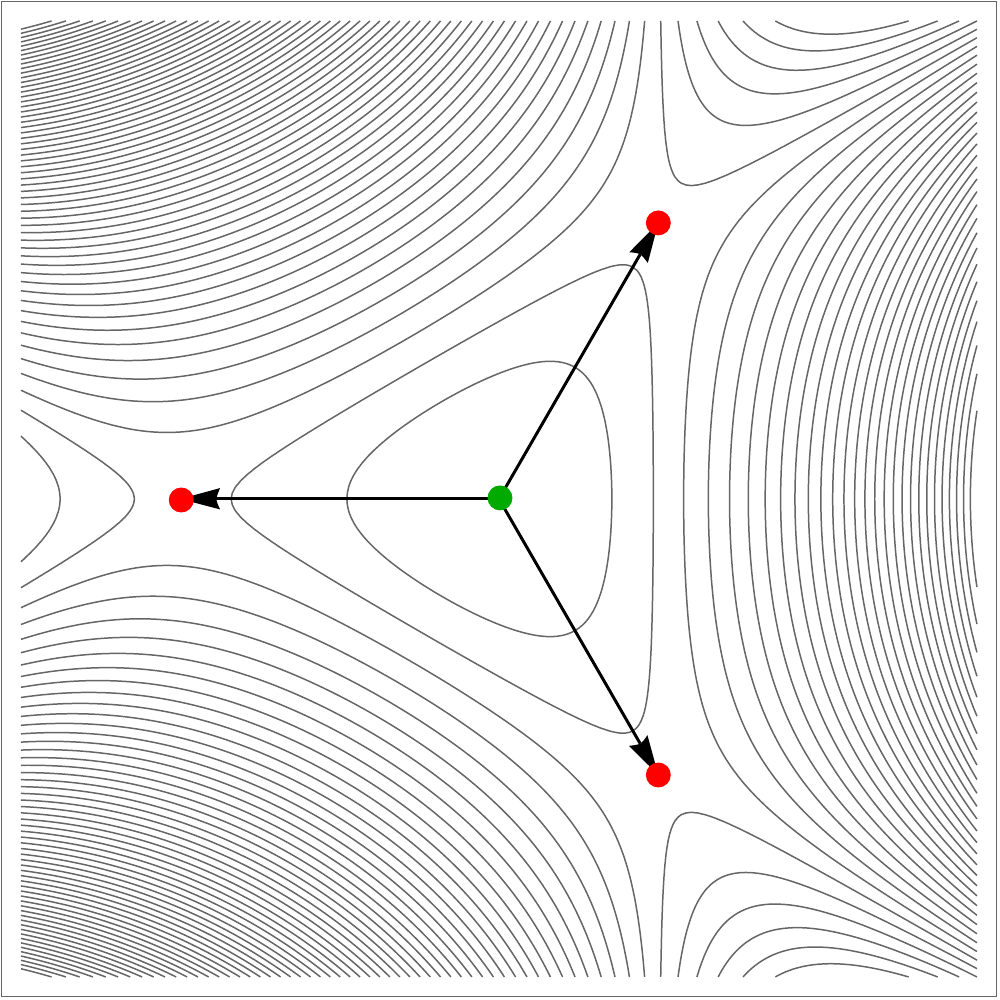}
    \caption{$\eps<0$}
  \end{subfigure}
\hfill  \begin{subfigure}[b]{0.23\textwidth}
    \centering
    \includegraphics[width=\textwidth]{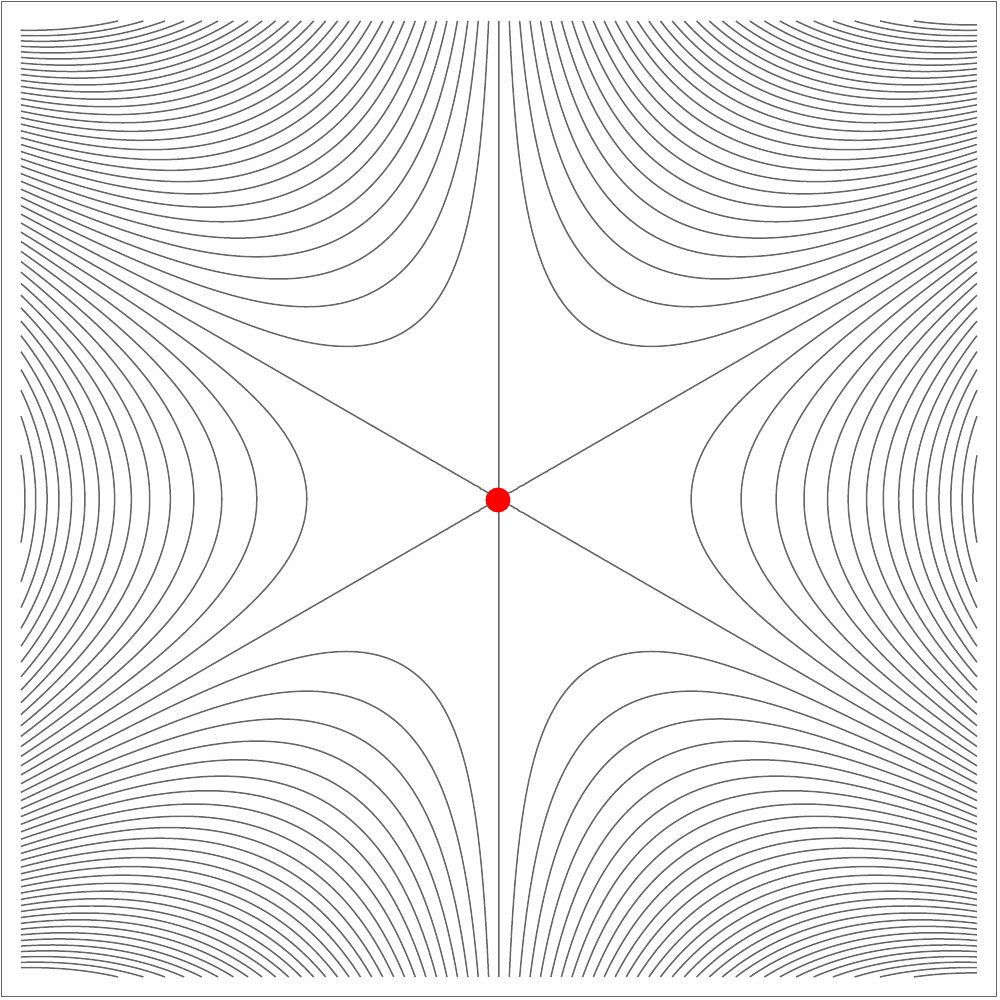}
    \caption{$\eps=0$}
  \end{subfigure}
\hfill  \begin{subfigure}[b]{0.23\textwidth}
    \centering
    \includegraphics[width=\textwidth]{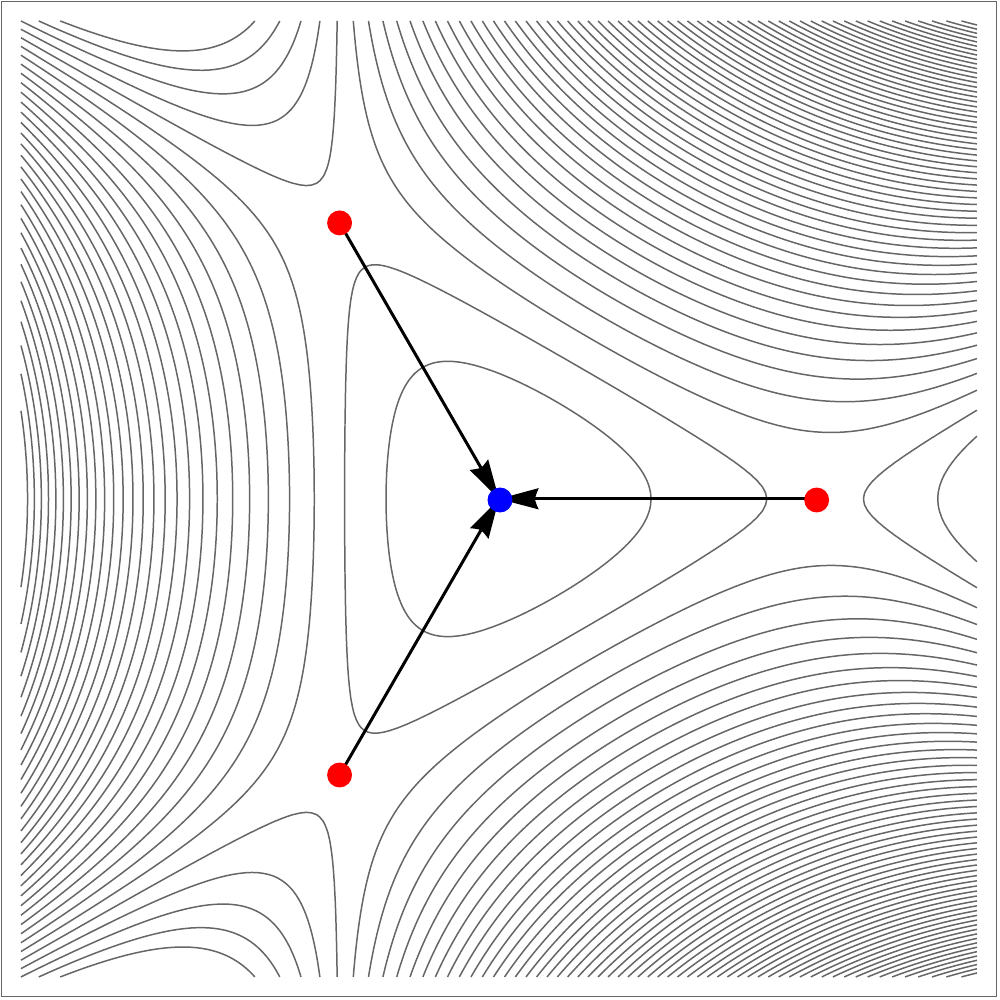}
    \caption{$\eps>0$}
  \end{subfigure}

  \caption{Level curves of $G_\eps$ and gradient flows of 3-phantom kiss.}
  \label{Figure - Phantom kiss}
\end{figure}

Before discussing the $k=4$ case, we discuss the $k\geq5$ cases.
The leading terms of the unwrapped generating Hamiltonian are
    \[
    G_\eps (\rho,\vp) =  \frac{\alpha}{k}\rho^k\cos k\vp+\frac{\beta}{4}\rho^4 +  \eps\frac{ a}{2}\rho^2 + \ldots.
    \]
Again, the origin is always a critical point.
The conditions for the other critical points are
    \[
    \begin{aligned}
    \pp_\rho G_\eps &= \eps a \rho + \beta \rho^3=0,\\
    \pp_\vp G_\eps &=\alpha \rho^{k-1}\sin k\vp =0,
    \end{aligned}
    \]
whose solutions are
    \[
    \begin{aligned}
        \rho^2&=-\frac{\eps a}{\beta},\\
        \vp_j& = \frac{j\pi}{k}\text{ for }j=0,1,\ldots,2k-1.
    \end{aligned}
    \]
Unlike the $k=3$ case, only one sign of $\eps$ is available for these solutions.
We choose a direction for $\eps$ such that $-\eps a \beta>0$ for $\eps>0$, so the solutions exist for $\eps>0$.

The origin transforms from a minimum to a maximum, or vice versa, as $\eps$ passes through $0$, depending on the sign of $a$.
For the other critical points, we compute the leading terms of the Hessian:
    \[
    \Hess G_\eps = \begin{pmatrix}
        \eps a + 3\beta \rho^2 &0\\
        0& -k\alpha \rho^{k-1} \cos k\vp
    \end{pmatrix}
    =
    \begin{pmatrix}
        -2\eps a&0\\
        0&\mp k\alpha \rho^{k-1}
    \end{pmatrix},
    \]
depending on the sign of $\cos k\vp$.
We divide the cases into two:
\begin{itemize}
    \item If $a>0$, $p_\eps$ changes from a maximum to a minimum.
    For $\eps>0$, $-2\eps a<0$, so there are $k$ saddle points and $k$ maxima.
    \item If $a<0$, $p_\eps$ changes from a minimum to a maximum.
    For $\eps>0$, $-2\eps a>0$, so there are $k$ saddle points and $k$ minima.
\end{itemize}
In both cases, we denote the saddle points by $p_H^j$ and the elliptic points by $p_E^j$.
To summarize, we have the final class of the bifurcation.

\begin{thm}[\cite{Meyer_70}, $k$-bifurcation]\label{Theorem - FPem}
    Let $p$ be a $k$-bifurcation point of $\Psi_\eps$ for $k\geq5$.
    Then there exists a family of fixed points $p_\eps$ of $\Psi_\eps$ such that $p_0=p$.
    Also, with an appropriate choice of the bifurcation parameter, there exists a family of hyperbolic $k$-periodic orbits $\{p_H^j(\eps)\}$ and a family of elliptic $k$-periodic orbits $\{p_E^j(\eps)\}$, parametrized by $\eps>0$, which converge to $p$ as $\eps\to0$.
    
    Moreover, as critical points of the unwrapped generating Hamiltonian of $\Psi_\eps^k$, $p_\eps$ is a minimum (resp. maximum) if $\eps<0$ and a maximum (resp. minimum) if $\eps>0$, and the elliptic periodic points $p_E^j(\eps)$ are minima (resp. maxima) for $\eps>0$.
\end{thm}

\begin{thm}[Emission]\label{Theorem - POem}
Let $H_\eps$ be a 1-parameter family of Hamiltonians and $\g$ be a periodic orbit of $H_0$ with Floquet multiplier $\ld=\exp(2\pi i l/k)$ where $(k,l)=1$ and $k\geq 5$.
Assume that a point $p$ in $\g$ is a $k$-bifurcation point of the return map, and let $\g_\eps$ be the orbit cylinder centered at $\g$.

Then, with an appropriate choice of the bifurcation parameter, there exists a family of hyperbolic periodic orbits $\{\g_H(\eps)\}$ and a family of elliptic periodic orbits $\{\g_E(\eps)\}$ of $H_\eps$, both parametrized by $\eps>0$, which converge to $\g^k$ as $\eps\to0$.
\end{thm}

We call this phenomenon an \textbf{emission}.
\cite{Abraham_Marsden_78} called the opposite direction an \emph{absorption}.
As in the period doubling and materialization cases, we distinguish the two possible situations by the existence of incoming or outgoing Morse trajectories at the origin after the bifurcation, calling each case \textbf{attracting} or \textbf{repelling}.

\begin{figure}[htbp]

    \begin{subfigure}[b]{0.23\textwidth}
    \centering
    \includegraphics[width=\textwidth]{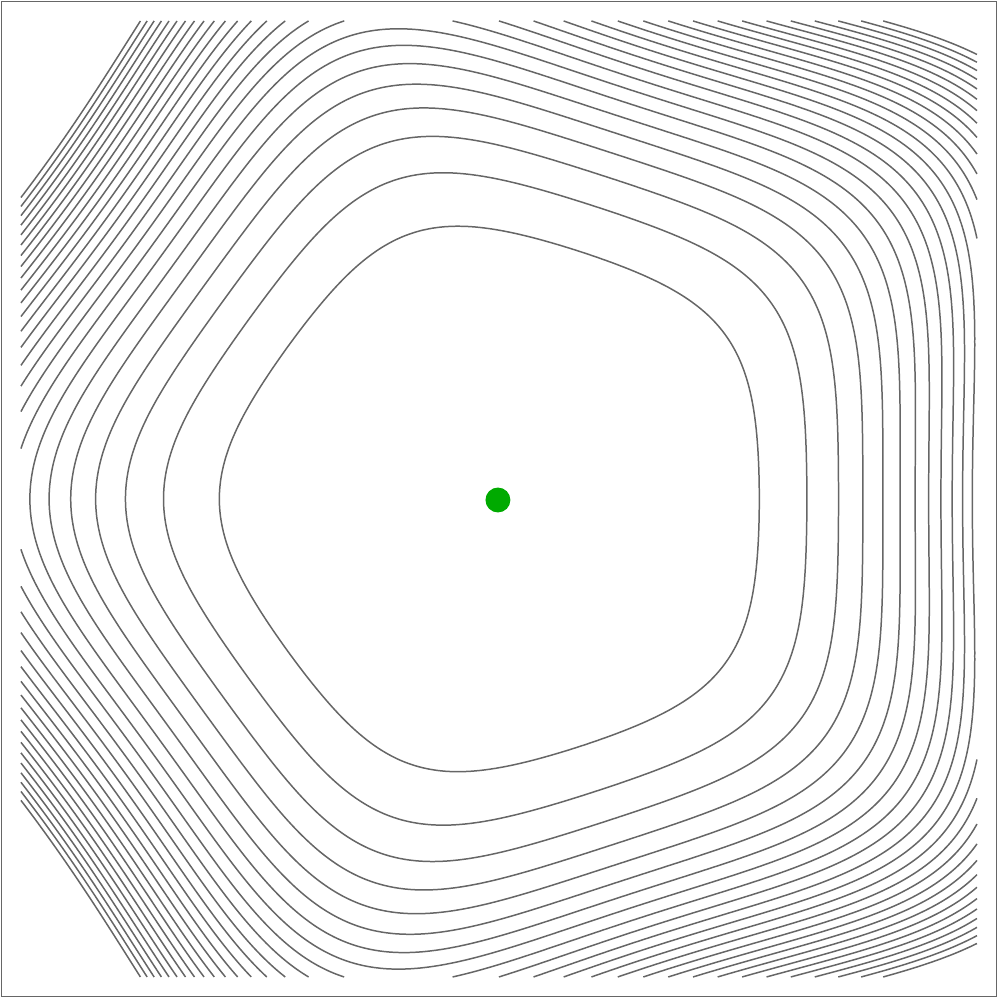}
    \caption{Attracting, $\eps<0$}
  \end{subfigure}
\hfill
\begin{subfigure}[b]{0.23\textwidth}
    \centering
    \includegraphics[width=\textwidth]{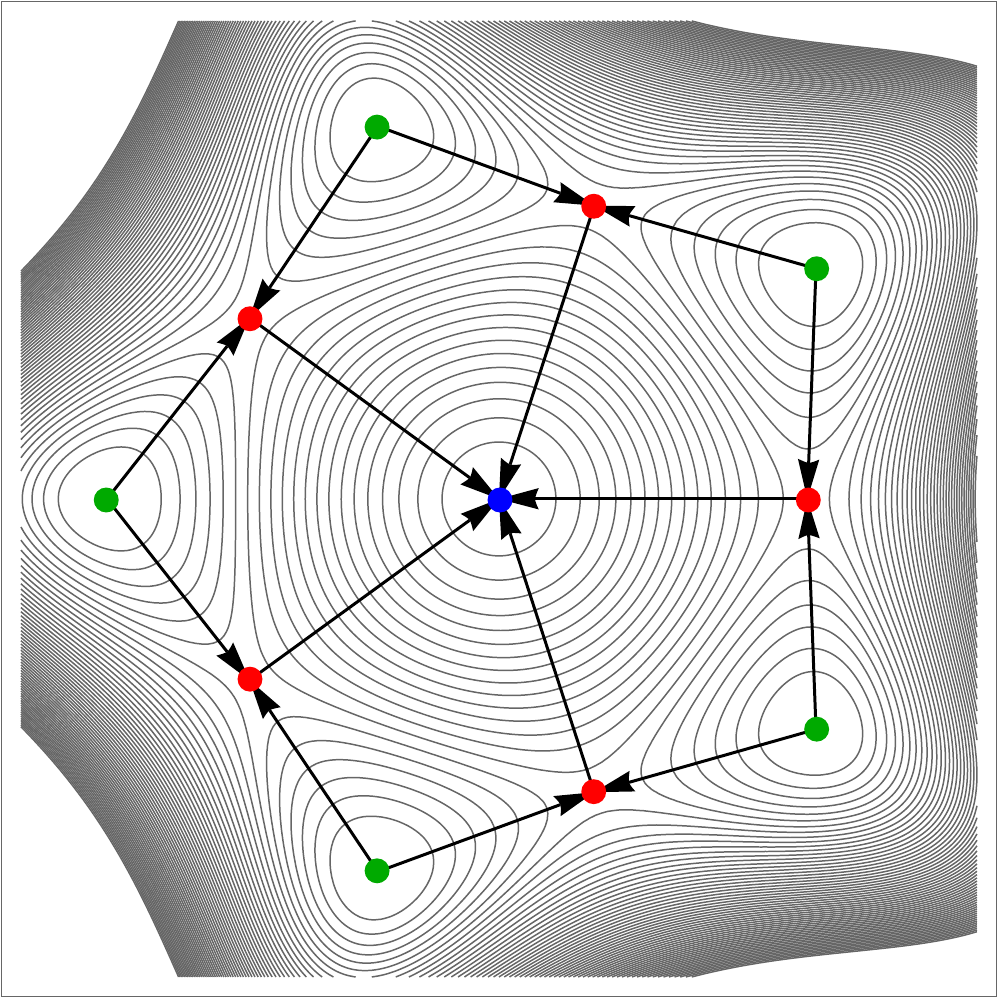}
    \caption{Attracting, $\eps>0$}
    \end{subfigure}
    \hfill
  \begin{subfigure}[b]{0.23\textwidth}
    \centering
    \includegraphics[width=\textwidth]{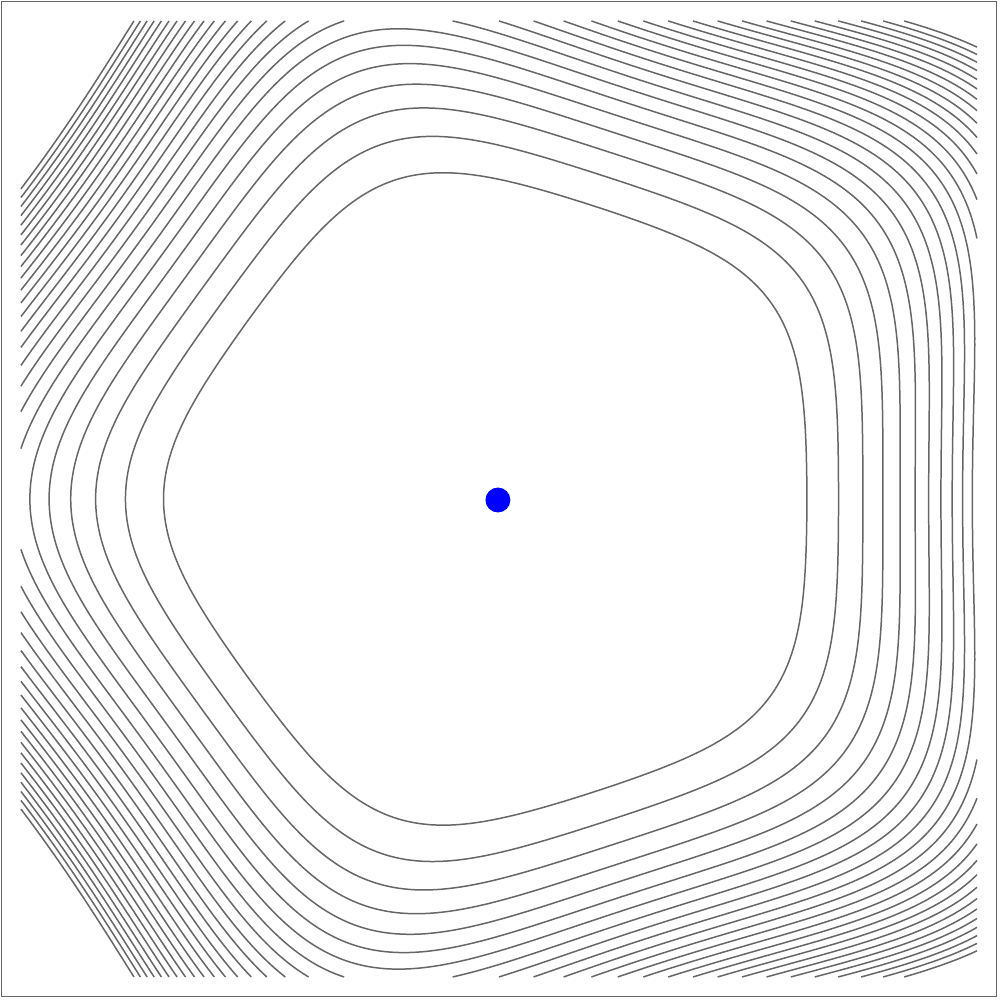}
    \caption{Repelling, $\eps<0$}
  \end{subfigure}
 \hfill
  \begin{subfigure}[b]{0.23\textwidth}
    \centering
    \includegraphics[width=\textwidth]{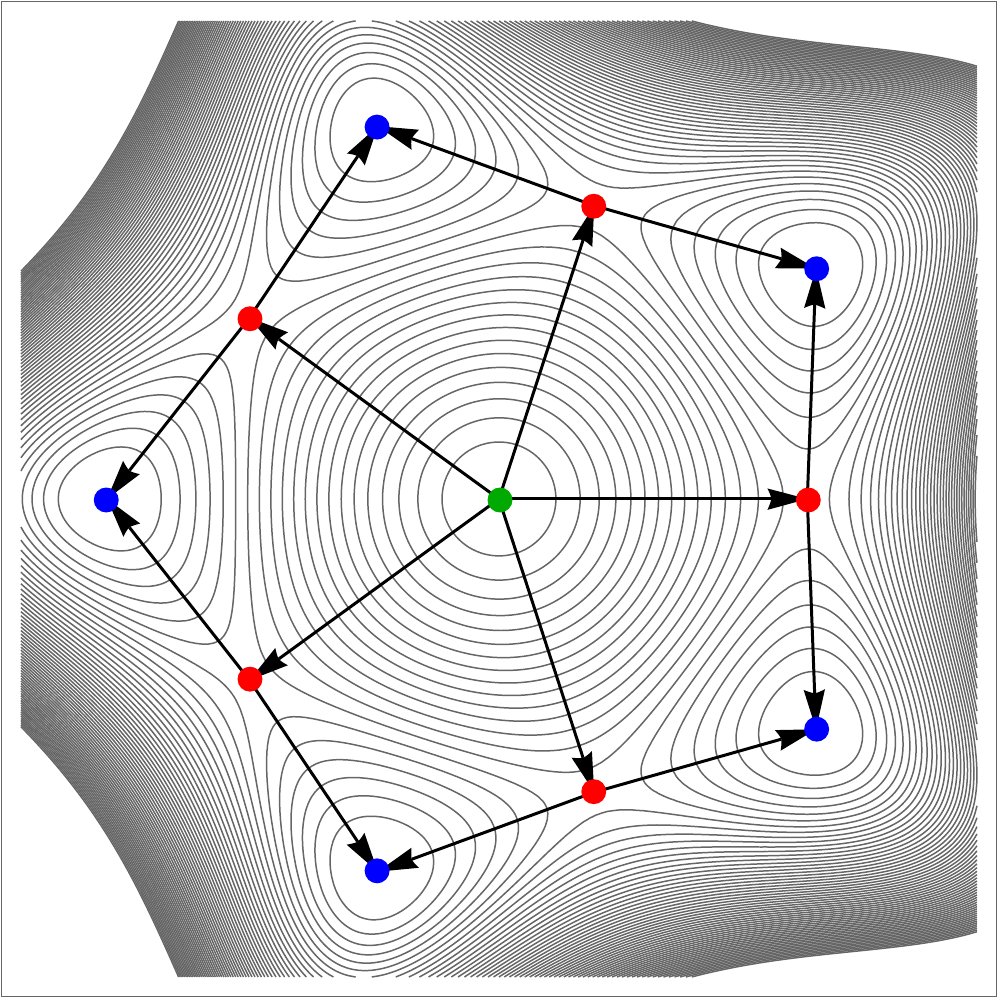}
    \caption{Repelling, $\eps>0$}
  \end{subfigure}
  \caption{Level curves of $G_\eps$ and gradient flows of 5-emission.}
  \label{Figure - Emission}
\end{figure}

The final case is $k=4$.
The leading terms of the unwrapped generating Hamiltonian are
    \[
    G_\eps(\rho,\vp) =  \left(\frac{\alpha}{4}\cos 4\vp + \frac{\beta}{4}\right)\rho^4 + \eps\frac{a}{2}\rho^2 + \ldots.
    \]
In this case, the absolute values of $\alpha$ and $\beta$ determine the bifurcation structure.
If $|\alpha|>|\beta|$, the $\rho^4 \cos 4\vp$ term is dominant and the phantom kiss occurs.
If $|\alpha|<|\beta|$, the $\rho^4$ term is dominant and the emission occurs.
One can see this explicitly by solving the $\rho$-part of $\d G_\eps=0$,
    \[
    \rho^2 = -\frac{\eps a}{\alpha \cos 4\vp + \beta},
    \]
which implies that the allowed sign of $\eps$ is different for the 4 solutions depending on the absolute values of $\alpha$ and $\beta$.
The remaining discussion is analogous to the previous arguments.

\subsection{Summary and Morse Trajectories}

\begin{table}[htbp]
	\renewcommand{\arraystretch}{1.5} 
	\centering
	\small
	\begin{tabular}{|c|c|c|c|c|c|c|c|c|}
		\hline
		\multicolumn{3}{|c|}{Generic Bifurcations}
		& \multicolumn{3}{c|}{$\eps<0$}
		& \multicolumn{3}{c|}{$\eps>0$} \\
		\hline
		Floquet Mult.\ & Name & Type
		& Orbits & Period & Index
		& Orbits & Period & Index \\
		\hline
		
		\multirow{4}{*}{$1$}
		& \multirow{4}{*}{Birth-death}
		& \multirow{2}{*}{Source}
		&  &  &  & $\g_H$ & 1 & 1 \\
		&  &  &  &  &  & $\g_E$ & 1 & 0 \\
		\cline{3-9}
		&  & \multirow{2}{*}{Sink}
		&  &  &  & $\g_H$ & 1 & 1 \\
		&  &  &  &  &  & $\g_E$ & 1 & 2 \\
		\hline
		
		\multirow{8}{*}{$-1$}
		& \multirow{4}{*}{\shortstack{Period\\doubling}}
		& \multirow{2}{*}{Attracting}
		& $\g^2$ & 1 & 0 & $\g^2$ & 1 & 1 \\
		&  &  &  &  &  & $\g_E$ & 2 & 0 \\
		\cline{3-9}
		&  & \multirow{2}{*}{Repelling}
		& $\g^2$ & 1 & 2 & $\g^2$ & 1 & 1 \\
		&  &  &  &  &  & $\g_E$ & 2 & 2 \\
		\cline{2-9}
		
		& \multirow{4}{*}{Materialization}
		& \multirow{2}{*}{Attracting}
		& $\g^2$ & 1 & 1 & $\g^2$ & 1 & 2 \\
		&  &  &  &  &  & $\g_H$ & 2 & 1 \\
		\cline{3-9}
		&  & \multirow{2}{*}{Repelling}
		& $\g^2$ & 1 & 1 & $\g^2$ & 1 & 0 \\
		&  &  &  &  &  & $\g_H$ & 2 & 1 \\
		\hline
		
		\multirow{2}{*}{\shortstack{$\exp(2\pi i l/k)$\\$k=3,4$, $(l,k)=1$}}
		& \multirow{2}{*}{Phantom Kiss}
		& \multirow{2}{*}{$-$}
		& $\g^k$ & 1 & 0 & $\g^k$ & 1 & 2 \\
		&  &  & $\g_H^-$ & $k$ & 1 & $\g_H^+$ & $k$ & 1 \\
		\hline
		
		\multirow{6}{*}{\shortstack{$\exp(2\pi i l/k)$\\$k\geq 4$, $(l,k)=1$}}
		& \multirow{6}{*}{Emission}
		& \multirow{3}{*}{Attracting}
		& $\g^k$ & 1 & 0 & $\g^k$ & 1 & 2 \\
		&  &  &  &  &  & $\g_H$ & $k$ & 1 \\
		&  &  &  &  &  & $\g_E$ & $k$ & 0 \\
		\cline{3-9}
		&  & \multirow{3}{*}{Repelling}
		& $\g^k$ & 1 & 2 & $\g^k$ & 1 & 0 \\
		&  &  &  &  &  & $\g_H$ & $k$ & 1 \\
		&  &  &  &  &  & $\g_E$ & $k$ & 2 \\
		\hline
		
	\end{tabular}
	\caption{Classification of generic bifurcations.}
	\label{Table - Bifurcation}
\end{table}

\Cref{Table - Bifurcation} is a summary of the generic bifurcation phenomena of Hamiltonian systems.
In this table, the period means the period of the periodic points of the return map on the Poincar\'{e} section; in general, there must be a defect in the return time, so the period of the periodic orbit corresponding to the $k$-periodic points is not exactly $k$, though it converges to $k$.
The index means the Morse index of the periodic points as critical points of the unwrapped generating Hamiltonian $G_\eps$ of $\Psi^k_\eps$, which is closely related to the Conley-Zehnder indices of the corresponding periodic orbit.
As mentioned in the classification procedure, we split each case (except the phantom kiss) into two types depending on the Morse index, which will result in different Floer chain complexes in the later sections.

Since the conditions for a periodic point to be an extremal, transitional, or $k$-bifurcation point are open, we can expect that these are generic conditions.
Meyer verified this property in his 1970 paper.

\begin{thm}[\cite{Meyer_70}]\label{Theorem - Generic bifurcation}
    Let $\mathcal{G}$ be the set of 1-parameter families of symplectomorphisms on a surface $\Sigma$ such that every periodic point is either an elliptic, hyperbolic, extremal, transitional, or $k$-bifurcation point.
    Then $\mathcal{G}$ is a residual set in the space of all 1-parameter families of $\Sigma$.
\end{thm}
This theorem establishes the classification of generic bifurcations of periodic points, which can also be applied to the case of periodic orbits.

Moreover, we discuss about the positive gradient flowlines which connect the critical points of $G_\eps$.
The structure is described in Figures \ref{Figure - birth-death}, \ref{Figure - Period doubling}, \ref{Figure - Materialization}, \ref{Figure - Phantom kiss} and \ref{Figure - Emission} with the level sets of $G_\eps$, and we prove that these are all such flowlines.
\begin{prop}[Structure of Gradient Trajectories]\label{Proposition - Morse Trajectories}
    For every sufficiently small $\eps\neq 0$, positive gradient flowlines connecting critical points of $G_\eps$ with Morse index difference 1 are precisely those depicted in Figures \ref{Figure - birth-death}, \ref{Figure - Period doubling}, \ref{Figure - Materialization}, \ref{Figure - Phantom kiss} and \ref{Figure - Emission}. In particular, the gradient flow of $G_\eps$ is Morse--Smale.
\end{prop}
\begin{proof}
Throughout, let $\eta$ denote a positive gradient trajectory $\dot\eta=\nabla G_\eps(\eta)$, along which $G_\eps$ is strictly increasing.
    Since $G_\eps$ is a strict Lyapunov function, any trajectory that remains in the isolating neighborhood $U$ for all forward (resp. backward) time has its $\omega$- (resp. $\alpha$-)limit equal to a critical point: the limit set is then nonempty, compact, connected, flow-invariant, and contained in a level set of $G_\eps$, hence a single critical point.
    A trajectory that is not trapped instead exits $U$ in finite time, and we say its limit within $U$ does not exist.
    Near each bifurcation the nonzero critical points occur at radius $\rho_*=O(\sqrt{\eps})$.
    After rescaling, the higher-order terms of $G_\eps$ are $O(\sqrt\eps)$ relative to the leading terms, so the connecting trajectories are transverse intersections of (un)stable manifolds and persist under this $C^1$-small correction.
    It therefore suffices to analyze the leading terms of $G_\eps$.

    \emph{Birth-death and period doubling} (\Cref{Figure - birth-death,Figure - Period doubling,Figure - Materialization}).
    By \Cref{Subsection - Birth-death}, the leading terms give
    \[
    \nabla G_\eps = \begin{pmatrix} \alpha x+\eps a\\ \beta y^2+\eps b \end{pmatrix},\qquad
    \Hess G_\eps(p_{H,E}) = \begin{pmatrix}  \alpha & 0\\ 0 & \mp 2\beta\sqrt{-\eps b/\beta} \end{pmatrix}.
    \]
    We treat the sink case, where $p_E$ is a maximum, so $\Ind(p_H)=1$ and $\Ind(p_E)=2$.
    The unstable manifold $W^u(p_H)$ is one-dimensional, tangent to the $y$-axis, and has two branches.
    By the leading sign of $\pp_y G_\eps=\beta y^2+\eps b$, the branch directed toward $p_E$ converges to it, while the other leaves the isolating neighborhood $U$ without meeting a critical point.
    Hence exactly one branch of $W^u(p_H)$ converges to $p_E$, giving a unique connecting trajectory.
    The source case $\Ind(p_E)=0$ is symmetric, with the stable manifold of $p_H$ in place of the unstable one.
    The period doubling and materialization cases are treated analogously.

    \emph{Phantom kiss and emission} (\Cref{Figure - Phantom kiss,Figure - Emission}).
    In polar coordinates the leading terms satisfy
    \[
    \pp_\vp G_\eps = \alpha\rho^{k-1}\sin k\vp,
    \]
    which vanishes on the radial rays $\vp = m\pi/k$.
    The gradient field is therefore tangent to these rays, so each ray is a one-dimensional invariant submanifold, and no trajectory can cross it; every trajectory is confined to a single angular sector $\{m\pi/k\le\vp\le(m+1)\pi/k\}$, which contains exactly one nonzero critical point.

    For the phantom kiss, the Hessian at the saddles $p^j_\mp$ is
    \[
    \Hess G_\eps(p^j_\mp) = \begin{pmatrix} -\eps a & 0\\ 0 & \mp 2\alpha\rho^3 \end{pmatrix},
    \]
    so $W^s(p^j_-)$ and $W^u(p^j_+)$ are one-dimensional and radial.
    Along each invariant ray containing a saddle, the radial equation is one-dimensional, and its sign is constant between the origin and the saddle. Hence there is a unique connecting trajectory between these two critical points up to translation. This verifies \Cref{Figure - Phantom kiss}.
    Note that there is no connection between two saddles because of the invariance of the ray; the relevant one-dimensional branch of the stable or unstable manifold is contained in an invariant ray.

    For the emission, the connections between $p_H^j$ and the origin are radial, exactly as above.
    In the attracting case the relevant manifolds $W^s(p_H^j)$ are instead tangent to $\pp_\vp$.
    By the angular confinement above, each of the two angular branches stays in its sector and converges to the unique elliptic point $p_E^j$ therein.
    Thus the two branches of $W^s(p_H^j)$ converge to the two adjacent elliptic points, verifying \Cref{Figure - Emission} and completing the proof.

    Finally, the positive gradient flow is Morse--Smale. Indeed, on a surface, every intersection involving an extremum is automatically transverse, since either the stable or the unstable manifold involved is two-dimensional. Thus the only possible failure of the Morse--Smale condition would be a heteroclinic trajectory between two distinct index-one saddles. The invariant-ray and sector arguments above exclude all such saddle connections.
\end{proof}

Here's a simple lemma which will be used in the following sections.
\begin{lmm}\label{Lemma - C2small}
    The unwrapped generating Hamiltonians $G_\eps$ for the birth-death, period doubling, materialization, phantom kiss and emission are $C^2$-small if we take a Poincar\'{e} section and the bifurcation parameter small enough. Precisely, for every $\delta>0$, after a linear symplectic change of coordinates, there exist a sufficiently small Poincar\'{e} section $\Sigma$ and $\eps_0>0$ such that $\|G_\eps\|_{C^2}<\delta$
for every $|\eps|<\eps_0$.
\end{lmm}
\begin{proof}
We compute the $C^2$-norm with respect to the Euclidean metric in the chosen Darboux coordinates. After adding an $\eps$-dependent constant, we may assume that $G_\eps(0)=0$, and in all cases $dG_\eps(0)=O(\eps)$.

For the phantom kiss and emission, the quadratic part of $G_0$ vanishes. Consequently,
\[
\sup_{z\in \Sigma}|\Hess G_\eps(z)|\to  0
\]
as both $|\eps|$ and the diameter of $\Sigma$ tend to zero. Taylor's theorem then gives the conclusion.

For the birth-death, the only non-small term is the shear term $\alpha x^2/2$. For $c>0$, consider the linear symplectic transformation $\Phi_c(x,y)=(cx,y/c)$. Then,
\[
G_\eps\circ\Phi_c = \frac{c^2\alpha}{2}x^2 +\frac{\beta}{3c^3}y^3 +\eps\left(acx+\frac bc y\right) +\text{(higher-order terms)}.
\]
Given $\delta>0$, we first choose $c$ so that $c^2|\alpha|<\delta/3$. With this $c$ fixed, we shrink $\Sigma$ so that the Hessians of the cubic and higher-order terms are bounded by $\delta/3$. Finally, we decrease $\eps_0$ so that the parameter-dependent terms, as well as $dG_\eps(0)$, are bounded by $\delta/3$ for $|\eps|<\eps_0$. This proves the claim.

The period-doubling and materialization are treated by the same symplectic rescaling.
\end{proof}

\subsection{Relative Conley--Zehnder Indices}\label{Subsection - Relative CZ index}

In the preceding subsections, we determined the Morse indices of the periodic points as critical points of the unwrapped generating Hamiltonian $G_\eps$, which are summarized as the Index column in \Cref{Table - Bifurcation}.
To utilize these results in Floer theory, we must translate these Morse indices into the Conley--Zehnder indices of the corresponding periodic orbits.

Recall from \Cref{Subsection - CZ index} that the absolute Conley--Zehnder index is framing-dependent, but the relative index between two orbits sharing a common tubular neighborhood is canonical.
In our 4-dimensional setting, the Poincar\'{e} section is 2-dimensional.
Let $p$ be a critical point of $G_\eps$ and $\g_p$ be the corresponding periodic orbit. With respect to the canonical framing of the local tubular neighborhood, the Conley--Zehnder index $\mu_{CZ}(\g_p)$ is related to the Morse index $\Ind(p)$ of $G_\eps$ by the formula
    \[
    \mu_{CZ}(\g_p) = 1 - \Ind(p),
    \]
where $1$ comes from half the dimension of the section. 

Consequently, for any two bifurcating orbits $\g_{p}$ and $\g_{q}$ generated by the same bifurcation parameter $\eps$, their relative Conley--Zehnder index is given by
    \[
    \mu_{CZ}(\g_{p}, \g_{q}) = \mu_{CZ}(\g_{p}) - \mu_{CZ}(\g_{q}) = \Ind(q) - \Ind(p).
    \]

This simple algebraic relation is the key reason why the classification of generic bifurcations directly dictates the structure of the local Floer chain complex.
By reading the Morse indices off \Cref{Table - Bifurcation}, we can immediately deduce the relative gradings of the generators in our local Floer chain complex, which we will formally construct in the subsequent sections.

\section{Classification of the Bifurcations of Involutive Systems}\label{Section - Symmetric}
While a complete classification of generic bifurcations in unconstrained Hamiltonian systems was established in \Cref{Section - Classification of bifurcations}, many physically significant systems inherit additional geometric symmetries that fundamentally alter the admissible bifurcation patterns. For instance, classic models in celestial mechanics, such as the circular restricted three-body problem and Hill's lunar problem, naturally possess an underlying $\Z_2$-symmetry. The direct and retrograde periodic orbits, which frequently serve as the dynamical backbone of these systems, emerge precisely as symmetric orbits under this involution.

This geometric restriction enforces structural constraints on the corresponding unwrapped generating Hamiltonian, forcing specific coefficients to vanish. Consequently, the system cannot be treated as fully generic in the asymmetric sense; instead, it gives rise to new generic transition behaviors within the symmetric regime, most notably the pitchfork bifurcation, which would be non-generic in the absence of symmetry. While bifurcations of symmetric systems have been classically classified in foundational works such as \cite{Golubitsky_Stewart_Marsden_87} or \cite{Golubitsky_Stewart_Schaeffer_88}, we provide a refined classification tailored to our homological framework. Accordingly, in this section, we systematically classify the generic bifurcations of $\Z_2$-symmetric Hamiltonian systems by explicitly tracking the Morse indices of the unwrapped generating Hamiltonians, which allows us to compute the associated local Floer chain complexes in \Cref{Section - LFCC}.

\subsection{Involutive Systems}\label{Subsection - UGH of symmetric systems}
Let $(W,\o)$ be a 4-dimensional symplectic manifold and $H$ be a Hamiltonian. 
We say a diffeomorphism $\imath:W\to W$ is a \textbf{symplectic involution} if
\[
\imath^*\o=\o,\quad \imath^2=\Id_W,\quad \imath\neq \Id_W,
\]
and an \textbf{anti-symplectic involution} if
\[
\imath^*\o=-\o,\quad \imath^2=\Id_W.
\]
If $\imath^*H=H$, we say the Hamiltonian system is \textbf{involutive}.
We can understand this as a Hamiltonian system with $\Z_2$-symmetry, generated by $\imath$.

Let $\g$ be a periodic orbit of an involutive system $(W,\o,H,\imath)$.
There are two possibilities:
\begin{itemize}
    \item $\imath^*\g=\g$, i.e., $\g$ is $\Z_2$-symmetric.
    \item $\imath^*\g\neq \g$. In this case, $(\g,\imath^*\g)$ is a symmetric pair.
\end{itemize}
We are interested in the bifurcation of a symmetric orbit (the first case). 

Let $p\in \g$ and assume that $\imath(p)=p$.
Since $\imath^2=\Id_W$, the eigenvalues of $d_p\imath$ are $\pm1$ with the same multiplicity, and we have an eigenspace decomposition of the tangent space
\[
T_p W = E_1\oplus E_{-1}
\]
where $E_{\pm1}$ are 2-dimensional eigenspaces corresponding to the eigenvalues $\pm1$.
We consider the Poincar\'{e} section $\Sigma$ of $\g$ at $p$, which can be locally identified with a subspace $T_p\Sigma < T_p W$.
\begin{enumerate}
    \item (Symplectic involution) If $\imath$ is a symplectic involution, $\Sigma$ can be identified with either $E_1$ or $E_{-1}$. However, since $\imath_*X_H=X_H$, we have $X_H\in E_1$, and so is its symplectic conjugate. It follows that $\Sigma$ must be identified with $E_{-1}$, which imposes an extra $\Z_2$-symmetry on the return map $\Psi$ corresponding to a rotation by $\pi$.
\item (Anti-symplectic involution) If $\imath$ is an anti-symplectic involution,
then both $E_{\pm 1}$ are Lagrangian subspaces. Therefore, $T_p\Sigma$ again
admits an eigenspace decomposition $S_1\oplus S_{-1}$. In suitable symplectic
coordinates on $\Sigma$, the restriction $\rho:=\imath|_\Sigma$ is a reflection
across one of these eigen-directions. Since $\imath^*H=H$ and $\imath$ is
anti-symplectic, it reverses the Hamiltonian flow, and hence the Poincare
return map satisfies
\[
        \rho\circ \Psi\circ \rho=\Psi^{-1}.
\]
Thus the return map has a $\mathbb Z_2$-reversing symmetry, instead of an
ordinary $\mathbb Z_2$-equivariance.
\end{enumerate}

Now we derive the normal form of the unwrapped generating Hamiltonians for involutive systems.
Let an involution $\imath:W\to W$ and a 1-parameter family of Hamiltonians $H_\eps$ satisfying $\imath^*H_\eps=H_\eps$ be given.
Assume that $\g_0$ is an $\imath$-symmetric periodic orbit of $H_0$ with Floquet multiplier $\ld=\exp(2\pi i l/k)$ for some $(l,k)=1$, so that its $k$-th cover is degenerate, leading to a bifurcation.

We first consider the case of a symplectic involution.
Since the return map must be symmetric with respect to a rotation by $\pi$, so must the unwrapped generating Hamiltonian.
For the case $\ld=\exp(2\pi i l/k)$, the symmetry of $G_\eps$ is given by $\Z_k=\il R_{2\pi i l/k}\ir$, where $R_\alpha$ denotes a rotation by $\alpha$ around the origin.
We have that
\[
\il R_{2\pi il/k},R_\pi\ir = \begin{cases}
    \il R_{2\pi il/k}\ir\simeq \Z_k &\text{ if }$k$\text{ is even},\\
    \il R_{\pi i l /k}\ir \simeq \Z_{2k} &\text{ if }$k$\text{ is odd}.
\end{cases}
\]
Thus, we obtain the following proposition.
\begin{prop}\label{Proposition - Symplectic involution GH}
    Let $(W,\o,H_\eps,\imath)$ be a 1-parameter family of symplectic involutive systems, and let $\g_0$ be an $\imath$-symmetric orbit with a point $\g_0(t_0)=p$ such that $\imath(p)=p$.
    Moreover, assume that the Floquet multiplier of $\g_0$ is given by $\ld=\exp(2\pi il/k)$.
    Then, the unwrapped generating Hamiltonian $G_\eps$, which generates the $k$-th iterate of the return map $\Psi^k$, is given by
    \[
    G_\eps = \begin{cases}
        {\alpha}x^2 + \sum_{n\geq 3}P_n(x,y)\in \R_{\Z_2}[x,y]=\R[x^2,xy,y^2]&\text{ if }k=1\text{ or }2,\\
        \sum_{n\geq 3}P_n(\rho,\vp)\in \R_{\Z_{2k}}[\rho,\vp]=\R[\rho^2,\rho^{2k}\cos(2k\vp),\rho^{2k}\sin(2k\vp)]&\text{ if }k\geq 3\text{ is odd},\\
        \sum_{n\geq 3}P_n(\rho,\vp)\in \R_{\Z_k}[\rho,\vp]=\R[\rho^2,\rho^k\cos(k\vp),\rho^k\sin(k\vp)]&\text{ if }k\geq4\text{ is even}.
    \end{cases}
    \]
    Here, $\R_G[x,y]$ denotes the ring of $G$-symmetric polynomials.
\end{prop}
In particular, we directly see that the generic bifurcation behaviors of even covers are the same as in the case without any symmetry. This is because the $\Z_2$-symmetry of the system is already implicitly imposed by the time symmetry.

Now we consider systems with an anti-symplectic involution. As mentioned above, in the anti-symplectic case, the relevant normal forms are reversible rather than equivariant. The generating Hamiltonian of a reversible symplectic map can be chosen to be invariant under the reversing reflection. Indeed, if $G_\eps\circ\rho=G_\eps$ and $\rho$ is anti-symplectic, then the Hamiltonian flow $Fl_t$ of $G_\eps$ satisfies
\[
        \rho\circ Fl_t\circ \rho
        =
        Fl_{-t}.
\]
For $k=1$ or $2$, we distinguish between the cases where $\rho(x)=x$ or
$\rho(x)=-x$.
\begin{enumerate}
    \item ($\rho(x)=x$) In this case, $\rho(y)=-y$, so $G_\eps$ can only contain
    even-degree terms in $y$.
    \item ($\rho(x)=-x$) Similarly, $\rho(y)=y$, so $G_\eps$ can only contain
    even-degree terms in $x$.
\end{enumerate}
For $k\geq 3$, it is more convenient to use complex coordinates.
The reflection can be represented by $z\mapsto \bar{z}$ after an appropriate rotation, and $G_\eps$ must be invariant under this action.
The $\Z_k$-symmetric polynomials in complex coordinates are generated by
$|z|^2$, $\mathrm{Re}(z^k)$, and $\mathrm{Im}(z^k)$,
and since $\mathrm{Im}(\bar{z}^k)=-\mathrm{Im}(z^k)$, only $|z|^2=\rho^2$ and $\mathrm{Re}(z^k)=\rho^k\cos(k\vp)$ survive as generators. In particular, these are $D_k$-symmetric polynomials. To summarize, we have the following proposition.
\begin{prop}\label{Proposition - Anti-symplectic involution GH}
    Let $(W,\o,H_\eps,\imath)$ be a 1-parameter family of anti-symplectic involutive systems, and let $\g_0$ be an $\imath$-symmetric orbit with a point $\g_0(t_0)=p$ such that $\imath(p)=p$.
    Moreover, assume that the Floquet multiplier of $\g_0$ is given by $\ld=\exp(2\pi il/k)$.
    Then, the unwrapped generating Hamiltonian $G_\eps$, which generates the $k$-th iterate of the return map $\Psi^k$, is given by
    \[
    G_\eps = \begin{cases}
        {\alpha}x^2 + \sum_{n\geq 3}P_n(x,y)\in \R[x^2,y]\text{ or }\R[x,y^2]&\text{ if }k=1,\\
        {\alpha}x^2 + \sum_{n\geq 3}P_n(x,y)\in \R[x^2,y^2]&\text{ if }k=2,\\
        \sum_{n\geq 3}P_n(\rho,\vp)\in \R[\rho^2,\rho^{k}\cos(k\vp)]&\text{ if }k\geq 3.
    \end{cases}
    \]
\end{prop}
Note that the bifurcation behavior for $k\geq 2$ essentially does not change in this case, considering that we can take $\rho^k\cos(k\vp)$ as the leading term of $G_\eps$ after an appropriate rotation.

\begin{remark}\rm
The actual effect of the anti-symplectic involution on the multiple cover bifurcation is that the positions of the critical points of $G_\eps$ on the Poincar\'{e} section become highly rigid; they are fixed on the rays $\vp=n\pi/k$, not only for the leading term, but for every term in the reflection-invariant normal form. This property significantly reduces the effort required to locate periodic orbits in practice. However, a detailed analysis of this phenomenon is beyond the scope of this paper.
\end{remark}

\subsection{Pitchfork}\label{Subsection - Pitchfork}
As mentioned in \Cref{Subsection - UGH of symmetric systems}, we only need to examine the odd covers for a symplectic involution, and the simple orbit for an anti-symplectic involution to observe a new phenomenon.
We begin with the single cover case.
Let $(W,\o,H_\eps,\imath)$ be a 1-parameter family of symplectic involutive systems, and let $\g_0$ be a periodic orbit of $H_0$ with Floquet multiplier $1$.
From \Cref{Proposition - Symplectic involution GH}, we have the unwrapped generating Hamiltonian $G_\eps$ with leading terms
\[
G_\eps = \frac{\alpha}{2}x^2 + \frac{\beta}{4}y^4 + \eps \frac{a}{2}y^2.
\]
Assuming $0$ is a transitional fixed point of $G_\eps$, this case is exactly the same as the period doubling or materialization cases described in Figures \ref{Figure - Period doubling} and \ref{Figure - Materialization}, including the changes in stability types and Morse indices.
The difference appears when we relate the critical points back to the periodic orbits.
Now, $G_\eps$ generates $\Psi_\eps$ instead of its double cover, so one critical point corresponds to one periodic orbit with an approximate period equal to that of $\g_0$.
Therefore, we have the following generic phenomenon.

\begin{thm}[Pitchfork, Symplectic]
    \label{Theorem - Symplectic pitchfork}
    Let $(W,\o,H_\eps,\imath)$ be a 1-parameter family of symplectic involutive Hamiltonian systems, and let $\g$ be a periodic orbit of $H_0$ with Floquet multiplier $\ld=1$.
    Assume that $\g$ intersects the fixed locus of $\imath$, and the intersection is a transitional fixed point of the return map $\Psi_\eps$.
    Then, there exists an orbit cylinder $\g_\eps$ centered at $\g$, and depending on an appropriate choice of the bifurcation parameter, exactly one of the following happens:
    \begin{enumerate}
        \item $\g_\eps$ is elliptic if $\eps<0$ and hyperbolic if $\eps>0$.
              There exist two families of elliptic periodic orbits $\g_{1,2}(\eps)$ of $H_\eps$, both parametrized by $\eps>0$, which converge to $\g$ as $\eps\to0$.
        \item $\g_\eps$ is hyperbolic if $\eps<0$ and elliptic if $\eps>0$.
              There exist two families of hyperbolic periodic orbits $\g_{1,2}(\eps)$ of $H_\eps$, both parametrized by $\eps>0$, which converge to $\g$ as $\eps\to0$.
    \end{enumerate}
\end{thm}
We call this phenomenon a \textbf{pitchfork bifurcation}, and there are $4$ possible changes in the Morse indices.
The first case of \Cref{Theorem - Symplectic pitchfork} is called \textbf{subcritical}, which corresponds to the period doubling, and the second case is called \textbf{supercritical}, corresponds to the materialization. We also distinguish between the \textbf{attracting} and \textbf{repelling} pitchforks, depending on whether the positive gradient flowlines are directed toward or away from the origin.

Now we consider anti-symplectic involutive systems.
According to \Cref{Proposition - Anti-symplectic involution GH}, we have two possible forms for the generating Hamiltonian of $\Psi_\eps$, which are
\[
\begin{aligned}
        G_\eps(x,y) &= \frac{\alpha}{2}x^2 + \frac{\beta}{3}y^3 + \eps \left(\frac{a}{2}x^2 + by\right)\in \R[x^2,y],\\
        G_\eps(x,y) &=\frac{\alpha}{2}x^2 + \frac{\beta}{4}y^4 + \eps \left(ax + \frac{b}{2}y^2\right)\in \R[x,y^2].
\end{aligned}	
\]
Note that the first case does not change the bifurcation significantly, since the additional effect is a change in the coefficient of $\alpha$ by $O(\eps)$, which cannot alter the sign for small $\eps$.
The second case is again essentially the same as period doubling, which leads to the pitchfork bifurcation.

\begin{thm}[Pitchfork, Anti-symplectic]
    Let $(W,\o,H_\eps,\imath)$ be a 1-parameter family of anti-symplectic involutive Hamiltonian systems, and let $\g$ be a periodic orbit of $H_0$ with Floquet multiplier $\ld=1$.
    Assume that $\g$ intersects the fixed locus of $\imath$, and the intersection is an extremal fixed point or a transitional fixed point of the return map $\Psi_\eps$.
    Then, exactly one of the following happens, depending on the shear direction of the return map and the fixed locus of $\imath$:
    \begin{enumerate}
        \item A pitchfork bifurcation occurs at $\g_0$.
        \item A birth-death bifurcation occurs at $\g_0$.
    \end{enumerate}
\end{thm}

The level curves and gradient trajectory structures for the pitchfork are the same as one described in Figures \ref{Figure - Period doubling} and \ref{Figure - Materialization}.

\subsection{Double Emission}\label{Subsection - Double Emission}
Now we consider the case of a symplectic involutive system with a periodic orbit $\g_0$ having a Floquet multiplier $\ld=\exp(2\pi il/k)$ where $k\geq 3$ is odd.
The leading terms of the unwrapped generating Hamiltonian for $\Psi^k_\eps$ are given by
\[
G_\eps = \frac{\alpha}{2k}\rho^{2k}\cos(2k\vp) + \frac{\beta}{4}\rho^4 + \eps \frac{a}{2}\rho^2.
\]
We assume that the origin is a $2k$-bifurcation point.
This yields the same configuration of critical points as the $2k$-emission, since $2k\geq 6$.
The critical points are given by
\[
\rho^2 = -\frac{\eps a}{\beta},\quad \vp_j = \frac{j\pi}{2k}\text{ for }j=0,1,\ldots,4k-1.
\]
There are $2k$ saddles and $2k$ minima or maxima, depending on the signs of the coefficients.
The return map $\Psi_\eps$ identifies critical points with an angular difference of $2\pi/k$, which groups a $k$-tuple of critical points into a single periodic orbit.
Therefore, we have $2$ hyperbolic periodic orbits and $2$ elliptic periodic orbits emerging at $\eps=0$.
To summarize, we have the following.

\begin{thm}[Double Emission]
    Let $(W,\o,H_\eps,\imath)$ be a 1-parameter family of symplectic involutive Hamiltonian systems, and let $\g$ be a periodic orbit of $H_0$ with Floquet multiplier $\ld=\exp(2\pi i l/k)$ where $(k,l)=1$ and $k\geq 3$ is odd.
    Assume that $\g$ intersects the fixed locus of $\imath$, and the intersection is a $2k$-bifurcation point of the return map $\Psi^k_\eps$.

    Then, with an appropriate choice of the bifurcation parameter, there exist two families of hyperbolic periodic orbits $\g_{H,i}(\eps)$ and two families of elliptic periodic orbits $\g_{E,i}(\eps)$ of $H_\eps$ where $i=1,2$, all parametrized by $\eps>0$, which converge to $\g^k$ as $\eps\to0$.
\end{thm}
We call this a \textbf{double $k$-emission}.
As in \Cref{Subsection - k bifurcation}, there are two possible changes in the Morse index: if the origin is a minimum for $\eps<0$, it becomes a maximum for $\eps>0$, and the intersections of $\g_{E,i}$ are all minima, and vice versa.
We call the former \textbf{attracting} and the latter \textbf{repelling}.
The level sets and gradient-flow structure of the double emission are the same as those depicted in \Cref{Figure - Emission}. The difference is that, in the emission bifurcation, all elliptic critical points correspond to a single periodic orbit, as do all hyperbolic critical points, whereas in the double-emission bifurcation, adjacent elliptic critical points belong to distinct periodic orbits.

\subsection{Summary of $\Z_2$-symmetric Bifurcations}
\begin{table}[htbp]
	\renewcommand{\arraystretch}{1.5} 
	\centering
	\small
	\begin{tabular}{|c|c|c|c|c|c|c|c|c|}
		\hline
		\multicolumn{3}{|c|}{Generic Bifurcations}
		& \multicolumn{3}{c|}{$\eps<0$}
		& \multicolumn{3}{c|}{$\eps>0$} \\
		\hline
		Floquet Mult.\ & Name & Type
		& Orbits & Period & Index
		& Orbits & Period & Index \\
		\hline
		
		\multirow{8}{*}{$1$}
		& \multirow{4}{*}{\shortstack{Subcritical\\Pitchfork}}
		& \multirow{2}{*}{Attracting}
		& $\g$ & 1 & 0 & $\g$ & 1 & 1 \\
		&  &  &  &  &  & $\g_1,\g_2$ & 1 & 0 \\
		\cline{3-9}
		&  & \multirow{2}{*}{Repelling}
			& $\g$ & 1 & 2 & $\g$ & 1 & 1 \\
		&  &  &  &  &  & $\g_1,\g_2$ & 1 & 2 \\
		\cline{2-9}
		
		& \multirow{4}{*}{\shortstack{Supercritical\\Pitchfork}}
		& \multirow{2}{*}{Attracting}
		& $\g$ & 1 & 1 & $\g$ & 1 & 2 \\
		&  &  &  &  &  & $\g_1,\g_2$ & 1 & 1 \\
		\cline{3-9}
		&  & \multirow{2}{*}{Repelling}
			& $\g$ & 1 & 1 & $\g$ & 1 & 0 \\
		&  &  &  &  &  & $\g_1,\g_2$ & 1 & 1 \\
		\cline{2-9}
		\hline

		\multirow{6}{*}{\shortstack{$\exp(2\pi i l/k)$\\$k\geq 3$ odd,\\$(l,k)=1$}}
		& \multirow{6}{*}{Double Emission}
		& \multirow{3}{*}{Attracting}
		& $\g^k$ & 1 & 0 & $\g^k$ & 1 & 2 \\
		&  &  &  &  &  & $\g_{H,1},\g_{H,2}$ & $k$ & 1 \\
		&  &  &  &  &  & $\g_{E,1},\g_{E,2}$ & $k$ & 0 \\
		\cline{3-9}
		&  & \multirow{3}{*}{Repelling}
		& $\g^k$ & 1 & 2 & $\g^k$ & 1 & 0 \\
		&  &  &  &  &  & $\g_{H,1},\g_{H,2}$ & $k$ & 1 \\
		&  &  &  &  &  & $\g_{E,1},\g_{E,2}$ & $k$ & 2 \\
		\hline
		
	\end{tabular}
	\caption{Classification of generic bifurcations of involutive systems.}
	\label{Table - Symmetric Bifurcation}
\end{table}

Here, we summarize the classification of symmetric bifurcations as shown in \Cref{Table - Symmetric Bifurcation}.
Note that the pitchfork bifurcation appears for both symplectic and anti-symplectic involutive systems, while the double emission appears only for symplectic involutive systems.
Moreover, for the even covers, the classification follows the general case in \Cref{Table - Bifurcation}.
Since the same generating Hamiltonian is essentially used, the proof for the resulting gradient flow structures and Morse--Smale property is identical to that in \Cref{Proposition - Morse Trajectories}.

As in the previous section, we note $C^2$-smallness of the unwrapped generating Hamiltonians.
\begin{lmm}\label{Lemma - C2smallsymmetric}
    The unwrapped generating Hamiltonians $G_\eps$ for the pitchfork and double emission are $C^2$-small if we take a Poincar\'{e} section and the bifurcation parameter small enough.
\end{lmm}
\begin{proof}
    The proof is analogous to \Cref{Lemma - C2small}.
\end{proof}

\begin{remark}\label{Remark - more symmetry}\rm
In principle, Hamiltonian systems possessing a general $\Z_m$-symmetry can be treated in an entirely analogous manner by imposing the corresponding geometric symmetry on the unwrapped generating Hamiltonian. For example, if one wishes to study the bifurcation of the $k$-th cover of a symmetric orbit in the H\'{e}non--Heiles system, which naturally exhibits a $\Z_3$-symmetry, it suffices to analyze the critical points of an unwrapped generating Hamiltonian composed of $\Z_{\mathrm{lcm}(3,k)}$-symmetric polynomials. 

However, as the order of the symmetry group increases, the algebraic analysis becomes more intricate. In particular, when $\gcd(m,k) \neq 1$, the interplay between the symmetry of the phase space and the multiple-cover twist introduces substantial complications.
Consequently, a comprehensive investigation of such higher-order symmetric bifurcations lies beyond the scope of the present paper and is deferred to future research.
\end{remark}

\section{Local Floer Homology}\label{Section - LFH}
	In this paper, we will use the non-degenerate autonomous Hamiltonians and the Morse--Bott chain complex introduced in \cite{Bourgeois_Oancea_09}.
	Fix a (possibly degenerate) autonomous Hamiltonian $H_0:M\rightarrow \R$ and an isolated non-constant 1-periodic orbit $\g_0:S^1\rightarrow M$. Let $U$ be a tubular neighborhood of $\g_0$ containing no other periodic orbits of $H_0$. We will now follow \cite[Section 3.2]{Ginzburg_Gurel_10} to construct the local Floer homology near $\gamma_0$.
    
	\subsection{Chain Complex}\label{Subsection - CC of LFH}
    First, we will define the generators and their index of $CF_*^{\loc}(\gamma_0,H)$. Take any non-degenerate autonomous Hamiltonian $H:M\rightarrow \R$ such that $H-H_0$ is supported in $U$ and is $C^2$-small. Define the set of unparametrized $1$-periodic orbits of $H$ lying in $U$
    \begin{equation*}
        \mathcal{P}(H|_U)=\{\zeta:S^1\rightarrow U\,|\,\partial_t\zeta=X_H(\zeta)\}/\sim,
    \end{equation*}
    where $\zeta_1\sim \zeta_2$ if and only if $\zeta_1(t)=\zeta_2(t+\tau)$ for some $\tau\in S^1$. For each $\g\in\mathcal{P}(H|_U)$, we define the set of parametrizations of $\g$ as
    \[S_\g=\{\zeta:S^1\rightarrow U\,|\,\zeta \text{ represents }\g\}.\]
    We can identify $S_\g$ with its image via the natural evaluation map
    \begin{align*}
        S_\g &\rightarrow \im \g\subset U\\
        \zeta&\mapsto \zeta(0),
    \end{align*}
    thus, each $S_\g$ is diffeomorphic to $S^1$. Next, take a generic Morse function $f_\g:S_\g\rightarrow\R$ with exactly one maximum and one minimum point for each $\g\in \mathcal{P}(H|_U)$. We then define the local Floer chain complex near $\gamma_0$ as
    \begin{equation*}
        CF_*^{\loc}(\gamma_0,H)=\bigoplus_{\g\in \mathcal{P}(H|_U)} \Z_2\il\gm, \gM\ir,
    \end{equation*}
where $\gm$ and $\gM$ correspond to the minimum and maximum points of $f_\g$. Now, let $\mathcal{T}$ be a trivialization of $TU$. We can define the Conley-Zehnder index $\mu_{CZ}^\mathcal{T}(\g)$ of $\g$ with respect to the trivialization $\mathcal{T}$. Then, we assign $\gM$ and $\gm$ the following indices:
    \begin{equation*}
            \begin{cases}
                \mu_{CZ}^\fT(\gM)=\mu_{CZ}^\fT(\g)+1\\
                \mu_{CZ}^\fT(\gm)=\mu_{CZ}^\fT(\g)
            \end{cases}
    \end{equation*}

	\subsection{Floer Cascades of Morse--Bott Type}\label{Subsection - MBCC}
Now we will define the differential $\partial$ of the chain complex. Fix a generic almost complex structure $J$. For $\g_\pm\in \mathcal{P}(H|_U)$, which might be multiple covers, we define the moduli space $\widehat{\mathcal{M}}_0(\g_-,\g_+)$ consisting of cylinders $u:\R\times S^1 \rightarrow U$ such that
    \begin{enumerate}[label=(\roman*)]
        \item $\partial_su+J(\partial_tu-X_H(u))=0$,
        \item $\lim_{s\rightarrow-\infty}u(s,t)\in S_{\g_-}$,
        \item $\lim_{s\rightarrow\infty}u(s,t)\in S_{\g_+}$.
    \end{enumerate}
    Unless $\g_-=\g_+$, there is a free $\R$-action on $\widehat{\mathcal{M}}_0(\g_-,\g_+)$ given by $s_0\cdot u(s,t)=u(s+s_0,t)$. We define the moduli space of Floer cylinders connecting $\g_-$ and $\g_+$ as
    \begin{equation*}
    \mathcal{M}_0(\g_-,\g_+)=\widehat{\mathcal{M}}_0(\g_-,\g_+)/\R.
    \end{equation*}
    Next, let $\bar{\g}_+\in \crit(f_{\g_+})$ (that is, $\bar{\g}_+$ is either $\gM^+$ or $\gm^+$) and $\bar{\g}_-\in \crit(f_{\g_-})$. We define the moduli space $\widehat{\mathcal{M}}(\bar{\g}_-,\bar{\g}_+)$ of \textbf{parametrized Floer cascades} consisting of tuples
    \begin{equation*}
        \mathbf{u}=(c_0,u_1,c_1, u_2, \cdots, u_m, c_m)
    \end{equation*}
	such that
\begin{enumerate}[label=(\roman*)]
        \item For some periodic orbits $\g_-=\g_0, \g_1, \ldots, \g_m=\g_+\in \mathcal{P}(H|_U)$, $u_i$ is a Floer cylinder connecting $\g_{i-1}$ and $\g_i$; that is, $u_i \in \mathcal{M}_0(\g_{i-1},\g_i)$.
        \item For some real numbers $r_1<r_2<\cdots<r_m$, the maps $c_i:[r_i,r_{i+1}]\rightarrow S_{\g_i}$ satisfy:
        \begin{itemize}
            \item $c_i$ is a Morse trajectory of $f_{\g_i}$; that is, $c_i'=\nabla f_{\g_i}\circ c_i$,
            \item $c_i(r_i)=\lim_{s\rightarrow\infty}u_i(s,t)$, i.e., $c_i(r_i)(0)=\lim_{s\rightarrow\infty}u_i(s,0)$,
            \item $c_i(r_{i+1})=\lim_{s\rightarrow-\infty}u_{i+1}(s,t)$, i.e., $c_i(r_{i+1})(0)=\lim_{s\rightarrow-\infty}u_{i+1}(s,0)$,
        \end{itemize}
        for each $1\le i\le m-1$.
        \item $c_0:(-\infty,r_1)\rightarrow S_{\g_-}$ satisfies:
        \begin{itemize}
            \item $c_0$ is a Morse trajectory of $f_{\g_-}$; that is, $c_0'=\nabla f_{\g_-}\circ c_0$,
            \item $c_0(r_1)=\lim_{s\rightarrow-\infty}u_1(s,t)$, i.e., $c_0(r_1)(0)=\lim_{s\rightarrow-\infty}u_1(s,0)$,
            \item $\lim_{r\rightarrow-\infty}c_0(r)=\bar{\g}_-$, i.e., $\lim_{r\rightarrow-\infty}c_0(r)(0)=\bar{\g}_-(0)$.
        \end{itemize}
        \item $c_m:(r_m,\infty)\rightarrow S_{\g_+}$ satisfies:
        \begin{itemize}
            \item $c_m$ is a Morse trajectory of $f_{\g_+}$; that is, $c_m'=\nabla f_{\g_+}\circ c_m$,
            \item $c_m(r_m)=\lim_{s\rightarrow\infty}u_m(s,t)$, i.e., $c_m(r_m)(0)=\lim_{s\rightarrow\infty}u_m(s,0)$,
            \item $\lim_{r\rightarrow\infty}c_m(r)=\bar{\g}_+$, i.e., $\lim_{r\rightarrow\infty}c_m(r)(0)=\bar{\g}_+(0)$.
        \end{itemize}
    \end{enumerate}
Unless $\bar{\g}_+=\bar{\g}_-$, there is a free $\R$-action defined on $\widehat{\mathcal{M}}(\bar{\g}_-,\bar{\g}_+)$ by shifting all $r_i$ by a constant, i.e., $r_0\cdot c_i(r)=c_i(r+r_0)$.
    We define the moduli space $\mathcal{M}(\bar{\g}_-, \bar{\g}_+)$ of \textbf{Floer cascades} as
    \begin{equation*}
        \mathcal{M}(\bar{\g}_-, \bar{\g}_+)=\widehat{\mathcal{M}}(\bar{\g}_-,\bar{\g}_+)/\R.
    \end{equation*}
An example of Floer cascade is illustrated in \Cref{Figure - Cascades}.

    \begin{figure}[htbp]
  \begin{subfigure}[b]{0.45\textwidth}
    \centering
    \includegraphics[width=\textwidth]{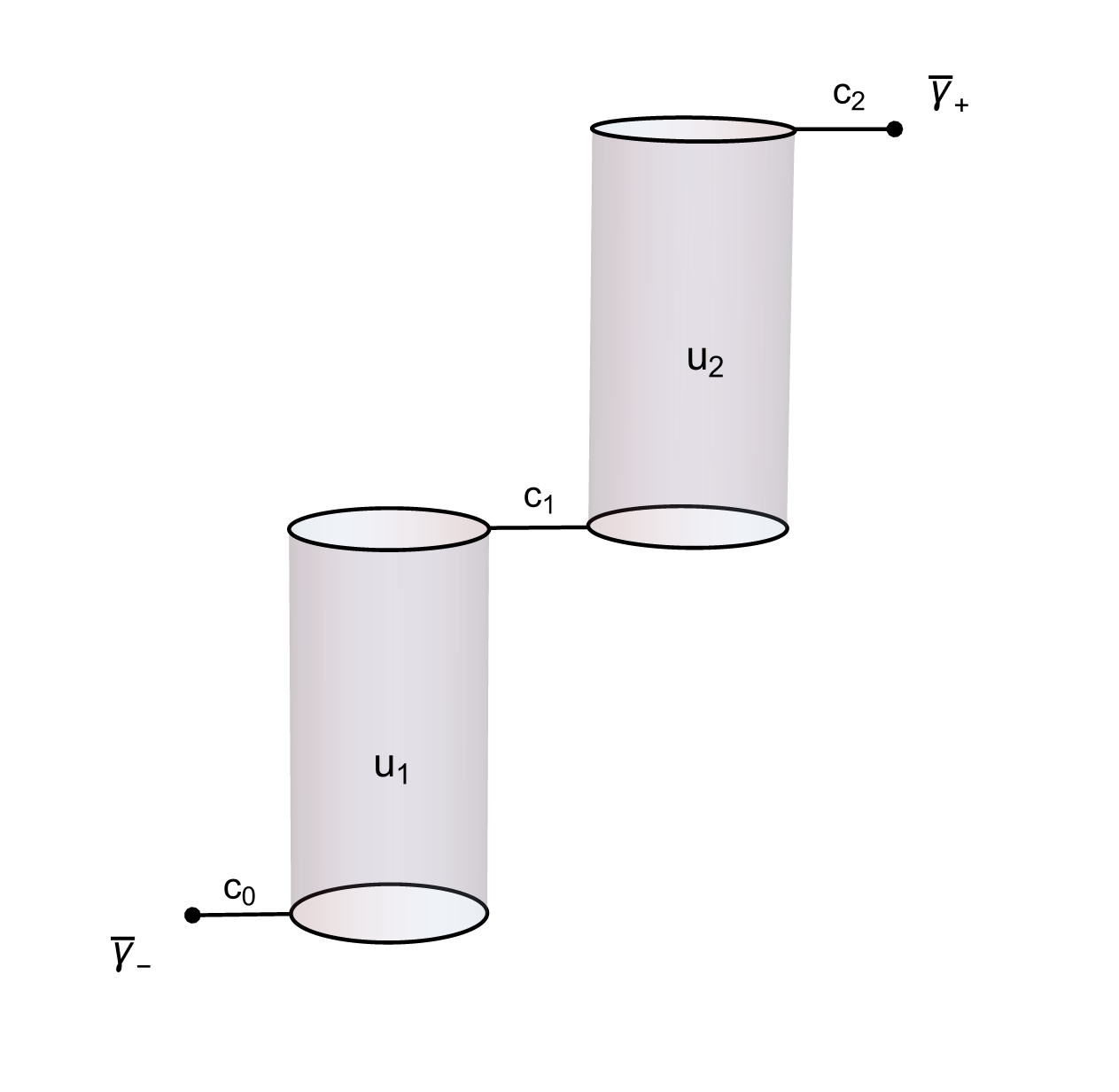}
    \caption{Floer cascade with $m=2$}
  \end{subfigure}
  \hspace{0.03\textwidth}
  \begin{subfigure}[b]{0.4\textwidth}
    \centering
    \includegraphics[width=\textwidth]{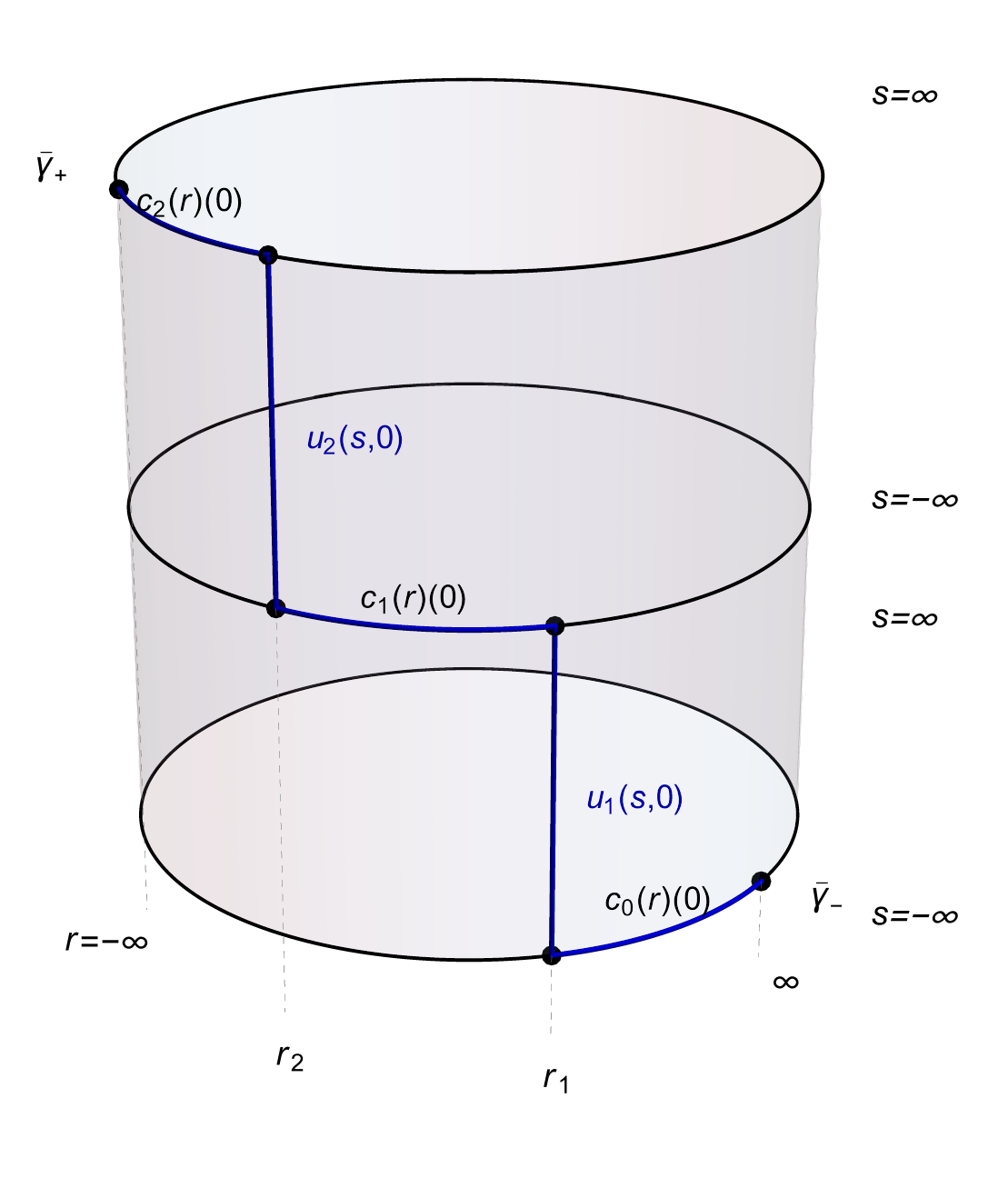}
    \caption{Equivalent expression}
  \end{subfigure}
  \caption{Floer Cascade and its equivalent expression via evaluating at $t=0$}
    \label{Figure - Cascades}
  \end{figure}
According to \cite{Bourgeois_Oancea_09}, this moduli space is a smooth manifold of dimension $\mu(\bar{\g}_+)-\mu(\bar{\g}_-)-1$, and it is compact when $\mu(\bar{\g}_+)-\mu(\bar{\g}_-)-1=0$. In addition, the assertions in \cite[Section 3.2]{Ginzburg_10} guarantee that the Floer cylinders, and thus the Floer cascades between orbits in $U$, lie entirely in $U$. Therefore, we can define the differential on the local Floer chain complex as follows:
    \begin{align*}
        \partial:CF_*^\loc(\gamma_0,H)&\rightarrow CF_{*-1}^\loc(\gamma_0,H)\\
        \bar{\g}_+&\mapsto\bigoplus_{\bar{\g}_-}\#_2\mathcal{M}(\bar{\g}_-,\bar{\g}_+)\bar{\g}_-.
    \end{align*}
Furthermore, since the broken Floer cascades also lie entirely in $U$, we have $\partial^2=0$, and we can define the local Floer homology of $H$ at $\gamma_0$ by
    \begin{equation*}
        HF_*^{\loc}(\gamma_0,H)=H_*(CF_*^{\loc}(\gamma_0,H)).
    \end{equation*}

As in \cite[Section 3.2]{Ginzburg_10}, a continuation argument shows that $HF_*^{\loc}(\gamma_0,H)$ is invariant under the choice of $H$ and $J$. In particular, if $H_0$ is non-degenerate, $\gm_0$ and $\gM_0$ are the only generators of the chain complex $CF_*^{\loc}(\gamma_0,H_0)$, and there are exactly two Floer cascades between them, which correspond to the Morse trajectories of $f_{\g_0}$.
	Therefore, we have
    \begin{equation*}
        HF_*^{\loc}(\gamma_0,H)\cong HF_*^{\loc}(\gamma_0,H_0)=\Z_2\langle\gm_0,\gM_0\rangle.
    \end{equation*}
Additionally, for a 1-parameter family of Hamiltonians $H_\eps$, suppose $H_0$ is degenerate and a bifurcation occurs at $\gamma_0$. Then, there exists a sufficiently small $\eps'>0$ such that $H_{\pm\eps'}$ is non-degenerate, and thus
    \begin{equation*}
        HF_*^{\loc}(\gamma_0,H_{-\eps'})\cong HF_*^{\loc}(\gamma_0,H_{\eps'}).
    \end{equation*}
    
\subsection{Counting Floer Cascades}
In order to describe the Floer chain complex explicitly, counting specific type of Floer cascades is necessary. To count such Floer cascades, we need the following theorem.
        \begin{thm}\label{Thm - Cylinder Counting}
        Let $\g_\pm$ be two simple non-constant $1/k_\pm$-periodic orbits, so that $\g_\pm^{k_\pm}$ are 1-periodic orbits. Also, let $J$ be a generic almost complex structure.
        Moreover, assume that
        \begin{equation*}
            \mu_{CZ}^\fT(\gamma^{k_+}_+)-\mu_{CZ}^\fT(\gamma^{k_-}_-)=1.
        \end{equation*}
        Then, the quotient $\mathcal{M}_0(\g^{k_-}_-,\g^{k_+}_+)/S^1$ is a zero-dimensional $\Z_{(k_+,k_-)}$-orbifold, and we have
        \begin{align*}
            \#_2\mathcal{M}(\gM_-^{k_-},\gM_+^{k_+})&\equiv k_-\cdot\#(\mathcal{M}_0(\g^{k_-}_-,\g^{k_+}_+)/S^1)\pmod 2\\
            \#_2\mathcal{M}(\gm_-^{k_-},\gm_+^{k_+})&\equiv k_+\cdot\#(\mathcal{M}_0(\g^{k_-}_-,\g^{k_+}_+)/S^1)\pmod 2,
        \end{align*}
        where $(k_+,k_-)$ is the greatest common divisor of $k_+$ and $k_-$, and the right hand side is an orbifold count.
    \end{thm}
    \begin{proof}
First, by \cite[Proposition 3.5]{Bourgeois_Oancea_09}, $\mathcal{M}_0(\g^{k_-}_-,\g^{k_+}_+)$ is a $1$-dimensional manifold for generic almost complex structure $J$. Therefore, bubbling does not occur since it is codimension 2 phenomenon. On the other hand, suppose these cylinders break through another orbit $\delta$. Then $\mu_{CZ}({\delta})$ must be equal to either $\mu_{CZ}(\g^{k_+}_+)$ or $\mu_{CZ}(\g^{k_-}_-)$. Since Floer cylinders between orbits of the same index can only be constant cylinders, $\delta$ must be either $\g^{k_\pm}_\pm$, which is a contradiction. Therefore, breaking also does not occur, meaning $\mathcal{M}_0(\g^{k_-}_-,\g^{k_+}_+)$ is a closed $1$-dimensional manifold. Consequently, each connected component of $\mathcal{M}_0(\g^{k_-}_-,\g^{k_+}_+)$ is diffeomorphic to $S^1$. Now, we define an $S^1$-action on $\mathcal{M}_0(\g^{k_-}_-,\g^{k_+}_+)$ by
        \begin{equation*}
            \tau\cdot u (s,t)= u(s,t+\tau).
        \end{equation*}
This action acts as a rotation on each connected component $\mathcal{M}^i$ for $i \in I$. Let $K_i$ be the kernel of this action restricted to $\mathcal{M}^i$. Then, considering the asymptotics, we have
        \begin{equation*}
            K_i\leq \Z_{(k_+,k_-)} < S^1.
        \end{equation*}
Now, write $K_i\cong\Z_{k_i}$ for some $k_i|(k_+,k_-)$. Then, the quotient $\mathcal{M}_0(\g^{k_-}_-,\g^{k_+}_+)/S^1$ is a zero-dimensional $\Z_{(k_+,k_-)}$-orbifold with orbifold count
        \begin{equation*}
            \#(\mathcal{M}_0(\g^{k_-}_-,\g^{k_+}_+)/S^1)=\sum_{i\in I}\frac{1}{k_i}.
        \end{equation*}
        
Now, suppose $\mathbf{u}\in\mathcal{M}(\gM^{k_-}_-,\gM^{k_+}_+)$ consists of more than one Floer cylinder. Then, at least one of the Floer cylinders is contained in $\mathcal{M}_0(\g_i,\g_{i+1})$ such that $\mu_{CZ}(\gM_{i+1})-\mu_{CZ}(\gM_i)\le 0$, which is nonempty only if $\g_i=\g_{i+1}$. Moreover, $\mathcal{M}_0(\g_i,\g_i)$ contains only constant cylinders, which gives a contradiction. Therefore, we have
        \begin{equation*}
            \mathbf{u}=(c_0,u_1,c_1)
        \end{equation*}
        for any $\mathbf{u}$.
        Moreover, since $c_0$ is a Morse trajectory of $f_{\g_-}$ with 
        \begin{equation*}
            \lim_{r\rightarrow-\infty} c_0(r)=\gM_-=\max f_{\g_-},
        \end{equation*}
        $c_0$ is a constant trajectory, and thus we have
        \begin{equation}\label{asympcond}
            \lim_{s\rightarrow-\infty}u_1(s,0)=\gM_-(0).
        \end{equation}
        Furthermore, for any $u_1$ satisfying \cref{asympcond}, there is a unique Morse trajectory $c_1$ starting from
        \begin{equation*}
            \lim_{s\rightarrow\infty}u_1(s,t)\notin \crit f_{\g_+},
        \end{equation*}
and it satisfies $\lim_{r\rightarrow\infty}c_1(r)=\gM_+$. Therefore, we only need to count $u_1\in \mathcal{M}_0(\g^{k_-}_-,\g^{k_+}_+)$ satisfying \cref{asympcond}.
        Now, suppose $u\in\mathcal{M}^i$. Then, there exists $\tau_0\in S^1$ such that $\tau_0\cdot u$ satisfies \cref{asympcond}. Since $\lim_{s\rightarrow-\infty}u_1(s,t)$ is invariant under the action of $1/k_-$,
        \begin{equation*}
            \left(\tau_0+\frac{1}{k_-}\right)\cdot u, \,\,  \left(\tau_0+\frac{2}{k_-}\right)\cdot u, \,\, \ldots,\,\,\left(\tau_0+\frac{k_-/k_i}{k_-}\right)\cdot u=\left(\tau_0+\frac{1}{k_i}\right)\cdot u
        \end{equation*}
are $k_-/k_i$ distinct Floer cylinders in $\mathcal{M}^i$ satisfying \cref{asympcond}. Collecting all these, we obtain
        \begin{equation*}
            \#_2\mathcal{M}(\gM^{k_-}_-,\gM^{k_+}_+)\equiv \sum_{i\in I}\frac{k_-}{k_i}=k_-\cdot\#(\mathcal{M}_0(\g^{k_-}_-,\g^{k_+}_+)/S^1)\pmod 2.
        \end{equation*}
In the same way, for $\mathbf{u}\in\mathcal{M}(\gm^{k_-}_-,\gm^{k_+}_+)$, we also have $\mathbf{u}=(c_0,u_1,c_1)$, with a constant Morse trajectory $c_1$, and a unique $c_0$ followed by $u_1$. Therefore, we only need to count $u_1\in \mathcal{M}_0(\g^{k_-}_-,\g^{k_+}_+)$ satisfying
        \begin{equation*}
            \lim_{s\rightarrow\infty}u_1(s,0)=\gm_+(0).
        \end{equation*}
Similarly, there are $k_+/k_i$ distinct Floer cylinders in $\mathcal{M}^i$ satisfying the above condition, so we obtain
        \begin{equation*}
            \#_2\mathcal{M}(\gm^{k_-}_-,\gm^{k_+}_+)\equiv \sum_{i\in I}\frac{k_+}{k_i}=k_+\cdot\#(\mathcal{M}_0(\g^{k_-}_-,\g^{k_+}_+)/S^1)\pmod 2.\qedhere
        \end{equation*}
The case of a cylinder connecting the doubly and triply covered orbits is illustrated in \Cref{Figure - CascadeCounting}.
    \end{proof}
\begin{rmk}
    \rm If $(k_+,k_-)=1$ in the above theorem, $\mathcal{M}_0(\g^{k_-}_-,\g^{k_+}_+)/S^1$ becomes a manifold, and $\#(\mathcal{M}_0(\g^{k_-}_-,\g^{k_+}_+)/S^1)$ would just be its count. Also, the above theorem continues to hold for more general settings, such as whole symplectic manifold and capped orbits. 
\end{rmk}

Let $\g_\pm$ be periodic orbits, and let $u$ be a Floer cylinder between $\g_\pm$.
For a point $a \in \im(\g_\mp)$, we define the \textbf{evaluation map} as a set-valued map given by
\[
\ev^\pm(a;u) = \left\{b\in\im(\g_\pm)\,:\,\text{for }t_0\text{ such that }a=\lim_{s\to\mp\infty}u(s,t_0),\,b=\lim_{s\to \pm\infty}u(s,t_0) \right\}\subset \im(\g_\pm),
\]
double signs in same order. If $\g_\mp$ is a simple orbit, then $\ev^\pm(a;u)$ is a singleton. However, if $\g_\mp$ is a $k$-fold cover of a simple orbit, $\ev^\pm(a;u)$ may consist of multiple distinct points, since there can be multiple values of $t_0$ satisfying $\lim_{s\to\mp\infty}u(s,t_0)=a$ as in the proof of \Cref{Thm - Cylinder Counting}; see \Cref{Figure - CascadeCounting} for examples. Furthermore, note that this map is well-defined whenever a connecting Floer cylinder exists, even if the relative Conley--Zehnder index satisfies $\mu_{CZ}(\g_+,\g_-)\geq 2$.

\begin{figure}[htbp]
  \begin{subfigure}[b]{0.45\textwidth}
    \centering
    \includegraphics[width=\textwidth]{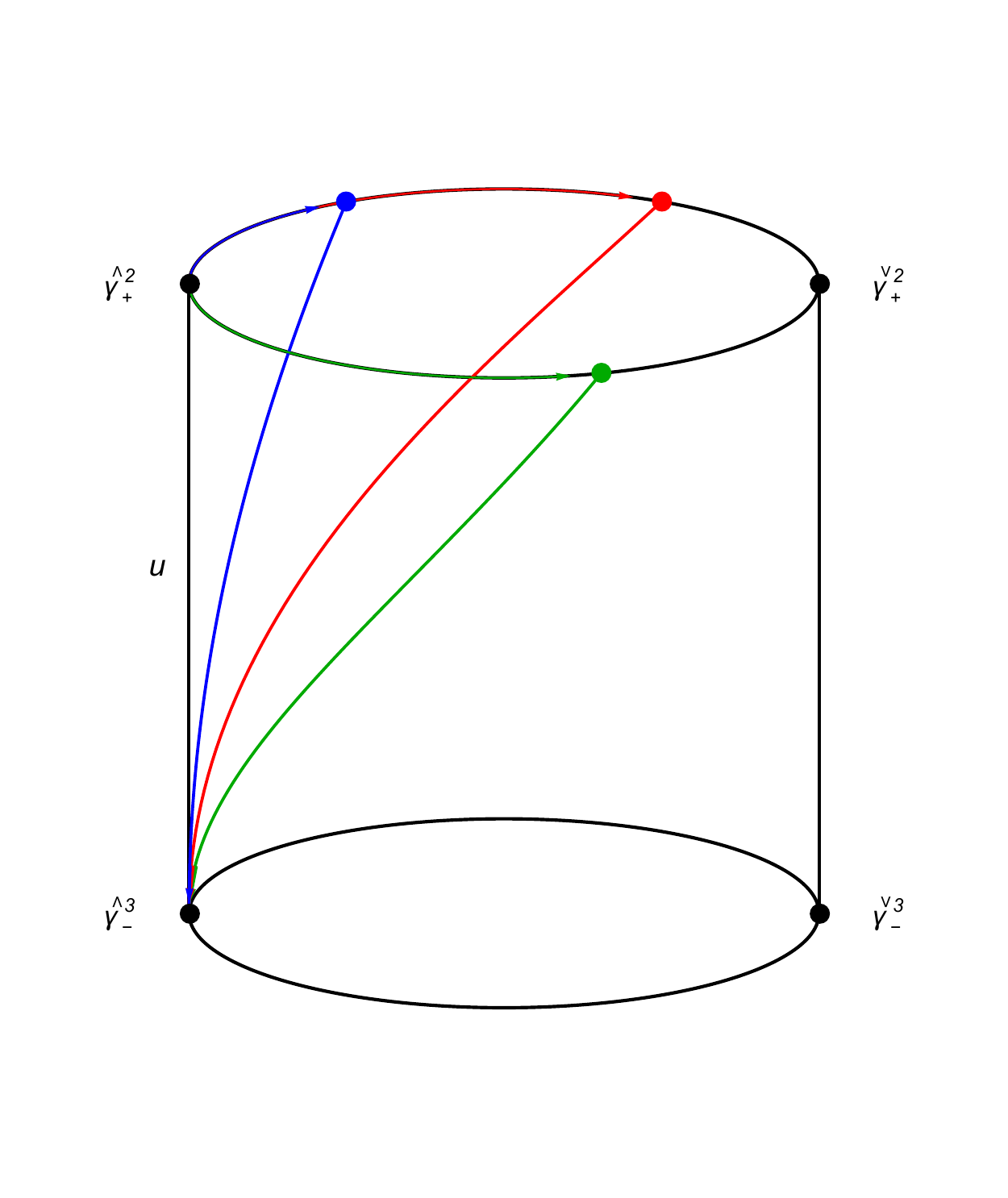}
    \caption{Cascades through $\ev^+(\gM_-^3;u)$}
  \end{subfigure}
  \hspace{0.07\textwidth}
  \begin{subfigure}[b]{0.45\textwidth}
    \centering
    \includegraphics[width=\textwidth]{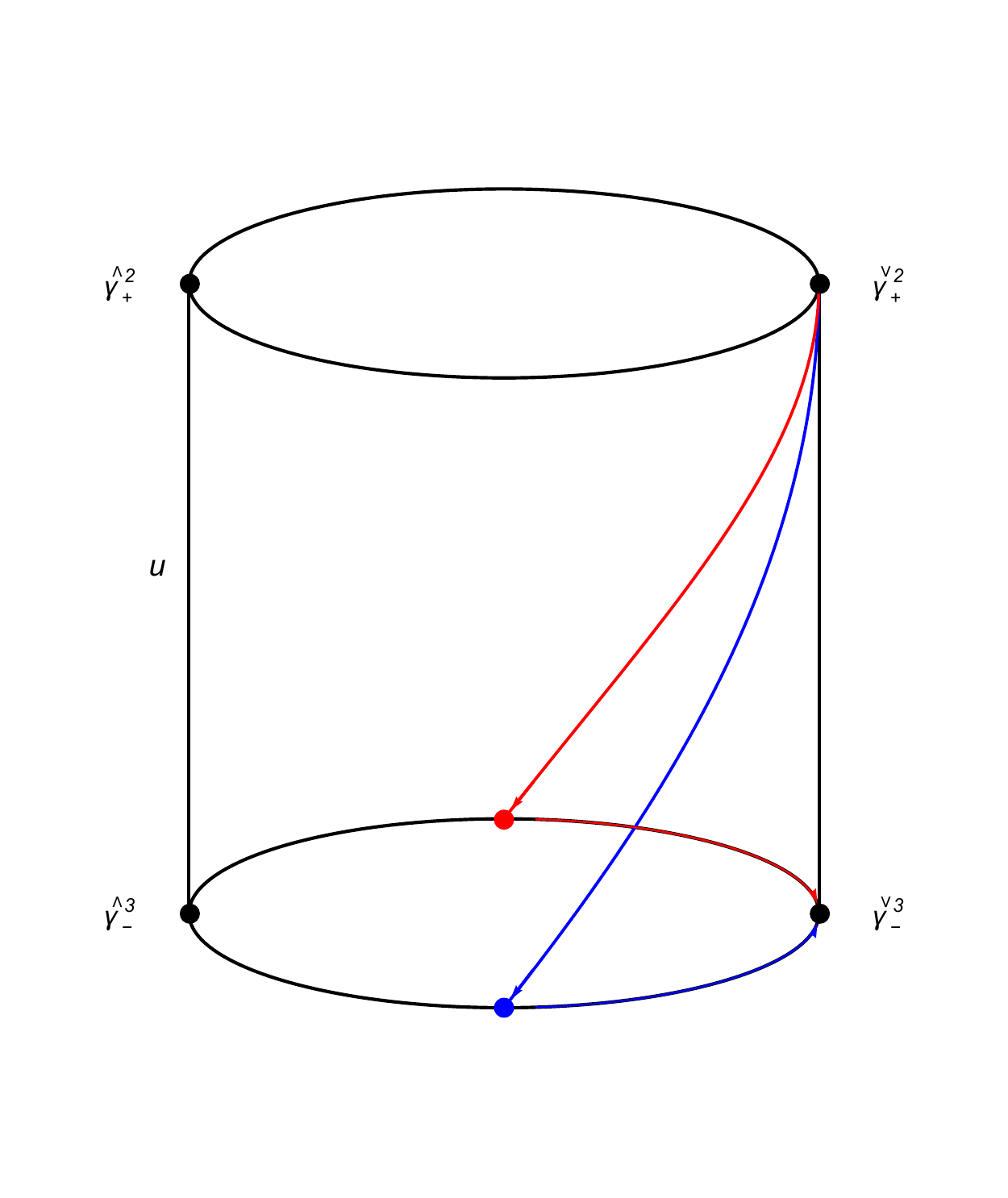}
    \caption{Cascades through $\ev^-(\gm_+^2;u)$}
  \end{subfigure}
  \caption{Floer cascades containing $u$ between $2$-cover and $3$-cover orbits}
    \label{Figure - CascadeCounting}
  \end{figure}

\section{Construction of the Floer Cylinders}\label{Section - cylinder}

\subsection{Model Hamiltonians}\label{Subsection - Model}
In this section, we construct a model Hamiltonian for each bifurcation type.
We first define the model of the tubular neighborhood of the orbit as
    $$
    W = I_r\times S^1_\theta \times D^2_{z=x+iy},
    $$
where $I$ is a small interval containing $0$, $S^1$ is regarded as $\R/\Z$, and we equip it with a split symplectic form
    $$
    \o =dr \w d\theta + dx \w dy.
    $$
We take a function $h=h(r)$ such that $h'(0)=1$ and $h''(0)>0$.
For a coprime pair $(k,l)$, we define a covering space
\[
\tilde{W}_{k,l}:=I\times \R/k\Z\times D^2= I\times \R\times D^2 / (r,\theta,z)\sim (r,\theta\text{ mod }1,e^{2\pi il\theta/k}z),
\]
equipped with a covering map
    $$
    \begin{aligned}
    \mathrm{cov}_{k,l}:\tilde{W}_{k,l}&\to W\\
    (r,\theta,z)&\mapsto (r,\theta\text{ mod }1, e^{2\pi i lm/k}z)\quad\text{if }m\leq \theta <m+1.
    \end{aligned}
    $$
This is exactly the same as the one mentioned in \Cref{Subsection - Unwrapped generating Hamiltonian}.
Let a coprime pair $(k,l)$ be fixed, and let $G_\eps$ be the unwrapped generating Hamiltonian of the bifurcation of the $k$-th cover of a periodic orbit with Floquet multiplier $\exp (2\pi i l/k)$.
We assume that $G_\eps$ is truncated at a finite degree.
We define a \textbf{model Hamiltonian} $H_\eps=H_{\eps;k,l}$ first on $\tilde{W}_{k,l}$ by
    $$
    \tilde{H}_{\eps;k,l}(r,\theta,x,y) = kh(r) +  G_\eps (x,y).
    $$
Note that $G_\eps$ is $\Z_k$-symmetric, so $\tilde{H}_{\eps;k,l}$ is equivariant under the deck transformation and descends to a well-defined Hamiltonian $H_{\eps;k,l}$ on $W$.

\begin{lmm}\label{Lemma - Critical points and Periodic orbits}
    For a sufficiently small $D^2$ and $\eps$, there exists a 1-1 correspondence between the critical points of $G_\eps$ and the 1-periodic orbits of $\tilde{H}_\eps$.
    The correspondence is given by
        $$
        \begin{aligned}
            \crit(G_\eps)&\to \mathcal{P}(\tilde{H}_\eps)\\
            p&\mapsto \g_p(t) = (0,kt,p).
        \end{aligned}
        $$
\end{lmm}
\begin{proof}
The Hamiltonian equations for $\tilde{H}_\eps$ are given by
    \[
    \begin{pmatrix}
        \dot{r}\\
        \dot{\theta}\\
        \dot{x}\\
        \dot{y}
    \end{pmatrix}
    =
    \begin{pmatrix}
        -\pp_\theta\tilde{H}_\eps\\
        \pp_r\tilde{H}_\eps\\
        -\pp_y\tilde{H}_\eps\\
        \pp_x\tilde{H}_\eps
    \end{pmatrix}
    =\begin{pmatrix}
        0\\
        kh'(r)\\
        -\pp_y G_\eps\\
        \pp_x G_\eps
    \end{pmatrix}.
    \]
We see that $r(t)=r_0$ is constant along the orbit, and so is $h'(r)=h'(r_0)$. For the orbit to be 1-periodic in $\tilde{W}_{k,l}$, $\theta$ must increase by exactly $k$ over $t \in [0,1]$, which requires $h'(r_0)=1$. Since $h'(0)=1$ and $h''(0) \neq 0$, this implies $r_0=0$ is constantly $0$ along the orbit.

The transverse part of the orbit corresponds to a 1-periodic orbit of the Hamiltonian flow of $G_\eps$. Since we take the neighborhood to be small enough so that $G_\eps$ is $C^2$-small, its only 1-periodic orbits are constant orbits, which correspond exactly to the critical points of $G_\eps$.
\end{proof}

\begin{lmm}[Model Hamiltonian]\label{Lemma - Model Hamiltonian}
Let $H_\eps=H_{\eps;k,l}$ be the model Hamiltonian on a sufficiently small $D^2$, constructed with an unwrapped generating Hamiltonian $G_\eps$ of a certain bifurcation type. Then:
\begin{enumerate}
    \item $H_0$ has only one periodic orbit $\g_0(t)=(0,kt,0)$ with period $1/k$.
    \item As $\eps$ varies, $\g_0$ experiences the bifurcation corresponding to $G_\eps$.
\end{enumerate}
\end{lmm}
\begin{proof}
This is a direct consequence of the construction of $G_\eps$ as a generating Hamiltonian of $k$-th return map of the corresponding bifurcation type. Note that there exists a $k$-to-$1$ correspondence between the periodic orbits of $\tilde{H}_{\eps;k,l}$ and the periodic orbits of $H_{\eps;,k,l}$ except the core orbit, given by the covering map.
\end{proof}
Now we rigorously verify the relative Conley--Zehnder indices for the periodic orbits of our model Hamiltonians, which were anticipated in \Cref{Subsection - Relative CZ index}.

\begin{prop}\label{Proposition - Relative CZ Index}
    Let $H_{\eps}$ be the model Hamiltonian on the base space $W$.
    For any two critical points $p$ and $q$ of the unwrapped generating Hamiltonian $G_\eps$, let $\g_p$ and $\g_q$ be the corresponding periodic orbits of $H_\eps$. 
    Then, their relative Conley--Zehnder index is given by
        \[
        \mu_{CZ}(\g_p, \g_q) = \Ind(q) - \Ind(p),
        \]
    where $\Ind$ denotes the Morse index as a critical point of $G_\eps$.
\end{prop}
\begin{proof}
    Let $\fT$ be a trivialization of the symplectic subbundle $\xi$ over $W$ along the central axis $\{r=0\}$, and let $\tilde{\fT} = (\mathrm{cov}_{k,l})^*\fT$ be its pullback 
    to $\tilde{W}_{k,l}$.
    Since $\mathrm{cov}_{k,l}$ is a local symplectomorphism intertwining the two Hamiltonian flows, and since each orbit $\g_p$ is the image of its lift $\tilde{\g}_p$ under $\mathrm{cov}_{k,l}$ with the same period, the linearized flow along $\tilde{\g}_p$ expressed in $\tilde{\fT}$ and the linearized flow 
    along $\g_p$ expressed in $\fT$ are the same path in $Sp(2)$. Hence
    \[
    \mu_{CZ}^{\tilde{\fT}}(\tilde{\g}_p) = \mu_{CZ}^{\fT}(\g_p)
    \]
    for every critical point $p$ of $G_\eps$, and it suffices to compute the relative index in the covering space.

    On $\tilde{W}_{k,l}$, the Hamiltonian is decoupled as $\tilde{H}_{\eps;k,l} = kh(r) + G_\eps(x,y)$, and we may choose $\tilde{\fT}$ compatible with the splitting, and the relative index is independent of this choice.
    The linearized flow along each $\tilde{\g}_p$ then splits into a longitudinal part, governed by $kh(r)$, and a transverse part, governed by the linearized flow of the autonomous Hamiltonian $G_\eps$ at $p$.
    The longitudinal path is the same for all the orbits $\tilde{\g}_p$, so its index contribution cancels in the relative index. For the transverse part, since $G_\eps$ is $C^2$-small, the contribution is $1-\Ind(p)$ as in \Cref{Subsection - Relative CZ index}. Therefore
    \[
    \mu_{CZ}(\g_p,\g_q) 
    = \mu_{CZ}^{\tilde{\fT}}(\tilde{\g}_p) 
      - \mu_{CZ}^{\tilde{\fT}}(\tilde{\g}_q)
    = \Ind(q) - \Ind(p).\qedhere
    \]
    \end{proof}

\subsubsection*{Relation between a general Hamiltonian}

The model Hamiltonians introduced above should be understood as Floer-theoretic normal forms of the local bifurcation germ, rather than as literal representatives of arbitrary Hamiltonian systems. We explain this explicitly.

Let $\bar{H}_\eps$ be a generic one-parameter family of Hamiltonians defined on a tubular neighborhood $W$ of a periodic orbit $\g_0$, and suppose that $\bar{H}_\eps$ undergoes one of the bifurcations considered in this paper at $\eps=0$. Let $\bar{\Psi}_\eps$ denote the corresponding Poincare return map, and let $k$ be the covering number relevant for the bifurcation. We write $H_\eps$ for the model Hamiltonian constructed above from the finite degree truncation of the unwrapped generating Hamiltonian of $\bar{\Psi}_\eps^k$.

We first explain which part of a general Hamiltonian flow is seen by the model Hamiltonian. After choosing a Poincare section and a framing along $\g_0$, the flow of $\bar{H}_\eps$ can be decomposed into three pieces:
\[
        \text{(the part encoded by the return map)}
        \,+\,
        \text{(a transverse rotation)}
        \,+\,
        \text{(a Hamiltonian loop)}.
\]
More precisely, the transverse linearized flow along $\g_0$ may contain a $2\pi N$-rotation in the chosen framing. This rotational part is not an additional bifurcation datum, and it can be discarded by an appropriate choice of framing.
This choice changes absolute Conley--Zehnder indices by the usual framing shift
\[
        \mu_{CZ}^{\fT'}(\g)
        =
        \mu_{CZ}^{\fT}(\g) + 2N,
\]
with the sign depending on the convention for the orientation of the loop. Since the same rotation acts on all nearby orbits, this shift cancels in all relative Conley--Zehnder indices used below.

After this rotational part is recorded in the framing, we consider the Hamiltonian loop, of which the time 1-flow is an identity. This loop part can be removed by Floer-theoretic naturality; see \cite{Seidel_97,Ginzburg_Gurel_10}. If $\psi_\eps^t$ denotes this loop, then the time-dependent change of variables
\[
        u(s,t) \mapsto (\psi_\eps^t)^{-1}(u(s,t))
\]
identifies the corresponding Floer equations, at the cost of replacing the almost complex structure $J_t$ by the pulled-back almost complex structure
\[
        \widetilde J_t
        =
        (d\psi_\eps^t)^{-1}
        \circ J_t \circ
        d\psi_\eps^t .
\]
Thus the Hamiltonian loop part can be recorded at the choice of a regular almost complex structure, which is irrelevant up to chain homotopy. 

It remains to discuss the finite degree truncation. Let $\bar{G}_\eps$ be the unwrapped generating Hamiltonian of $\bar\Psi_\eps^k$, and let $G_\eps$ be its truncation to a degree $n$ chosen larger than the determinacy degree, say $n\geq k+2$ for $k$-th cover bifurcation. The model Hamiltonian $H_\eps$ is constructed from $G_\eps$, so the return map of
$H_\eps$ agrees with $\bar{\Psi}_\eps$ only up to order $n$. The discrepancy beyond degree $n$ is not a Hamiltonian loop, and is not absorbed by the framing.

The point is that the bifurcation types considered here are finite-jet phenomena. This is a standard argument in the analysis of bifurcations, see \cite{Meyer_70,Abraham_Marsden_78,Arnold_89,Meyer_Hall_Offin_13}. After shrinking the isolating neighborhood if necessary, the finite jet $G_\eps$ determines the local data used in our calculation: the number of nearby periodic orbits, their relative positions and Morse indices as critical points of the unwrapped generating Hamiltonian, stability types and the resulting relative Conley--Zehnder indices. The higher-order remainder can change the literal chain-level representative for a particular Hamiltonian family, but as long as no additional periodic orbit enters or leaves the chosen isolating neighborhood, it is related to the model by the standard isolated continuation argument for local Floer homology.

Consequently, the computations in \Cref{Section - LFCC} should be read as explicit computations for a normal form representative. We fix the model Hamiltonian $H_\eps$, the split almost complex structure adapted to it, and Morse functions on the periodic orbits with exactly one maximum and one minimum. With these choices, the relevant Floer cylinders are the gradient revolutions constructed below, which will be proven in \Cref{Prop - Uniqueness} and the local Floer chain complex can be counted explicitly. Different Hamiltonian families with the same return map, or different auxiliary data, may give different chain complexes, but they represent the same local Floer homology under isolated continuation in the sense of \cite{Usher_Zhang_16}. The displayed complexes capture the essential chain-level mechanism of the bifurcation in this representative.

\subsection{Gradient Revolution}\label{Subsection - Gradient Revolution}
In this section, we construct a Floer cylinder directly from the gradient flow of the unwrapped generating Hamiltonian.
First, let $G_\eps$ be the unwrapped generating Hamiltonian of a fixed bifurcation type and $\tilde{H}_{\eps;k,l}$ be model Hamiltonian defined on a covering space $\tilde{W}_{k,l}$.
Let $p_\pm$ be critical points of $G_\eps$, and $\g_\pm$ be periodic orbits corresponding to $p_\pm$.
Suppose there is a negative gradient trajectory $\eta(s)$ on $D^2$ from $p_+$ to $p_-$, satisfying $\lim_{s\to\pm\infty}\eta(s)=p_\pm$ and $\pp_s \eta =-\nabla G_\eps(\eta)$.
By the index relation, if $\Ind(p_+)=\Ind(p_-)-1$, then \Cref{Proposition - Relative CZ Index} guarantees that the relative Conley--Zehnder index is $\mu_{CZ}(\tilde{\g}_{+})=\mu_{CZ}(\tilde{\g}_{-})+1$.

We equip $\tilde{W}_{k,l}$ with an $\omega$-compatible almost complex structure $J$ such that its restriction to the transverse section $D^2$ coincides with the standard complex structure $J_0$.
With respect to the standard Euclidean metric on $D^2$, our sign convention for the Hamiltonian vector field yields
    \[
    J_0 X_{G_\eps} =-\nabla G_\eps. 
    \]
We define the \textbf{gradient revolution} from $\tilde{\g}_{+}$ to $\tilde{\g}_{-}$ as the map $\tilde{u}: \R \times \R/\Z \to \tilde{W}_{k,l}$ given by
    \[
    \tilde{u}(s,t) = (0,kt,\eta(s)).
    \]

Now, we will define an almost complex structure $J$ on $\tilde{W}_{k,l}$ as follows. Define $J_\Sigma$ be the standard almost complex structure on $\Sigma$, and $J_R$ be the almost complex structure on $R=\R \times S^1$ such that $J(\partial_r)=\partial_\theta$. Now, define
\begin{equation}\label{J}
    J=J_R\oplus J_\Sigma.
\end{equation}
Then, clearly, $J$ is $\mathrm{cov}_{k,l}$-equivariant and $\o$-compatible.

\begin{prop}\label{Proposition - Gradient revolution in cover}
    The gradient revolution $\tilde{u}(s,t)$ is a Floer cylinder in $\tilde{W}_{k,l}$ satisfying \[\lim_{s\to \pm\infty}\tilde{u}(s,t)=\tilde{\g}_\pm(t).\]
\end{prop}
\begin{proof}
    We need to verify that $\tilde{u}(s,t)$ satisfies the Floer equation $\bar{\pp}_J u = 0$. 
    By evaluating the derivatives along the cylinder, we note that the longitudinal component is exactly $\theta(t) = kt$, yielding $\pp_t \tilde{u} = k\pp_\theta$.
    Recall that the model Hamiltonian is $\tilde{H}_{\eps;k,l} = kh(r) + G_\eps(u,v)$. Since $r=0$ and $h'(0)=1$, its Hamiltonian vector field evaluated along the cylinder is $X_{\tilde{H}_\eps} = k\pp_\theta + X_{G_\eps}$.
    Substituting these into the Floer equation gives
        \[
        \begin{aligned}
            \bar{\pp}_J \tilde{u} &= \pp_s \tilde{u} + J(\pp_t \tilde{u} - X_{\tilde{H}_\eps}) \\
            &= \pp_s \eta + J\big(k\pp_\theta - (k\pp_\theta + X_{G_\eps})\big) \\
            &= -\nabla G_\eps - J_0 X_{G_\eps}= 0.
        \end{aligned}
        \]
    The asymptotic boundary conditions follow from the limits of the Morse trajectory $\eta(s)$.
\end{proof}

\begin{prop}\label{Propotision - Gradient Revolution}
    The gradient revolution $\tilde{u}$ canonically descends to a Floer cylinder $u$ in the base space $W$, connecting the corresponding periodic orbits $\g_{+}$ and $\g_{-}$.
\end{prop}
\begin{proof}
    The covering map $\mathrm{cov}_{k,l}: \tilde{W}_{k,l} \to W$ is locally a symplectomorphism.
    Since $J$ is $\mathrm{cov}_{k,l}$-equivariant on $\tilde{W}_{k,l}$, it descends to a well-defined $\omega$-compatible almost complex structure on $W$.
    Because the projection of a $J$-holomorphic curve under a locally holomorphic and symplectic covering map remains a holomorphic curve, the projected map $u = \mathrm{cov}_{k,l} \circ \tilde{u}$ automatically satisfies the Floer equation for the base Hamiltonian $H_\eps$ on $W$.
    Furthermore, since $\mathrm{cov}_{k,l}$ maps the lifted orbits $\tilde{\g}_{\pm}$ to the base orbits $\g_{\pm}$, the asymptotic boundary conditions are preserved.
\end{proof}

    \subsection{Uniqueness and Regularity}\label{Subsection - Uniqueness}

    Having explicitly constructed the gradient revolutions, we must ensure that they are regular, and that they are the only Floer cylinders contributing to the Floer differential. In this subsection, we establish this regularity and uniqueness properties. This step is crucial, as it guarantees that no unexpected trajectories exist, thereby allowing us to completely determine the local Floer chain complex strictly via Floer cascade counting.
    
    \begin{prop}\label{Prop - Uniqueness}
        For $(W,H_\eps)$ defined in \Cref{Subsection - Model,Subsection - Gradient Revolution}, there are no other Floer cylinders of $H_\eps$ other than gradient revolutions.
    \end{prop}
    \begin{proof}
        Let $\g_\pm$ be 1-periodic orbits of $H_\eps$ in $W$ with winding number $k$, i.e. $[\g_\pm]=k\in \Z\cong \pi_1(W)$, and $u:\R\times S^1\rightarrow W$ satisfies the Floer equation
        \begin{equation*}
            \partial_su+J(\partial_tu-X_{H_\eps})=0,
        \end{equation*}
        and asymptotic to $\g_\pm$. Since
        $(\mathrm{cov}_{k,l})_*(\pi_1(W_{k,l}))= k\Z\subset \pi_1(W)$, there exists a lift $\tilde{u}:\R\times S^1\rightarrow \tilde{W}_{k,l}$ of $u$, whose asymptotes $\tilde{\g}_\pm$ are also lifts of $\g_\pm$. Also, $\tilde{u}$ satisfies the Floer equation
        \begin{equation*}
            \partial_s\tilde{u}+J(\partial_t\tilde{u}-X_{\tilde{H}_\eps})=0.
        \end{equation*}
         Now, define projection maps
    \begin{equation*}
        \pr_1:\tilde{W}_{k,l}\rightarrow R, \qquad \pr_2:\tilde{W}_{k,l}\rightarrow \Sigma.
    \end{equation*}
    Then, since $J=J_R\oplus J_\Sigma$ and $X_{\tilde{H}_\eps}=X_{kh}+X_{G_\eps}$, the projections
        \begin{equation*}
            u_1=\pr_1\circ \tilde{u}, \qquad u_2=\pr_2\circ \tilde{u}
        \end{equation*}
        of $\tilde{u}$ satisfy Floer equations
        \begin{equation*}
            \partial_su_1+J_R(\partial_tu_1-X_{kh})=0, \qquad\partial_su_2+J_\Sigma(\partial_tu_2-X_{G_\eps})=0,
        \end{equation*}
        and their asymptotic conditions are
        \begin{align*}
            \lim_{s\rightarrow \pm\infty}u_1&=\pr_1\circ \g_\pm=\{0\}\times S^1,\\
            \lim_{s\rightarrow \pm\infty}u_2&=\pr_2\circ \g_\pm=p_\pm,
        \end{align*}
        where $p_\pm$ are critical points of $G_\eps$ corresponding to $\g_\pm$. Therefore, $u_1$ should be a constant Floer cylinder
        \begin{equation*}
            u_1(s,t)=(0,kt).
        \end{equation*}
        By \Cref{Lemma - C2small,Lemma - C2smallsymmetric}, $G_\eps$ is $C^2$-small for each case if we choose the bifurcation parameter and the Poincar\'{e} section small enough.
        Therefore, $u_2$ is a negative gradient flow line of $G_\eps$ between $p_\pm$, namely,
        \begin{equation*}
            u_2(s,t)=\eta(s)
        \end{equation*}
        for some $\eta:\R\rightarrow D^2$ such that
        \begin{equation*}
            \lim_{s\rightarrow \pm\infty}\eta(s)=p_\pm, \qquad \partial_s\eta=-\nabla G_\eps(\eta).
        \end{equation*}
        Therefore, we have
        \begin{equation*}
            \tilde{u}(s,t)=(0,kt,\eta(s)),
        \end{equation*}
        which concludes $\tilde{u}$ is a gradient revolution.
    \end{proof}

    \begin{prop}[Regularity of the gradient revolution]\label{Proposition - Regularity of revolution}
    Let $J$ be as in (\ref{J}), and let $\eta(s)$ be a non-constant gradient flow of $G_\eps$. Then the gradient revolution $\tilde u(s,t)=(0,kt,\eta(s))$ is regular, i.e., the linearized operator $D_{\tilde u}$ is surjective.
\end{prop}
\begin{proof}
   
    Considering the proof of \Cref{Prop - Uniqueness}, $u_1=\pr_1\circ \tilde{u}$ is a Floer cylinder of $h:R\rightarrow\R$ whose both asymptotes are $\{0\}\times S^1$(which should be a constant cylinder), and $u_2=\pr_2\circ \tilde{u}$ is a Floer cylinder of $G_\eps:\Sigma\rightarrow\R$ which is precisely a Morse trajectory $\eta$ of $G_\eps$. Now define a smooth cutoff function
        \begin{equation*}
            \beta:\R\rightarrow [0,1], \qquad \beta'\ge0,\ \beta(0)=0,\ \beta(1)=1,\ \beta(s)=0\text{ if }s\leq 0,\ \beta(s)=1\text{ if }s\geq 1.
        \end{equation*}
        Trivialize $\tilde{u}^*TW$ with the basis $(\tilde{u}^*\partial_r,\tilde{u}^*\partial_\theta,\tilde{u}^*\partial_x,\tilde{u}^*\partial_y)$, and define the Banach space $W_{\theta}^{k,p,\delta}(\mathbb{R}\times S^1, \mathbb{R}^4)$ consisting of maps $\xi:\mathbb{R}\times S^1 \rightarrow\mathbb{R}^4$ such that the map
		\begin{equation*}
			(s,t)\mapsto e^{\delta\beta(s)s-\delta(1-\beta(s))s}(\xi(s,t)-c\beta(s)\tilde{u}^*\partial_\theta-d(1-\beta(s))\tilde{u}^*\partial_\theta)
		\end{equation*}
		belongs to $W^{k,p}(\mathbb{R}\times S^1, \mathbb{R}^4)$ for some $c,d\in\mathbb{R}$. Similarly, define the space $L^{p,\delta}(\mathbb{R}\times S^1, \mathbb{R}^4)$ consisting of $\eta$ such that the map
		\begin{equation*}
			(s,t)\mapsto e^{\delta\beta(s)s-\delta(1-\beta(s))s}\xi(s,t)
		\end{equation*}
        belongs to $L^p(\mathbb{R}\times S^1,\mathbb{R}^4)$. Under the above trivialization, $D_{\tilde{u}}$ is of the form
        \begin{align*}
            D_{\tilde{u}}:W_\theta^{1,p,\delta}(\mathbb{R}\times S^1, \mathbb{R}^4)&\rightarrow L^{p,\delta}(\mathbb{R}\times S^1, \mathbb{R}^4)\\
				\xi&\mapsto\partial_s\xi+J\partial_t\xi+
                \begin{pmatrix}
                    \begin{matrix}
					kh''_s(v) &0\\
					0&0
				\end{matrix} &0\\
                0 &\Hess G_\eps
                \end{pmatrix}\xi.
        \end{align*}
    Since $J=J_R\oplus J_\Sigma$ respects the splitting $T\tilde W_{k,l}=\langle\pp_r,\pp_\theta\rangle\oplus\langle\pp_u,\pp_v\rangle$ along $\tilde u$, the linearized operator can be written as
    \[
        D_{\tilde u} = D^{R}_{\tilde u}\oplus D^{\mathrm{\Sigma}}_{\tilde u},
    \]
    where
    \begin{align*}
        D^{R}_{\tilde u}:W_\theta^{1,p,\delta}(\mathbb{R}\times S^1, \mathbb{R}^2)&\rightarrow L^{p,\delta}(\mathbb{R}\times S^1, \mathbb{R}^2)\\
        \xi&\mapsto\partial_s\xi+J_{R}\partial_t\xi+
                \begin{pmatrix}
					kh''_s(v) &0\\
					0&0
				\end{pmatrix}\xi
    \end{align*}
    is a linearized operator of $u_1$, and
    \begin{align*}
        D^{\Sigma}_{\tilde u}:W^{1,p}(\mathbb{R}\times S^1, \mathbb{R}^2)&\rightarrow L^{p}(\mathbb{R}\times S^1, \mathbb{R}^2)\\
        \xi&\mapsto\partial_s\xi+J_{\Sigma}\partial_t\xi+
                \Hess G_\eps \cdot\xi
    \end{align*}
    is a linearized operator of $u_2$. Therefore, it suffices to show that $D^{R}_{\tilde u}$ and $D^{\Sigma}_{\tilde u}$ are surjective.
    \begin{enumerate}[label=(\roman*), topsep=0pt, leftmargin=15pt]
        \item (Surjectivity of $D^{R}_{\tilde u}$) 
        By \cite[Section 5.2]{Diogo_Lisi_19}, for sufficiently small $\delta$, $D^{R}_{\tilde u}$ is Fredholm of index 1, and its adjusted Chern number is $0$. Therefore, by the automatic transversality (\cite[Proposition 2.2]{Wendl_10b}), $D^{R}_{\tilde u}$ is surjective.
        \item (Surjectivity of $D^{\Sigma}_{\tilde u}$.) By \Cref{Proposition - Morse Trajectories}, $G_\eps$ is Morse-Smale. Therefore, by \cite[Section 10.1]{Audin_Damien_14}, $D^{\Sigma}_{\tilde u}$ is surjective.\qedhere
    \end{enumerate}
\end{proof}
	\section{Local Floer Chain Complex Mutation}\label{Section - LFCC}

   In this section, we carry out the main computations of the paper. As in the previous section, all Floer chain complexes are taken over the coefficient ring $\mathbb Z_2$. Thus all coefficients appearing in the differentials are understood modulo $2$.

Nevertheless, we will often write coefficients such as $2$, $k$, or $c(k)$ explicitly in the formulas for the differential. These coefficients record the actual number of Floer cascades obtained from \Cref{Thm - Cylinder Counting} and from the subsequent finite cascade-counting arguments. The differential itself is the $\Z_2$-linear differential obtained by reducing these counts modulo $2$. Thus, for example, a term such as $2\g$ vanishes in the chain complex, while a term such as $k\g$ contributes according to the parity of $k$.
    
    For each of the five generic bifurcations and the two $\Z_2$-symmetric involutive bifurcations, we determine the local Floer chain complex of the corresponding model Hamiltonian and describe how it mutates across the degenerate value.
    Throughout, the integer $k$ records the order of the resonance, i.e., the Floquet multiplier $\ld$ is a primitive $k$-th root of unity, so that it is the $k$-th cover of the \emph{core orbit} --- the orbit corresponds to the origin of $\Sigma$ --- that becomes degenerate and undergoes the bifurcation. The differentials follow from the cascade counting theorem of \Cref{Section - LFH} together with a finite combinatorial analysis of the Floer cascades, and the associated Morse--Bott spectral sequence computes the resulting local Floer homology, confirming the expected invariance across the bifurcation.

    To avoid pathological situations, we impose the following assumptions on the Morse functions for each orbit. Let $\g_0, \g_1, \g_2$ be periodic orbits of an autonomous Hamiltonian with relative Conley--Zehnder indices $\mu_{CZ}(\g_1,\g_0)=1$ and $\mu_{CZ}(\g_2,\g_0)=2$, and assume that there exists a Floer cylinder $u_i$ from $\g_i$ to $\g_0$ for $i=1,2$. 
    After fixing Morse functions on $\g_1$ and $\g_2$, we require the Morse function on $\g_0$ to satisfy:
    \begin{enumerate}
    \item $\gM_0\notin\ev^-(\gm_1;u_1)$ and $\gm_0\notin\ev^-(\gM_1;u_1)$.
    \item $\gM_0\notin\ev^-(\gm_2;u_2)$.
    \end{enumerate}
    The first condition can be satisfied by choosing a Morse function on $\g_0$ whose maximum and minimum avoid a finite set of points, which is a generic choice. 
    The second condition is more subtle because the moduli space of Floer cylinders from $\g_2$ to $\g_0$ is 1-dimensional. Consequently, for a 1-parameter family of cylinders $u_\sigma$, the image $\ev(\gm_2;u_\sigma)$ is generally a 1-dimensional subset of $\g_0$. 
    However, this issue does not arise in our explicit setup, as will be briefly discussed in the proof of \Cref{Theorem - FCEM}.
    
	\subsection{Birth-death Bifurcation}\label{Subsection - LFCC of BD}
	
	We consider the model of \Cref{Section - cylinder} with $k=1$; the single-cover case, in which the core orbit itself becomes degenerate and bifurcates. As discussed in \Cref{Theorem - PObd}, the resulting bifurcation of $H_\eps$ is the birth-death: there is no $1$-periodic orbit for $\eps<0$, while for $\eps>0$ there are two $1$-periodic orbits $\g_0$ and $\g_1$ with $\mu_{CZ}(\g_1,\g_0)=1$.
	We denote the elliptic orbit among $\g_{0,1}$ by $\g_E$ and the hyperbolic one by $\g_H$.
	By \Cref{Theorem - PObd} the two cases are distinguished by which orbit is elliptic.
    For the source-type the elliptic orbit is the upper orbit $\g_1$, while for the sink-type it is the lower orbit $\g_0$.
    As in \cref{Subsection - CC of LFH}, we have
    \[
    \mu_{CZ}^\fT(\gM)=\mu_{CZ}^\fT(\g)+1,\quad \mu_{CZ}^\fT(\gm)=\mu_{CZ}^\fT(\g)
    \]
    for any periodic orbit $\g$.
		\begin{thm}[Floer complex for birth-death]\label{Theorem - FCBD}
			Let $H_\eps$ be the model Hamiltonian of the birth-death bifurcation.
            Then,
			\[
					CF_*(W,H_\eps)=\left\{\begin{array}{cc}
						0&\text{ if }\eps<0,\\
						\Z_2\il\gm_E,\gM_E,\gm_H,\gM_H\ir&\text{ if }\eps>0.
						\end{array}\right.
			\]
            \begin{enumerate}
                \item (Source) The elliptic orbit is the upper orbit. Relative to $\gm_H$, the degrees are
			\[
			\mu_{CZ}(\gM_H,\gm_H)=\mu_{CZ}(\gm_E,\gm_H)=1,\quad 
                \mu_{CZ}(\gM_E,\gm_H)=2,
			\]
            and the differential is given by
            \[
            \pp(\gM_E)=\gM_H,\quad \pp(\gm_E)=\gm_H.
            \]
            \item (Sink) The elliptic orbit is the lower orbit. Relative to $\gm_E$, the degrees are
			\[
            \mu_{CZ}(\gM_E,\gm_E)=\mu_{CZ}(\gm_H,\gm_E)=1,\quad 
                \mu_{CZ}(\gM_H,\gm_E)=2,
			\]
            and the differential is given by
            \[
            \pp(\gM_H)=\gM_E,\quad \pp(\gm_H)=\gm_E.
            \]
            \end{enumerate}
			In both cases $HF_*(W,H_\eps)=0$ for $\eps\neq0$.
		\end{thm}
		\begin{proof}
			Since $\mu_{CZ}(\g_1,\g_0)=1$, the gradient revolution of \Cref{Subsection - Gradient Revolution} determines a unique cylinder between $\g_E$ and $\g_H$.
            The relative Conley--Zehnder indices follow from the Morse indices via \Cref{Proposition - Relative CZ Index}, together with the splitting $\mu_{CZ}^\fT(\gM)=\mu_{CZ}^\fT(\g)+1$, $\mu_{CZ}^\fT(\gm)=\mu_{CZ}^\fT(\g)$ above.
			The Floer cascade counting of \Cref{Thm - Cylinder Counting} then gives the differential, and the homology vanishes.
            The Morse--Bott spectral sequence computing the homology is described in \Cref{Figure - SSBD}.
		\end{proof}

      \begin{figure}[b]
  \begin{subfigure}[b]{0.45\textwidth}
    \centering
    \includegraphics[width=\textwidth]{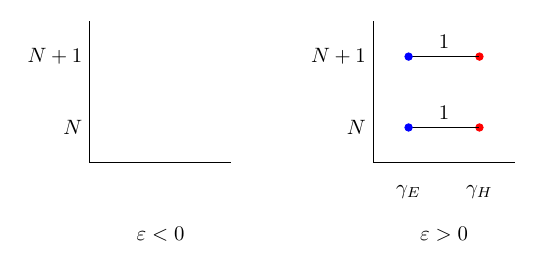}
    \caption{Source birth-death}
  \end{subfigure}
  \hspace{0.05\textwidth}
    \begin{subfigure}[b]{0.45\textwidth}
    \centering
    \includegraphics[width=\textwidth]{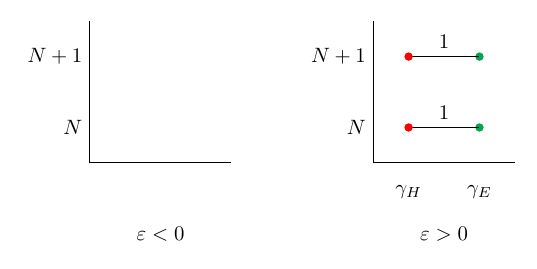}
    \caption{Sink birth-death}
  \end{subfigure}
  \caption{Morse--Bott spectral sequences for birth-death. As mentioned in \Cref{Note - Colors}, the colors of dots indicate the Morse index of corresponding critical points, which gives the relative Conley--Zehnder index of the periodic orbit.}
  \label{Figure - SSBD}
  \end{figure}
    
	\subsection{Period Doubling and Materialization}\label{Subsection - LFCC of PD}
We consider the model of \Cref{Section - cylinder} with $k=2$; the double-cover case, in which the second cover of the core orbit becomes degenerate and bifurcates.
	As discussed in \Cref{Theorem - POpd}, the resulting bifurcation of $H_\eps$ is the period doubling or the materialization.
	We follow the discussion of \Cref{Subsection - LFCC of BD}.
	To avoid ambiguity, we denote by $\g$ the core orbit, which exists before and after the bifurcation, and distinguish it by $\g_-$ for $\eps<0$ and $\g_+$ for $\eps>0$; as a double cover, it is written $\g^2$.
	The new orbit born at $\eps>0$ is denoted $\g_{new}$, which is elliptic ($\g_{new}=\g_E$) for period doubling and hyperbolic ($\g_{new}=\g_H$) for materialization.
	Independently of this, each type is attracting or repelling according to whether $\g_{new}$ lies below or above $\g^2$ in the relative grading, as in \Cref{Theorem - POpd}.
    
	\begin{thm}[Floer complex for period doubling and materialization]\label{Theorem - FCPD}
		Let $H_\eps$ be the model Hamiltonian of the period doubling or materialization bifurcation.
		Then,
			\[
			CF_*(W,H_\eps) = \begin{cases}
				\Z_2\il\gm_-^2,\gM_-^2\ir&\text{ if }\eps<0,\\
				\Z_2\il\gm_+^2,\gM_+^2,\gm_{new},\gM_{new}\ir&\text{ if }\eps>0,
			\end{cases}
			\]
        where $\g_{new}=\g_E$ is elliptic for period doubling and $\g_{new}=\g_H$ is hyperbolic for materialization.
        \begin{enumerate}
            \item (Attracting) Relative to $\gm_-^2$, the degrees are
            \[
            \begin{aligned}
            \mu_{CZ}(\gm_+^2,\gm_-^2)&=-1,\quad \mu_{CZ}(\gM_+^2,\gm_-^2)=0,\\
            \mu_{CZ}(\gm_{new},\gm_-^2)&=0,\quad \mu_{CZ}(\gM_{new},\gm_-^2)=1,
            \end{aligned}
			\]
            and the differential is
            \[
			\pp(\gM_{new})=2\gM_+^2=0,\quad \pp(\gm_{new})=\gm_+^2.
			\]
            \item (Repelling) Relative to $\gm_-^2$, the degrees are
            \[
            \begin{aligned}
            \mu_{CZ}(\gm_+^2,\gm_-^2)&=1,\quad \mu_{CZ}(\gM_+^2,\gm_-^2)=2,\\
            \mu_{CZ}(\gm_{new},\gm_-^2)&=0,\quad \mu_{CZ}(\gM_{new},\gm_-^2)=1,
            \end{aligned}
			\]
            and the differential is
            \[
			\pp(\gM_+^2)=\gM_{new},\quad \pp(\gm_+^2)=2\gm_{new}=0.
			\]
        \end{enumerate}
        In either case the homology is
            \[
            HF_*(W,H_\eps)=\begin{cases}
                \Z_2&\text{ if }*=0\text{ or }1\text{ (relative to }\gm_-^2),\\
                0&\text{ otherwise,}
                \end{cases}
            \]
        for any $\eps\neq0$; that is, $HF_*$ is supported in the two consecutive degrees $0$ and $1$ relative to $\gm^2_-$.
	\end{thm}
	\begin{proof}
		The gradient revolution of \Cref{Subsection - Gradient Revolution} determines a unique cylinder between $\g^2$ and $\g_{new}$, and the differential follows from \Cref{Thm - Cylinder Counting}.
        The computation of the homology is then straightforward, and the Morse--Bott spectral sequence computing it is described in \Cref{Figure - SSPD}.
        Writing the surviving generators in ascending order of degree, for $\eps>0$ we have
        \[
        HF_*(W,H_\eps)=\begin{cases}
            \Z_2\il \gM_+^2,\gM_{new} \ir&\text{ for attracting},\\
            \Z_2\il \gm_{new},\gm_+^2\ir&\text{ for repelling}.
        \end{cases}
        \]
	\end{proof}

      \begin{figure}[b]
  \begin{subfigure}[b]{0.45\textwidth}
    \centering
    \includegraphics[width=\textwidth]{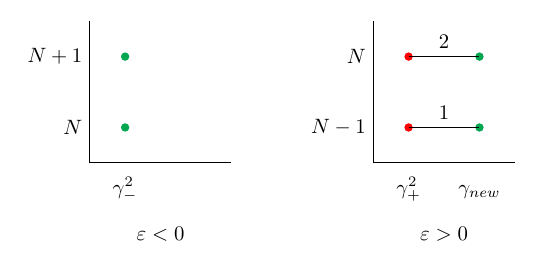}
    \caption{Attracting period doubling}
  \end{subfigure}
  \hspace{0.05\textwidth}
    \begin{subfigure}[b]{0.45\textwidth}
    \centering
    \includegraphics[width=\textwidth]{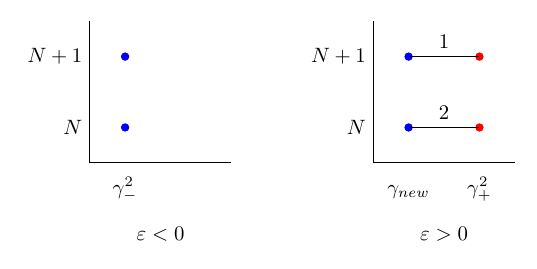}
    \caption{Repelling period doubling}
  \end{subfigure}
  \caption{Morse--Bott spectral sequences for period doubling. Materialization is the same with the opposite stability types.}
  \label{Figure - SSPD}
  \end{figure}

	\subsection{Phantom Kiss}\label{Subsection - LFCC of PK}
	We consider the model of \Cref{Section - cylinder} with $k=3$ or $4$; the case in which the $3$-rd or $4$-th cover of the core orbit becomes degenerate and bifurcates. As discussed in \Cref{Theorem - POpk}, the resulting bifurcation of $H_\eps$ is the phantom kiss, and we follow the discussion of \Cref{Subsection - LFCC of BD}. For the phantom kiss, a hyperbolic orbit is present on both sides of the bifurcation. We write $\g_{H,\pm}$ for the hyperbolic orbit and $\g^k_\pm$ denotes the $k$-th cover of the core orbit at $\eps\gtrless0$.

 	\begin{thm}[Floer complex for phantom kiss]\label{Theorem - FCPK}
		Let $H_\eps$ be the model Hamiltonian of the $k$-phantom kiss bifurcation, where $k=3$ or $4$.
		Then,
			\[
			CF_*(W_k,H_\eps) = \begin{cases}
				\Z_2\il\gm_-^k,\gM_-^k,\gm_{H,-},\gM_{H,-}\ir&\text{ if }\eps<0,\\
				\Z_2\il\gm_+^k,\gM_+^k,\gm_{H,+},\gM_{H,+}\ir&\text{ if }\eps>0.
			\end{cases}
			\]
        By an appropriate choice of the bifurcation parameter $\eps$, the degrees relative to $\gm^k_-$ are given by
            \[
            \begin{aligned}
                \eps<0:\quad&\mu_{CZ}(\gm_{H,-},\gm^k_-)=-1,\quad \mu_{CZ}(\gM_{H,-},\gm^k_-)=0,\\
                \eps>0:\quad&\mu_{CZ}(\gm^k_+,\gm^k_-)=-2,\quad \mu_{CZ}(\gM^k_+,\gm^k_-)=-1,\\
                &\mu_{CZ}(\gm_{H,+},\gm^k_-)=-1,\quad \mu_{CZ}(\gM_{H,+},\gm^k_-)=0,
            \end{aligned}
            \]
        and the differentials are given by
            \[
            \begin{aligned}
                \eps<0:&\quad \pp(\gM_-^k)=\gM_{H,-},\quad \pp(\gm_-^k)=k\,\gm_{H,-},\\
                \eps>0:&\quad \pp(\gM_{H,+})=k\,\gM^k_+,\quad \pp(\gm_{H,+})=\gm^k_+,
            \end{aligned}
            \]
        which results in
            \[
            \begin{aligned}
                k=3:&\quad HF_*(W,H_\eps)=0,\\
                k=4:&\quad HF_*(W,H_\eps)=\begin{cases}
                \Z_2&\text{ if }*=-1\text{ or }0\ (\text{relative to }\gm_-^k),\\
                0&\text{ otherwise,}
                \end{cases}
            \end{aligned}
            \]
        for any $\eps\neq0$.
	\end{thm}
    \begin{proof}
        There exists a unique Floer cylinder between $\g^k_\pm$ and $\g_{H,\pm}$, and the differential follows from \Cref{Thm - Cylinder Counting}.
        Hence $HF_*=0$ for $k=3$, while for $k=4$ the Morse--Bott spectral sequence of \Cref{Figure - SSPK} leaves
        \[
        HF_*(W,H_\eps)=\begin{cases}
            \Z_2\il \gm_{H,-},\gm^4_-\ir&\text{ if }\eps<0,\\
            \Z_2\il \gM^4_+,\gM_{H,+}\ir&\text{ if }\eps>0,
        \end{cases}
        \]
        where in each case the generators are listed in ascending order of degree.
    \end{proof}

      \begin{figure}[b]
    \centering
    \includegraphics[width=0.45\textwidth]{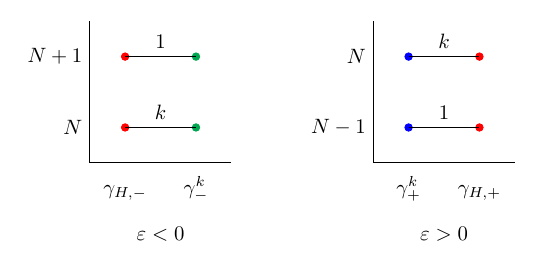}
  \caption{Morse--Bott spectral sequences for phantom kiss}
  \label{Figure - SSPK}
  \end{figure}
    
	\subsection{Emission}\label{Subsection - LFCC of EM}
			We consider the model of \Cref{Section - cylinder} with $k\geq4$; the case in which the $k$-th cover of the core orbit becomes degenerate and bifurcates. Note that in the case $k=4$, both the 4-phantom kiss and 4-emission can happen, depending on the coefficients of the normal form; see the discussion at the end of \Cref{Subsection - k bifurcation}. As discussed in \Cref{Theorem - POem}, the resulting bifurcation of $H_\eps$ is the emission, and we follow the discussion of \Cref{Subsection - LFCC of BD}. At $\eps>0$, the emission produces both a hyperbolic orbit $\g_H$ and an elliptic orbit $\g_E$, while $\g^k_\pm$ denotes the $k$-th cover of the core orbit at $\eps\gtrless0$. As in the period doubling and materialization cases, each emission is attracting or repelling.

	\begin{thm}[Floer complex for emission]\label{Theorem - FCEM}
		Let $H_\eps$ be the model Hamiltonian of the $k$-emission bifurcation, where $k\geq 4$.
		Then,
			\[
			CF_*(W,H_\eps) = \begin{cases}
				\Z_2\il\gm^k_-,\gM^k_-\ir&\text{ if }\eps<0,\\
                \Z_2\il\gm^k_+,\gM^k_+,\gm_H,\gM_H,\gm_E,\gM_E\ir&\text{ if }\eps>0.
			\end{cases}
		\]
        \begin{enumerate}
            \item (Attracting) The degrees relative to $\gm_-^k$ are
            \[
            \begin{aligned}
            \mu_{CZ}(\gm^k_+,\gm_-^k)&=-2,\quad \mu_{CZ}(\gM^k_+,\gm_-^k)=-1,\\
            \mu_{CZ}(\gm_H,\gm_-^k)&=-1,\quad \mu_{CZ}(\gM_H,\gm_-^k)=0,\\
            \mu_{CZ}(\gm_E,\gm_-^k)&=0,\quad \mu_{CZ}(\gM_E,\gm_-^k)=1,
            \end{aligned}
			\]
            and the differential is
            \[
			\begin{aligned}
		      \pp(\gM_E)&= 2\gM_H,\quad \pp(\gm_E)=2\gm_H+c(k)\gM^k_+,\\
            \pp(\gM_H)&=k\gM^k_+,\quad \pp(\gm_H)=\gm^k_+.
			\end{aligned}
			\]
            Here $c(k)$ is a constant depending on the choice of the Morse function, determined in \Cref{Lemma - emission count} by
            \[
            c(k)\equiv n\ \text{ or }\ k-n \pmod{2}.
            \]
            where $n$ is an integer such that $nl\equiv 1 \pmod{k}$, where the Floquet multiplier of $\g_0$ is given by $e^{2\pi il/k}$.
            In particular, $c(k)$ is odd whenever $k$ is even.
            \item (Repelling) The degrees relative to $\gm_-^k$ are
            \[
            \begin{aligned}
            \mu_{CZ}(\gm_E,\gm_-^k)&=0,\quad \mu_{CZ}(\gM_E,\gm_-^k)=1,\\
            \mu_{CZ}(\gm_H,\gm_-^k)&=1,\quad \mu_{CZ}(\gM_H,\gm_-^k)=2,\\
            \mu_{CZ}(\gm^k_+,\gm_-^k)&=2,\quad \mu_{CZ}(\gM^k_+,\gm_-^k)=3,
            \end{aligned}
			\]
            and the differential is
            \[
			\begin{aligned}
		      \pp(\gM^k_+)&= \gM_H,\quad \pp(\gm^k_+)=k\gm_H+c(k)\gM_E,\\
            \pp(\gM_H)&=2\gM_E,\quad \pp(\gm_H)=2\gm_E,
			\end{aligned}
			\]
            with $c(k)$ as in the attracting case.
        \end{enumerate}
        In either case the homology is
            \[
            HF_*(W,H_\eps)=\begin{cases}
                \Z_2&\text{ if }*=0\text{ or }1\ (\text{relative to }\gm_-^k),\\
                0&\text{ otherwise,}
                \end{cases}
            \]
        for any $\eps\neq0$.
	\end{thm}
    \begin{proof}
        We first treat the attracting case for $\eps>0$. There are two gradient revolutions from $\g_E$ to $\g_H$, denoted by $u_1$ and $u_2$. These yield the relation $\pp\gM_E=2\gM_H$ and contribute $2\gm_H$ to $\pp\gm_E$. Also, there is a single gradient revolution from $\g_H$ to $\g^k_+$, denoted by $u_3$, which gives $\pp\gM_H=k\gM^k_+$ and $\pp\gm_H=\gm^k_+$.

        Moreover, there exists a 1-parameter family of gradient revolutions from $\g_E$ to $\g^k_+$, say $u_\sigma$. Let $\g_E(t_0)=\gm_E$. By construction, we have 
\[
\ev^-(\gm_E;u_\sigma)=\left\{\g^k_+(t_0),\g^k_+(t_0+1/k),\ldots,\g^k_+(t_0+(k-1)/k)\right\} = \left\{\g^k_+(t_0)\right\}
\]
for any $\sigma$, which implies that the image reduces to a single point.
(Note that in this case, we parametrized $\g^k$ as 1-periodic orbit, so $\g^k(0)=\g^k(j/k)$ for any integer $j$.)
We choose a Morse function such that $\gM^k_+$ avoids this point, ensuring that the genericity assumption stated at the beginning of this section is satisfied. In particular, with this choice, there are no Floer cascades of length 1 consisting solely of Floer cylinders between autonomous orbits with an index difference of 2 from $\gm_E$ to $\gM^k_+$.
        
It remains to compute the coefficient of $\gM^k_+$ in $\pp\gm_E$, which is determined by counting Floer cascades formed by the composition of:
\begin{enumerate}
    \item the Floer cylinders $u_1$ and $u_2$ from $\g_E$ to $\g_H$,
    \item Morse trajectories on $\g_H$ from $\ev^-(\gm_E;u_1)$ and $\ev^-(\gm_E;u_2)$ to $\ev^+(\gM^k_+;u_3)$, and
    \item the Floer cylinder $u_3$ from $\g_H$ to $\g^k_+$.
\end{enumerate}
This reduces to a count of Morse trajectories on $\g_H$. Assume that $\gm_E = \g_E(0)$ and $\gM^k_+=\g^k_+(t_0)$. Then we have
        \[
        \ev^-(\gm_E;u_1) = \g_H(0),\quad        \ev^-(\gm_E;u_2) = \g_H(n/k),
        \]
        where $n$ is an integer such that $nl\equiv 1 \pmod{k}$, which must be coprime to $k$. We also have
        \[
        \ev^+(\gM^k_+;u_3) = \left\{\g_H(t_0),\g_H(t_0+ 1/k),\ldots \g_H(t_0+(k-1)/k)\right\},
        \]
        where $t_0\neq  j/k$ for any integer $j$. The resulting count is $c(k)$, as computed in the first case of \Cref{Lemma - emission count}.
        The Morse--Bott spectral sequence computing the homology is described in \Cref{Figure - SSEM}, and the surviving generators for $\eps>0$ are
        \[
        HF_*(W,H_\eps)=\begin{cases}
            \Z_2\il \gM_H,\gM_E \ir&\text{ if }k\text{ is even},\\
            \Z_2\il \gm_E,\gM_E\ir\ \text{ or }\ \Z_2\il (\gm_E+\gM_H),\gM_E\ir&\text{ if }k\text{ is odd},
        \end{cases}
        \]
        listed in ascending order of degree.

        The repelling case is analogous, with the Morse trajectories on $\g_H$ now running from $\ev(\gm^k_+)$ to $\ev^{-1}(\gM_E)$; this corresponds to the second case of \Cref{Lemma - emission count}. The surviving generators for $\eps>0$ are
        \[
        HF_*(W,H_\eps)=\begin{cases}
            \Z_2\il \gm_E,\gm_H \ir&\text{ if }k\text{ is even},\\
            \Z_2\il \gm_E,\gM_E\ir&\text{ if }k\text{ is odd},
        \end{cases}
        \]
        again in ascending order of degree.
    \end{proof}

          \begin{figure}[b]
  \begin{subfigure}[b]{0.45\textwidth}
    \centering
    \includegraphics[width=\textwidth]{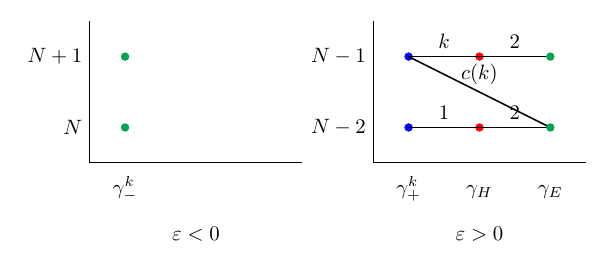}
    \caption{Attracting emission}
  \end{subfigure}
  \hspace{0.05\textwidth}
    \begin{subfigure}[b]{0.45\textwidth}
    \centering
    \includegraphics[width=\textwidth]{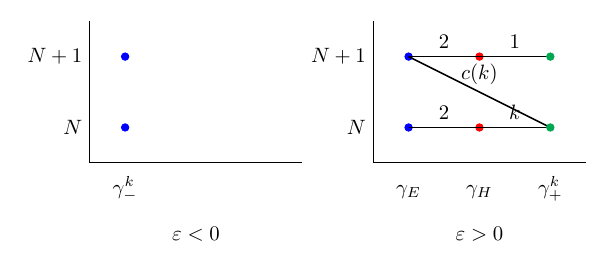}
    \caption{Repelling emission}
  \end{subfigure}
  \caption{Morse--Bott spectral sequences for emission}
  \label{Figure - SSEM}
  \end{figure}

 \begin{lmm}\label{Lemma - emission count}
        Let a circle be given by $S^1=\R/2k\Z$. Let there be $k+2$ points
            \[\begin{aligned}
                b_1&=0,\,b_2=2,\ldots,b_k=2(k-1),\\
                g_1&=1,\,g_2=2n-1
            \end{aligned}\]
        on the circle, where $(k,n)=1$. Take two points $M$ and $m$ in generic position, that is, $M,m\neq b_i,g_j$ for any $i,j$.
        We say an oriented arc on the circle is \textbf{admissible} if it is a sub-arc of one of the two arcs from $M$ to $m$.
        Define
            \[
            \begin{aligned}
            c_a(k) &= \#\{\text{admissible arc }\,a:\,a\text{ connects }g_i\text{ to }b_j\text{ for some }i,j\},\\ 
            c_r(k)&= \#\{\text{admissible arc }\,a:\,a\text{ connects }b_i\text{ to }g_j\text{ for some }i,j\}.
            \end{aligned}
            \]
        Then
            \[
            c_a(k)\equiv c_r(k)\equiv n\ \text{ or }\ k-n \pmod 2.
            \]
        In particular, both $c_a(k)$ and $c_r(k)$ are odd whenever $k$ is even.
    \end{lmm}
        \begin{proof}
        We cut the circle at the point $M$, which results in the interval $(0,2k)$.
        We denote the coordinates $t(p)$ in this interval by
            \[
            t(M)=0,\quad t(g_i)=\alpha_i,\quad t(b_i)=\beta_i,\quad t(m) = \delta,
            \]
        and write $\alpha_+=\max(\alpha_1,\alpha_2)$ and $\alpha_-=\min(\alpha_1,\alpha_2)$.
        The two admissible arcs are the intervals $(0,\delta)$ and $(\delta,2k)$.
        We first compute $c_a(k)$, dividing into cases.
        \begin{enumerate}
            \item ($\alpha_-<\alpha_+<\delta$) Only $(0,\delta)$ contributes, and
                \[
                \begin{aligned}
                    c_a(k) &= \#\{b_i\,:\,\alpha_-<\beta_i<\delta\} + \#\{b_i\,:\,\alpha_+<\beta_i<\delta\}\\
                    &=\#\{b_i\,:\,\alpha_-<\beta_i<\alpha_+\} + 2\#\{b_i\,:\,\alpha_+<\beta_i<\delta\}.
                \end{aligned}
                \]
            \item ($\delta<\alpha_-<\alpha_+$) This case is analogous:
            \[
                \begin{aligned}
                    c_a(k) &= \#\{b_i\,:\,\delta<\beta_i<\alpha_-\} + \#\{b_i\,:\,\delta<\beta_i<\alpha_+\}\\
                    &=2\#\{b_i\,:\,\delta<\beta_i<\alpha_-\} + \#\{b_i\,:\,\alpha_-<\beta_i<\alpha_+\}.
                \end{aligned}
                \]
            \item ($\alpha_-<\delta<\alpha_+$) In this case,
            \[
                \begin{aligned}
                    c_a(k) &= \#\{b_i\,:\,\alpha_-<\beta_i<\delta\} + \#\{b_i\,:\,\delta<\beta_i<\alpha_+\}\\
                    &=\#\{b_i\,:\,\alpha_-<\beta_i<\alpha_+\}.
                \end{aligned}
                \]
        \end{enumerate}
        These cases are illustrated in \Cref{Figure - Emission count}.
        In every case, we have
            \[
            c_a(k) \equiv \#\{b_i\,:\,\alpha_-<\beta_i<\alpha_+\}\pmod 2.
            \]
        The arc $(\alpha_-,\alpha_+)$ contains $n$ or $k-n$ of the points $b_i$, depending on the position of $m$, which gives the claim for $c_a(k)$.

        The same argument applies to $c_r(k)$. With the same notation, we again divide into cases.
        \begin{enumerate}
            \item ($\alpha_-<\alpha_+<\delta$) Only $(0,\delta)$ contributes:
            \[
            \begin{aligned}
                c_r(k)&=\#\{b_i\,:\,\beta_i<\alpha_-\}+\#\{b_i\,:\,\beta_i<\alpha_+\}\\
                &=2\#\{b_i\,:\,\beta_i<\alpha_-\}+\#\{b_i\,:\,\alpha_-<\beta_i<\alpha_+\}.
            \end{aligned}            
            \]
            \item ($\delta<\alpha_-<\alpha_+$) Here $(\delta,2k)$ contributes:
            \[
            \begin{aligned}
                c_r(k)&=\#\{b_i\,:\,\alpha_-<\beta_i\}+\#\{b_i\,:\,\alpha_+<\beta_i\}\\
                &=\#\{b_i\,:\,\alpha_-<\beta_i<\alpha_+\}+2\#\{b_i\,:\,\alpha_+<\beta_i\}.
            \end{aligned}            
            \]
            \item ($\alpha_-<\delta<\alpha_+$) In this case,
            \[
            \begin{aligned}
                c_r(k)&=\#\{b_i\,:\,\beta_i<\alpha_-\}+\#\{b_i\,:\,\alpha_+<\beta_i\}\\
                &=k - \#\{b_i\,:\,\alpha_-<\beta_i<\alpha_+\}.
            \end{aligned}
            \]
        \end{enumerate}
        Again $\#\{b_i\,:\,\alpha_-<\beta_i<\alpha_+\}$ determines the parity, which is $n$ or $k-n$ modulo $2$, completing the proof.
    \end{proof}

\begin{figure}[htbp]
  \begin{subfigure}[b]{0.3\textwidth}
    \centering
    \includegraphics[width=\textwidth]{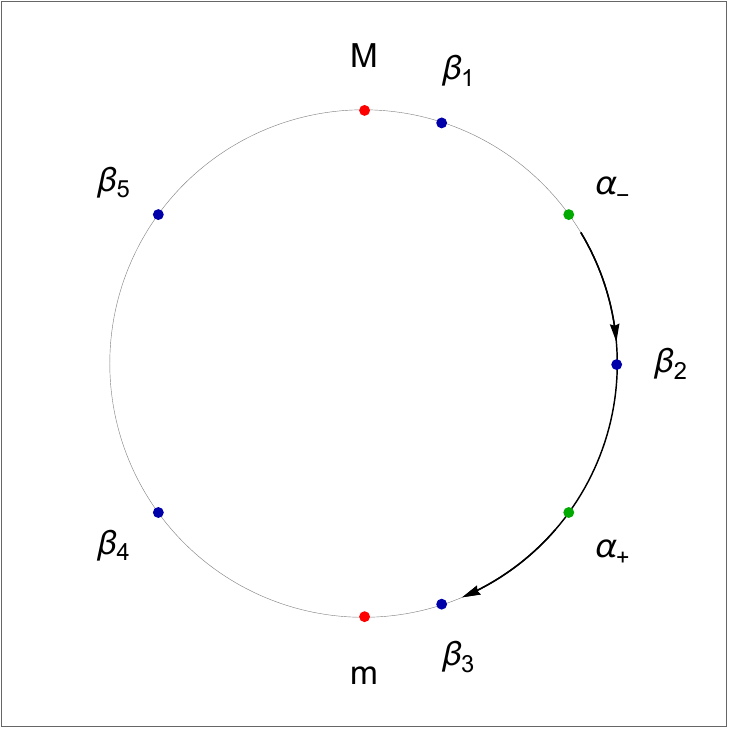}
    \caption{Case 1}
  \end{subfigure}
  \hspace{0.03\textwidth}
  \begin{subfigure}[b]{0.3\textwidth}
    \centering
    \includegraphics[width=\textwidth]{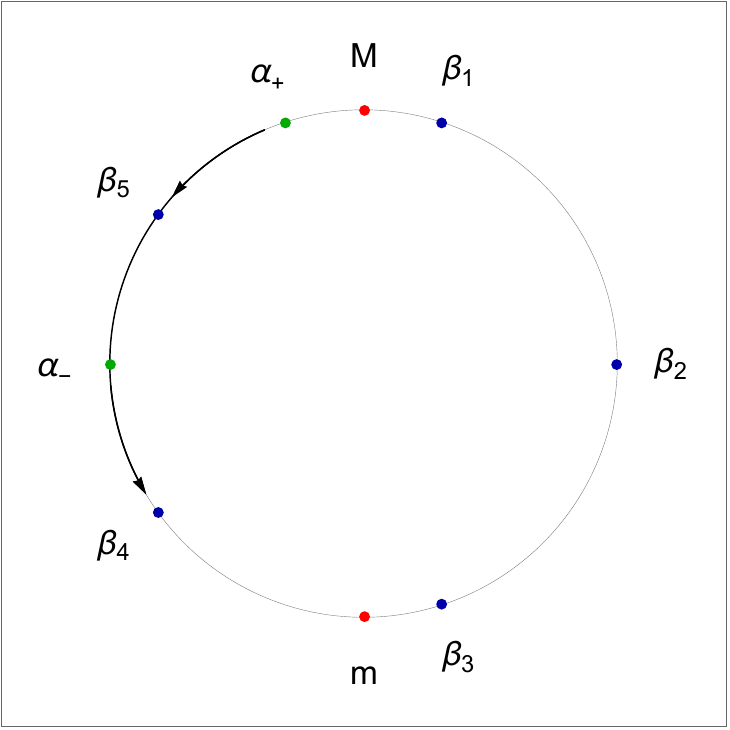}
    \caption{Case 2}
  \end{subfigure}
    \hspace{0.03\textwidth}
  \begin{subfigure}[b]{0.3\textwidth}
    \centering
    \includegraphics[width=\textwidth]{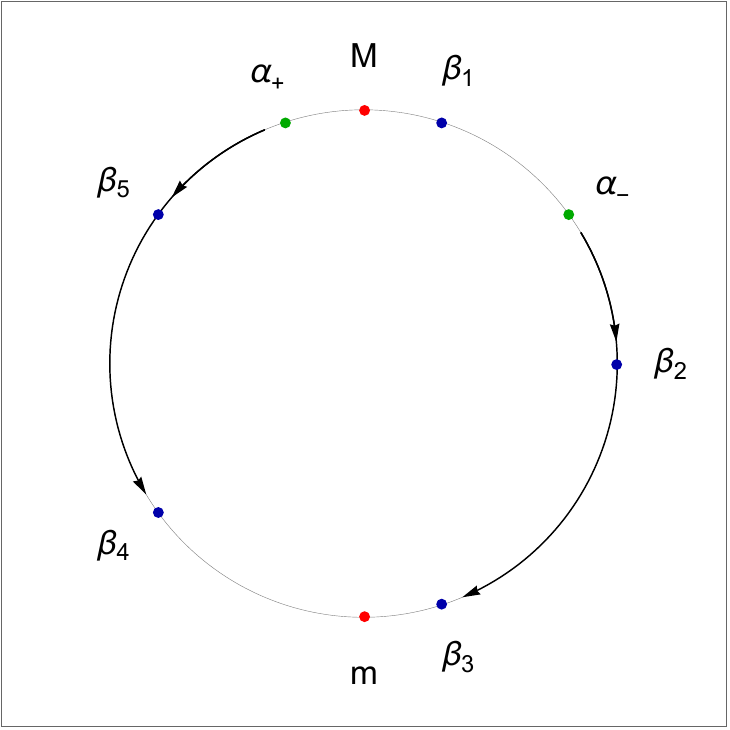}
    \caption{Case 3}
  \end{subfigure}
  \caption{Three cases in \Cref{Lemma - emission count}, for the attracting $5$-emission with $l=1$.}
    \label{Figure - Emission count}
  \end{figure}

\subsection{Pitchfork}\label{Subsection - LFCC of PF}

We consider the model of \Cref{Section - cylinder} with $k=1$ in the presence of a $\mathbb{Z}_2$-symmetry; the single-cover case, in which the core orbit itself becomes degenerate and bifurcates symmetrically. As discussed in \Cref{Subsection - Pitchfork}, the resulting bifurcation of $H_\eps$ is the pitchfork, and we follow the discussion of \Cref{Subsection - LFCC of BD}. The core orbit splits into three: the original core orbit $\g$, which we distinguish by $\g_\pm$ according to the sign of $\eps$, together with two new orbits $\g_1$ and $\g_2$ born at $\eps>0$. By analyzing the relative Conley--Zehnder indices and counting the rigid gradient revolutions connecting these orbits, we obtain the following Floer complex.

\begin{thm}[Floer complex for Pitchfork]\label{Theorem - FCPF}
    Let $H_\eps$ be the model Hamiltonian of the pitchfork bifurcation.
    Then,
    \[
            CF_*(W,H_\eps)=\left\{\begin{array}{cc}
                \Z_2\il \gm_-,\gM_- \ir&\text{ if }\eps<0,\\
                \Z_2\il \gm_+,\gM_+,\gm_1,\gM_1,\gm_2,\gM_2\ir&\text{ if }\eps>0.
                \end{array}\right.
    \]
    \begin{enumerate}
        \item (Attracting) Relative to $\gm_-$, the degrees are
        \[
        \begin{aligned}
        \mu_{CZ}(\gm_+,\gm_-)&=-1,\quad \mu_{CZ}(\gM_+,\gm_-)=0,\\
        \mu_{CZ}(\gm_1,\gm_-)&=\mu_{CZ}(\gm_2,\gm_-)=0,\\
        \mu_{CZ}(\gM_1,\gm_-)&=\mu_{CZ}(\gM_2,\gm_-)=1,
        \end{aligned}
        \]
        and the differential is given by
        \[
        \pp(\gm_1)=\pp(\gm_2)=\gm_+,\quad \pp(\gM_1)=\pp(\gM_2)=\gM_+.
        \]
        
        \item (Repelling) Relative to $\gm_-$, the degrees are
        \[
        \begin{aligned}
        \mu_{CZ}(\gm_+,\gm_-)&=1,\quad \mu_{CZ}(\gM_+,\gm_-)=2,\\
        \mu_{CZ}(\gm_1,\gm_-)&=\mu_{CZ}(\gm_2,\gm_-)=0,\\
        \mu_{CZ}(\gM_1,\gm_-)&=\mu_{CZ}(\gM_2,\gm_-)=1,
        \end{aligned}
        \]
        and the differential is given by
        \[
        \pp(\gm_+)=\gm_1+\gm_2,\quad \pp(\gM_+)=\gM_1+\gM_2.
        \]
    \end{enumerate}
    In either case, the homology is
    \[
    HF_*(W,H_\eps)=\begin{cases}
        \Z_2&\text{ if }*=0\text{ or }1\text{ (relative to }\gm_-),\\
        0&\text{ otherwise.}
        \end{cases}
    \]
\end{thm}
\begin{proof}
    The structures of the Morse trajectories and Floer cylinders are the same as in the period doubling case; there is one gradient revolution connecting $\g_+$ and $\g_1$, and similarly for $\g_+$ and $\g_2$.
    The differential structure can be directly derived.
    The computation of the homology is then straightforward, and the Morse--Bott spectral sequence computing it is described in \Cref{Figure - SSPF}.
    Writing the surviving generators in ascending order of degree, for $\eps>0$ we have
    \[
    HF_*(W,H_\eps)=\begin{cases}
        \Z_2\il (\gm_1+\gm_2),(\gM_1+\gM_2) \ir&\text{ for attracting},\\
        \Z_2\il \gm_1,\gM_1\ir = \Z_2\il \gm_2,\gM_2\ir&\text{ for repelling},
    \end{cases}
    \]
    since $[\gm_1] = [\gm_1]+[\gm_1+\gm_2] = [\gm_2]$ in the homology.
\end{proof}

\begin{figure}[htbp]
    \begin{subfigure}[b]{0.45\textwidth}
    \centering
    \includegraphics[width=\textwidth]{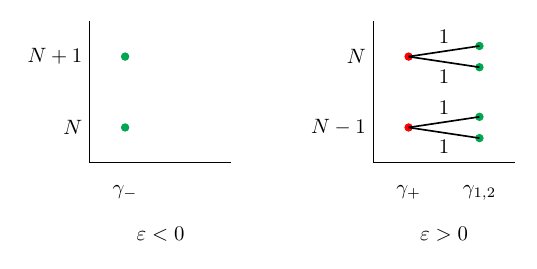}
    \caption{Attracting pitchfork}
    \end{subfigure}
    \hspace{0.05\textwidth}
    \begin{subfigure}[b]{0.45\textwidth}
    \centering
    \includegraphics[width=\textwidth]{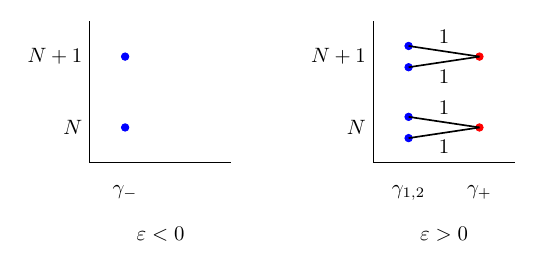}
    \caption{Repelling pitchfork}
    \end{subfigure}
    \caption{Morse--Bott spectral sequences for the subcritical pitchfork. The supercritical case is the same with the opposite stability types.}
    \label{Figure - SSPF}
\end{figure}

\subsection{Double Emission}\label{Subsection - LFCC of DE}

Finally, we consider the model of \Cref{Section - cylinder} with odd $k\geq3$ in the presence of a $\mathbb{Z}_2$-symmetry; the case in which the $k$-th cover of the core orbit becomes degenerate and bifurcates symmetrically. As discussed in \Cref{Subsection - Double Emission}, the resulting bifurcation of $H_\eps$ is the double emission, and we follow the discussion of \Cref{Subsection - LFCC of BD}. The symmetry of the unwrapped generating Hamiltonian yields a core orbit $\g^k_\pm$, accompanied by two hyperbolic orbits $\g_{H,1}$, $\g_{H,2}$, and two elliptic periodic orbits $\g_{E,1}$, $\g_{E,2}$. The differential of this complex is determined not only by the direct gradient revolutions but also by the Floer cascades connecting these branches. The resulting chain complex is described in the following theorem.

\begin{thm}[Floer complex for Double Emission]\label{Theorem - FCDE}			
    Let $H_\eps$ be the model Hamiltonian of the double $k$-emission bifurcation for odd $k\geq 3$.
    Then,
    \[
            CF_*(W,H_\eps)=\left\{\begin{array}{cc}
                \Z_2\il \gm^k_-,\gM^k_- \ir&\text{ if }\eps<0,\\
                \Z_2\il \gm^k_+,\gM^k_+,\gm_{H,1},\gM_{H,1},\gm_{H,2},\gM_{H,2},\gm_{E,1},\gM_{E,1},\gm_{E,2},\gM_{E,2} \ir&\text{ if }\eps>0.
                \end{array}\right.
    \]
    \begin{enumerate}
        \item (Attracting) Relative to $\gm_-^k$, the degrees are
    \[
    \begin{aligned}
    \mu_{CZ}(\gm^k_+,\gm_-^k)&=-2,\quad \mu_{CZ}(\gM^k_+,\gm_-^k)=-1,\\
    \mu_{CZ}(\gm_{H,i},\gm_-^k)&=-1,\quad \mu_{CZ}(\gM_{H,i},\gm_-^k)=0,\\
    \mu_{CZ}(\gm_{E,i},\gm_-^k)&=0,\quad \mu_{CZ}(\gM_{E,i},\gm_-^k)=1,
    \end{aligned}
    \]
    and the differential is given by
    \[
    \begin{aligned}
        \pp(\gM_{E,i})&= \gM_{H,1}+\gM_{H,2},\\
        \pp(\gm_{E,i})&=\gm_{H,1}+\gm_{H,2}+d_i\gM^k_+,\\
    \pp(\gM_{H,i})&=k\gM^k_+,\quad \pp(\gm_{H,i})=\gm^k_+,
    \end{aligned}
    \]
    where $d_i$ is a constant depending on the choice of Morse functions.
    
    \item (Repelling) The degrees relative to $\gm_-^k$ are
    \[
    \begin{aligned}
    \mu_{CZ}(\gm_{E,i},\gm_-^k)&=0,\quad \mu_{CZ}(\gM_{E,i},\gm_-^k)=1,\\
    \mu_{CZ}(\gm_{H,i},\gm_-^k)&=1,\quad \mu_{CZ}(\gM_{H,i},\gm_-^k)=2,\\
    \mu_{CZ}(\gm^k_+,\gm_-^k)&=2,\quad \mu_{CZ}(\gM^k_+,\gm_-^k)=3,
    \end{aligned}
    \]
    and the differential is
    \[
    \begin{aligned}
        \pp(\gM^k_+)&= \gM_{H,1}+\gM_{H,2}\\
        \pp(\gm^k_+)&=k(\gm_{H,1}+\gm_{H,2})+d_1\gM_{E,1}+d_2 \gM_{E,2},\\
    \pp(\gM_{H,i})&=\gM_{E,1}+\gM_{E,2},\quad \pp(\gm_{H,i})=\gm_{E,1}+\gm_{E,2}.
    \end{aligned}
    \]
    \end{enumerate}
    In either case, the homology is
    \[
    HF_*(W,H_\eps)=\begin{cases}
        \Z_2&\text{ if }*=0\text{ or }1\text{ (relative to }\gm_-^k),\\
        0&\text{ otherwise.}
        \end{cases}
    \]
\end{thm}
\begin{proof}
    We first consider the attracting case.
    There exist:
    \begin{enumerate}
        \item one gradient revolution from $\g_{E,i}$ to $\g_{H,j}$ for each $i,j=1,2$, say $u_{ij}$, and
        \item one gradient revolution from $\g_{H,j}$ to $\g_+^k$ for each $j=1,2$, say $v_j$,
        \item 1-parameter family of gradient revolutions from $\g_{E,i}$ to $\g_+^k$, which can be ignored under a choice of a generic Morse function as in the proof of \Cref{Theorem - FCEM}.
    \end{enumerate}
    This determines all differentials, except for the coefficients of $\gM_+^k$ in $\pp(\gm_{E,i})$, denoted by $d_i$.
    As in \Cref{Theorem - FCEM}, computing $d_i$ reduces to the sum over $j$ of the count of Morse trajectories from $\ev^-(\gm_{E,i};u_{ij})$ to $\ev^+(\gM_+^k;v_j)$ on $\g_{H,j}$, which depends on the choice of the Morse function on each periodic orbit.
    In contrast with the generic emission (\Cref{Theorem - FCEM}), here the two hyperbolic orbits carry independent Morse data. No parity constraint on $d_1+d_2$ arises: this is the same ambiguity as that of the parity of $c(k)$ for odd $k$ in \Cref{Lemma - emission count}, and since the double emission occurs only for odd $k$, it is also inevitable here.
    
    The homology is independent of these choices, but the generators depend on the parity of $d_1+d_2$.
    We have $\pp(\gm_{E,1}+\gm_{E,2})=0$ if $d_1+d_2\equiv0\pmod{2}$, and $\pp(\gm_{E,1}+\gm_{E,2}+\gM_{H,1})=0$ if $d_1+d_2 \equiv 1 \pmod{2}$. Thus,
    \[
        HF_*(W,H_\eps) =\begin{cases}
            \Z_2\il (\gm_{E,1}+\gm_{E,2}),(\gM_{E,1}+\gM_{E,2}) \ir&\text{ if }d_1+d_2\equiv 0\pmod{2},\\
        \Z_2\il (\gm_{E,1}+\gm_{E,2}+\gM_{H,i}),(\gM_{E,1}+\gM_{E,2}) \ir&\text{ if }d_1+d_2\equiv 1\pmod{2}.
        \end{cases}  \]
    
    The argument for the repelling case is analogous. In this case, $\pp(\gm_+^k)\neq 0$ for any values of $d_1,d_2$, so the resulting generators do not depend on their parities.
    The resulting homology can be represented by
    \[
    \begin{aligned}
        HF_*(W,H_\eps) &= \Z_2\il \gm_{E,1}, \gM_{E,1} \ir,
    \end{aligned}        
    \]
    where $[\gm_{E,1}]=[\gm_{E,2}]$, $[\gM_{E,1}]=[\gM_{E,2}]$ and $[\gm_{H,1}+\gm_{H,2}]=(d_1+d_2)[\gM_{E,1}]$ in homology.
    The results are described in \Cref{Figure - SSDE} in terms of Morse--Bott spectral sequences.
\end{proof}

\begin{figure}[htbp]
    \begin{subfigure}[b]{0.45\textwidth}
    \centering
    \includegraphics[width=\textwidth]{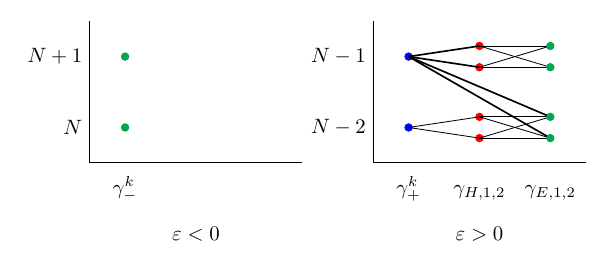}
    \caption{Attracting double emission}
    \end{subfigure}
    \hspace{0.05\textwidth}
    \begin{subfigure}[b]{0.45\textwidth}
    \centering
    \includegraphics[width=\textwidth]{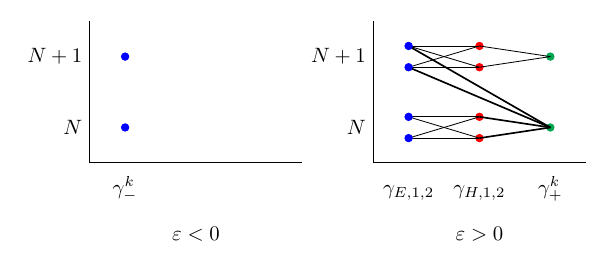}
    \caption{Repelling double emission}
    \end{subfigure}
    \caption{Morse--Bott spectral sequences for the double emission bifurcation. Thin lines represent $1$, while thick lines have nontrivial coefficients as described in \Cref{Theorem - FCDE}.}
    \label{Figure - SSDE}
\end{figure}

\section{Concluding Remarks and Future Directions}

Based on our explicit model Hamiltonians, we have explicitly computed the mutations of the local Floer chain complexes across all five generic Hamiltonian bifurcations, as well as two additional bifurcations arising in $\mathbb{Z}_2$-symmetric involutive Hamiltonian systems.

Let $K_\eps$ be a general 1-parameter family of Hamiltonians undergoing a bifurcation at $\g_0\in \mathcal{P}(K_0)$. As explained in \Cref{Subsection - Model}, after restricting to a sufficiently small isolating neighborhood of $\g_0$, the local bifurcation data of $K_\eps$ can be represented by a model Hamiltonian $H_\eps$ which is sufficiently $C^2$-close to $K_\eps$. Therefore, by the invariance discussed in
\Cref{Subsection - MBCC}, we obtain an isomorphism
\begin{equation*}
    HF_*^\loc(\g_0,K_\eps)\cong HF_*^\loc(\g_0,H_\eps).
\end{equation*}
Thus the local Floer homology obtained in \Cref{Section - LFCC} applies to arbitrary generic Hamiltonian bifurcations of the corresponding type. At the chain level, however, one should distinguish the model complexes from the complexes associated with an arbitrary Hamiltonian family. The main difficulty in describing $CF_*^\loc(\g_0,K_\eps)$ directly is that the bifurcated orbits of $K_\eps$ need not lie on a common energy hypersurface, so the Floer cylinders between them are no longer represented by gradient revolutions. In future work, we plan to develop a local Rabinowitz Floer theoretic framework to describe these chain complexes more intrinsically for arbitrary Hamiltonian families. We expect that the resulting local Rabinowitz Floer homology is isomorphic to the local Floer homology considered here, and that this perspective will provide a more flexible way to compute local Floer homology or related invariants near bifurcations.

The framework developed here opens up several promising avenues for future research on Hamiltonian bifurcations from a Floer-theoretic viewpoint. First, as noted in \Cref{Remark - more symmetry}, this machinery extends naturally to four-dimensional systems with general discrete symmetries, such as $\Z_m$-symmetry. By imposing these higher-order symmetries on the unwrapped generating Hamiltonians constructed in \Cref{Section - Classification of bifurcations}, one can directly derive the structure of the corresponding local Floer chain complexes.

Second, our theory is expected to generalize to higher-dimensional systems. As Meyer observed \cite{Meyer_70}, a generic one-parameter bifurcation typically involves only a single pair of Floquet multipliers reaching a root of unity, so that the essential dynamics, once restricted to the degenerate directions, is analogous to the four-dimensional cases studied here. A more substantial direction is that of multi-parameter bifurcations, where a richer variety of phenomena can occur, such as Krein collisions. Pursuing this will require a correspondingly more delicate treatment of the return maps and their generating Hamiltonians.

Finally, our framework can be used to investigate symmetry-breaking phenomena, such as the imperfect pitchfork bifurcation or the unfolding of the double emission. Such transitions are of significant dynamical interest and have been observed, for instance, in the triple cover of the vertical collision orbit during the passage from the Hill's lunar problem to the circular restricted three-body problem \cite{Joung_Koh_vanKoert_26}.

	\bibliographystyle{amsalpha}
	\bibliography{globalbib.bib}

@book {Golubitsky_Stewart_Schaeffer_88,
    AUTHOR = {Golubitsky, Martin and Stewart, Ian and Schaeffer, David G.},
     TITLE = {Singularities and groups in bifurcation theory. {V}ol. {II}},
    SERIES = {Applied Mathematical Sciences},
    VOLUME = {69},
 PUBLISHER = {Springer-Verlag, New York},
      YEAR = {1988},
     PAGES = {xvi+533},
      ISBN = {0-387-96652-8},
   MRCLASS = {58E07 (34C99 35B32 58C27)},
  MRNUMBER = {950168},
MRREVIEWER = {Norman\ Dancer},
       DOI = {10.1007/978-1-4612-4574-2},
       URL = {https://doi.org/10.1007/978-1-4612-4574-2},
}

@misc{Aydin_Batkhin_25,
      title={Studying network of symmetric periodic orbit families of the Hill problem via symplectic invariants}, 
      author={Cengiz Aydin and Alexander Batkhin},
      year={2025},
      eprint={2410.21245},
      archivePrefix={arXiv},
      primaryClass={math.DS},
      url={https://arxiv.org/abs/2410.21245}, 
}

@article{Joung_Koh_vanKoert_26,
  author = {Joung, Chankyu and Koh, Dayung and van Koert, Otto},
  title = {Bifurcations of highly inclined near halo orbits using Moser regularization},
  journal = {Celestial Mechanics and Dynamical Astronomy},
  year = {2026},
  volume = {138},
  number = {4},
  pages = {31},
  doi = {10.1007/s10569-026-10306-1},
  url = {https://doi.org/10.1007/s10569-026-10306-1},
  isbn = {1572-9478}
}

@book {Arnold_89,
	AUTHOR = {Arnol\`d, V. I.},
	TITLE = {Mathematical methods of classical mechanics},
	SERIES = {Graduate Texts in Mathematics},
	VOLUME = {60},
	NOTE = {Translated from the 1974 Russian original by K. Vogtmann and
	A. Weinstein,
	Corrected reprint of the second (1989) edition},
	PUBLISHER = {Springer-Verlag, New York},
	YEAR = {[1989?]},
	PAGES = {xvi+516},
	ISBN = {0-387-96890-3},
	MRCLASS = {70-02 (58F05 58Fxx 70Hxx)},
	MRNUMBER = {1345386},
}

@article {Floer_89,
	AUTHOR = {Floer, Andreas},
	TITLE = {Symplectic fixed points and holomorphic spheres},
	JOURNAL = {Comm. Math. Phys.},
	FJOURNAL = {Communications in Mathematical Physics},
	VOLUME = {120},
	YEAR = {1989},
	NUMBER = {4},
	PAGES = {575--611},
	ISSN = {0010-3616,1432-0916},
	MRCLASS = {58F05 (70H05)},
	MRNUMBER = {987770},
	MRREVIEWER = {Yong-Geun\ Oh},
	URL = {http://projecteuclid.org/euclid.cmp/1104177909},
}

@book{Meyer_Hall_Offin_13,
	title={Introduction to Hamiltonian Dynamical Systems and the N-Body Problem},
	author={Meyer, K. and Hall, G. and Offin, D.},
	isbn={9781475740738},
	lccn={91021064},
	series={Applied Mathematical Sciences},
	url={https://books.google.co.kr/books?id=KITdBwAAQBAJ},
	year={2013},
	publisher={Springer New York}
}

@article {Robbin_Salamon_93,
	AUTHOR = {Robbin, Joel and Salamon, Dietmar},
	TITLE = {The {M}aslov index for paths},
	JOURNAL = {Topology},
	FJOURNAL = {Topology. An International Journal of Mathematics},
	VOLUME = {32},
	YEAR = {1993},
	NUMBER = {4},
	PAGES = {827--844},
	ISSN = {0040-9383},
	MRCLASS = {58F05 (58E05)},
	MRNUMBER = {1241874},
	MRREVIEWER = {J.\ S.\ Joel},
	DOI = {10.1016/0040-9383(93)90052-W},
	URL = {https://doi.org/10.1016/0040-9383(93)90052-W},
}

@book {Audin_Damien_14,
	AUTHOR = {Audin, Mich\`ele and Damian, Mihai},
	TITLE = {Morse theory and {F}loer homology},
	SERIES = {Universitext},
	NOTE = {Translated from the 2010 French original by Reinie Ern\'e},
	PUBLISHER = {Springer, London; EDP Sciences, Les Ulis},
	YEAR = {2014},
	PAGES = {xiv+596},
	ISBN = {978-1-4471-5495-2; 978-1-4471-5496-9; 978-2-7598-0704-8},
	MRCLASS = {53-02 (53D40 58E05)},
	MRNUMBER = {3155456},
	MRREVIEWER = {Sonja\ Hohloch},
	DOI = {10.1007/978-1-4471-5496-9},
	URL = {https://doi.org/10.1007/978-1-4471-5496-9},
}

@article{Bourgeois_Oancea_09,
    author = {Fr{\'e}d{\'e}ric Bourgeois and Alexandru Oancea},
    title = {{Symplectic homology, autonomous Hamiltonians, and Morse-Bott moduli spaces}},
    volume = {146},
    journal = {Duke Mathematical Journal},
    number = {1},
    publisher = {Duke University Press},
    pages = {71 -- 174},
    year = {2009},
    doi = {10.1215/00127094-2008-062},
    URL = {https://doi.org/10.1215/00127094-2008-062}
}

@misc{Aydin_Frauenfelder_vanKoert_Koh_Moreno_24,
	author = {Aydin, Cengiz and Frauenfelder, Urs and van Koert, Otto and Koh, Dayoung and Moreno, Agustin},
	title = {Symplectic geometry and space mission design},
	eprint = {2308.03391},
	archivePrefix = {arXiv},
	primaryClass = {math.SG},
	year = {2024},
}

@article {Seidel_97,
    AUTHOR = {Seidel, P.},
     TITLE = {{$\pi_1$} of symplectic automorphism groups and invertibles in
              quantum homology rings},
   JOURNAL = {Geom. Funct. Anal.},
  FJOURNAL = {Geometric and Functional Analysis},
    VOLUME = {7},
      YEAR = {1997},
    NUMBER = {6},
     PAGES = {1046--1095},
      ISSN = {1016-443X,1420-8970},
   MRCLASS = {57R57 (58D99 58F05)},
  MRNUMBER = {1487754},
MRREVIEWER = {Bernd\ Siebert},
       DOI = {10.1007/s000390050037},
       URL = {https://doi.org/10.1007/s000390050037},
}

@article {Ginzburg_10,
    AUTHOR = {Ginzburg, Viktor L.},
     TITLE = {The {C}onley conjecture},
   JOURNAL = {Ann. of Math. (2)},
  FJOURNAL = {Annals of Mathematics. Second Series},
    VOLUME = {172},
      YEAR = {2010},
    NUMBER = {2},
     PAGES = {1127--1180},
      ISSN = {0003-486X,1939-8980},
   MRCLASS = {53D40 (37J05 53D35)},
  MRNUMBER = {2680488},
MRREVIEWER = {Hai-Long\ Her},
       DOI = {10.4007/annals.2010.172.1129},
       URL = {https://doi.org/10.4007/annals.2010.172.1129},
}

@article {Usher_Zhang_16,
    AUTHOR = {Usher, Michael and Zhang, Jun},
     TITLE = {Persistent homology and {F}loer-{N}ovikov theory},
   JOURNAL = {Geom. Topol.},
  FJOURNAL = {Geometry \& Topology},
    VOLUME = {20},
      YEAR = {2016},
    NUMBER = {6},
     PAGES = {3333--3430},
      ISSN = {1465-3060,1364-0380},
   MRCLASS = {53D40 (55U15)},
  MRNUMBER = {3590354},
MRREVIEWER = {Sonja\ Hohloch},
       DOI = {10.2140/gt.2016.20.3333},
       URL = {https://doi.org/10.2140/gt.2016.20.3333},
}

@article{Meyer_70,
	ISSN = {00029947},
	URL = {http://www.jstor.org/stable/1995662},
	author = {K. R. Meyer},
	journal = {Transactions of the American Mathematical Society},
	number = {1},
	pages = {95--107},
	publisher = {American Mathematical Society},
	title = {Generic Bifurcation of Periodic Points},
	urldate = {2025-11-24},
	volume = {149},
	year = {1970}
}

@book {Abraham_Marsden_78,
	AUTHOR = {Abraham, Ralph and Marsden, Jerrold E.},
	TITLE = {Foundations of mechanics},
	EDITION = {Second},
	NOTE = {With the assistance of Tudor Ra\c tiu and Richard Cushman},
	PUBLISHER = {Benjamin/Cummings Publishing Co., Inc., Advanced Book Program,
	Reading, MA},
	YEAR = {1978},
	PAGES = {xxii+m-xvi+806},
	ISBN = {0-8053-0102-X},
	MRCLASS = {58Fxx (70-02 70F07)},
	MRNUMBER = {515141},
	MRREVIEWER = {D.\ L.\ Rod},
}

@article{Golubitsky_Stewart_Marsden_87,
	title = {Generic bifurcation of Hamiltonian systems with symmetry},
	journal = {Physica D: Nonlinear Phenomena},
	volume = {24},
	number = {1},
	pages = {391-405},
	year = {1987},
	issn = {0167-2789},
	doi = {https://doi.org/10.1016/0167-2789(87)90087-X},
	url = {https://www.sciencedirect.com/science/article/pii/016727898790087X},
	author = {Martin Golubitsky and Ian Stewart and Jerrold Marsden}
}

@article{Ginzburg_Gurel_10,
author = {Viktor L. Ginzburg and Başak Z. G{\"u}rel},
title = {{Local Floer homology and the action gap}},
volume = {8},
journal = {Journal of Symplectic Geometry},
number = {3},
publisher = {International Press of Boston},
pages = {323 -- 357},
year = {2010},
}

@article{Diogo_Lisi_19,
  title={Morse--Bott split symplectic homology},
  author={Diogo, Lu{\'\i}s and Lisi, Samuel T},
  journal={Journal of Fixed Point Theory and Applications},
  volume={21},
  number={3},
  pages={77},
  year={2019},
  publisher={Springer}
}

@article{Wendl_10b,
  title={Automatic transversality and orbifolds of punctured holomorphic curves in dimension four},
  author={Wendl, Chris},
  journal={Commentarii Mathematici Helvetici},
  volume={85},
  number={2},
  pages={347--407},
  year={2010}
}
	
\end{document}